\documentclass[11pt,oneside, reqno, twoside]{amsart}
\pdfoutput=1 
\usepackage{accents}
\usepackage[portrait, margin=1in]{geometry}
\usepackage[OT2,T1]{fontenc}
\usepackage{amsmath}
\usepackage{amsthm}
\usepackage{amssymb}
\usepackage{mathtools}
\usepackage{mathrsfs}
\usepackage[shortlabels]{enumitem}
\usepackage{fancyhdr}
\newcommand\e{\varepsilon}

\def\GL{\text{GL}}

\def\cc{\mathbb{C}}
\def\hh{\mathbb{H}}

\def\rr{\mathbb{R}}
\def\R{\mathbb{R}}

\def\zz{\mathbb{Z}}

\def\fa{\mathfrak{a}}
\def\mcA{\mathcal A}

\def\mcB{\mathcal B}

\def\mcC{\mathcal C}
\def\mcD{\mathcal{D}}

\def\mcE{\mathcal{E}}
\def\mcF{\mathcal{F}}

\def\ff{\mathfrak{f}}
\def\fg{\mathfrak{g}}
\def\fh{\mathfrak{h}}
\def\mcH{\mathcal{H}}

\def\mcL{\mathcal L}
\def\mcM{\mathcal M}

\def\mfm{\mathfrak{m}}
\def\fn{\mathfrak{n}}
\def\mcN{\mathcal{N}}
\def\fo{\mathfrak{o}}

\def\fp{\mathfrak{p}}
\def\fq{\mathfrak{q}}
\def\mcS{\mathcal{S}}
\def\mcT{\mathcal T}
\def\mcW{\mathcal{W}}
\newcommand{\fc}{\mathfrak{c}}

\newcommand\nonsense[1]{
\ifthenelse{
\equal{#1}{0}
}
{\frac{1}{N(Q)}\sum_q \Psi\left(\frac qQ\right)}
{\frac{1}{N(Q)}\sum_q \Psi\left(\frac qQ\right) \frac{1}{(\log q)^{#1}}}
}

\newcommand\nonsenseqe[1]{
\ifthenelse{
\equal{#1}{0}
}
{\frac{Q^\e}{N(Q)}\sum_q \Psi\left(\frac qQ\right)}
{\frac{Q^\e}{N(Q)}\sum_q \Psi\left(\frac qQ\right) \frac{1}{(\log q)^{#1}}}
}

\newcommand\sumprimesdistinct[2]{
\sum_{\substack{p_{#1},\ldots, p_{#2} \\ p_i \neq p_j}}
}
\newcommand\sumprimesdistinctdos[2]{
\sum_{\substack{p_{#1},\ldots, p_{#2} \\ p_\alpha \neq p_\beta}}
}
\newcommand\esum{\sum_{\substack{e\mid L_2^\infty\\ e<E}}}

\def\nonsensen{\nonsense{n}}

\newcommand\phinonsense[1]{\frac{\log p_{#1}}{\sqrt{p_{#1}}}\widehat\Phi_{#1}\left(\frac{\log p_{#1}}{\log q}\right) }

\newcommand\phinonsenselambda[1]{\frac{\log p_{#1}\ \lambda_f(p_{#1})}{\sqrt{p_{#1}}}\widehat\Phi_{#1}\left(\frac{\log p_{#1}}{\log q}\right) }

\newcommand\sumoverpdistinct{\sum_{\substack{p_1, \ldots, p_n \nmid q\\ p_i\neq p_j}}}

\newcommand\princnonsense[2]{
\frac{1}{\varphi\left(\frac{(b,N/b)}{( #1 ,(b,N/b))}\right)}
        \frac{1}{L(1+2it,\ol{{ #2 }^2}{\chi}_0)}
}

\newcommand\conjdnonsense[2]{
\sum_{\substack{d\mid #1\\ (d,N/b)=1}}d^{2it} \conjchinonsense{#1}{#2}{d}}

\newcommand\conjdnonsenseprinc[2]{
\sum_{\substack{d\mid m\\ (d,N/b)=1}}\frac{d^{2it}}{\varphi\left(\frac{(b,N/b)}{(m,(b,N/b))}\right)}
\frac{
\mu\left(\frac{(b,N/b)}{(m,(b,N/b))}\right)\varphi{((b,N/b))}
}
{L(1-2it,{#2}_\circ^2 \chi_0)}
}

\newcommand\conjchinonsense[3]{
\ifthenelse{
\equal{#3}{1}
}
{\frac{1}{\varphi\left(\frac{(b,N/b)}{( #1 ,(b,N/b))}\right)}
        \sum_{#2 \bmod \frac{(b,N/b)}{( #1 ,(b,N/b))}}\frac{ \ol{#2} (-\frac{#1}{( #1 ,(b,N/b))}\ol{b_0} a)\ol{\tau( \ol{#2}) }}{L(1-2it, {{ #2 }^2}{\chi}_0)}}
{\frac{1}{\varphi\left(\frac{(b,N/b)}{( #1 ,(b,N/b))}\right)}
        \sum_{#2 \bmod \frac{(b,N/b)}{( #1 ,(b,N/b))}}\frac{ \ol{#2} (-\frac{#1}{( #1 ,(b,N/b))} \ol{b_0{#3}^2} a)\ol{\tau( \ol{#2})}}{L(1-2it, {{ #2 }^2}{\chi}_0)}}
}

\newcommand\dnonsense[2]{
\sum_{\substack{d\mid #1\\ (d,N/b)=1}}d^{-2it} \chinonsense{#1}{#2}{d}}

\newcommand\conjdnonsenseabs[2]{
\sum_{\substack{d\mid #1\\ (d,N/b)=1}}\conjchinonsenseabs{#1}{#2}
}

\newcommand\conjchinonsenseabs[2]{
\frac{1}{\varphi\left(\frac{(b,N/b)}{(#1,(b,N/b))}\right)}\sum_{#2\bmod \frac{(b,N/b)}{(#1,(b,N/b))}} \frac{|\tau{(#2)}|} {|L(1-2it, \ol{{#2}^2}\chi_0)|}
}

\newcommand\chinonsense[3]{
\ifthenelse{
\equal{#3}{1}
}
{\frac{1}{\varphi\left(\frac{(b,N/b)}{( #1 ,(b,N/b))}\right)}
        \sum_{#2 \bmod \frac{(b,N/b)}{( #1 ,(b,N/b))}}\frac{ #2 (-\frac{#1}{( #1 ,(b,N/b))}\ol{b_0} a)\tau(\ol {#2}) } {L(1+2it, \ol{{ #2 }^2}{\chi}_0)}}
{\frac{1}{\varphi\left(\frac{(b,N/b)}{( #1 ,(b,N/b))}\right)}
        \sum_{#2 \bmod \frac{(b,N/b)}{( #1 ,(b,N/b))}}\frac{ #2 (-\frac{#1}{( #1 ,(b,N/b))} \ol{b_0{#3}^2} a)\tau(\ol {#2})} {L(1+2it,\ol{{ #2 }^2}{\chi}_0)}}
}

\newcommand\hkq[0]{
\sum_{f \in \mcH_k(q)}\mkern -17mu^h\ 
}

\newcommand\bkd[0]{
\sum_{f \in B_k(d)}\mkern -17mu^h\ 
}

\newcommand\amod[0]{
\sum_{a\bmod (b,N/b)}\mkern -31mu^*\mkern25mu
}

\newcommand\bsum[0]{
\sum_{\substack{b\mid N\\ \text{\ref{eqn:8primecondition f}}}}
}

\newcommand\chinonsenseabs[3]{
\ifthenelse{
\equal{#3}{1}
}
{\frac{1}{\varphi\left(\frac{(b,N/b)}{( #1 ,(b,N/b))}\right)}
        \sum_{#2 \bmod \frac{(b,N/b)}{( #1 ,(b,N/b))}}\frac{ \left|\tau(\ol {#2})\right| }{\left|L(\ol{{ #2 }^2}{\chi}_0, 1+2it)\right|}}
{\frac{1}{\varphi\left(\frac{(b,N/b)}{( #1 ,(b,N/b))}\right)}
        \sum_{#2 \bmod \frac{(b,N/b)}{( #1 ,(b,N/b))}}\frac{ \left|\tau(\ol {#2})\right|}{\left|L(\ol{{ #2 }^2}{\chi}_0, 1+2it)\right|}}
}

\newcommand\chinonsensenonprinc[3]{
\ifthenelse{
\equal{#3}{1}
}
{\frac{1}{\varphi\left(\frac{(b,N/b)}{( #1 ,(b,N/b))}\right)}
        \sum_{\substack{#2 \bmod \frac{(b,N/b)}{( #1 ,(b,N/b))}\\ #2 \neq #2_0}}\frac{ #2 (-\frac{#1}{( #1 ,(b,N/b))}\ol{b_0} a)\tau(\ol {#2}) }{L(\ol{{ #2 }^2}{\chi}_0, 1+2it)}}
{\frac{1}{\varphi\left(\frac{(b,N/b)}{( #1 ,(b,N/b))}\right)}
        \sum_{\substack{#2 \bmod \frac{(b,N/b)}{( #1 ,(b,N/b))}\\ \\ #2 \neq #2_0}}\frac{ #2 (-\frac{#1}{( #1 ,(b,N/b))} \ol{b_0{#3}^2} a)\tau(\ol {#2})}{L(\ol{{ #2 }^2}{\chi}_0, 1+2it)}}
}

\newcommand\eightKloos[1]{S\left(\frac{ #1 }{(#1 ,(b,N/b))},0;b_0\right)}

\newcommand\eightlog[1]{\prod_{#1} \frac{\log p_{#1}}{\sqrt{p_{#1}}}  V\left(\frac{p_{#1}}{P_{#1}}\right)\mathrm{e}\left(v_{#1}\frac{p_{#1}}{P_{#1}}\right)}

\let\ol\overline

\def\lodts{\ldots}

\theoremstyle{plain}\newtheorem{thm}{Theorem}[section]
\theoremstyle{plain}\newtheorem*{thm*}{Theorem}
\theoremstyle{plain}\newtheorem{cor}[thm]{Corollary}
\theoremstyle{plain}
\theoremstyle{plain}\newtheorem{lemma}[thm]{Lemma}
\theoremstyle{plain}\newtheorem{prop}[thm]{Proposition}
\theoremstyle{plain}\newtheorem{remark}[thm]{Remark}
\theoremstyle{plain}\newtheorem{defn}[thm]{Definition}

\usepackage[style=alphabetic, maxnames=50]{biblatex}

\usepackage{fullpage}
\usepackage{xstring}

\usepackage{hyperref}
\usepackage[toc]{appendix}
\usepackage{dsfont}
\usepackage{indentfirst}

\hypersetup{
    colorlinks=true,
    linkcolor=blue,
    filecolor=magenta,      
    urlcolor=cyan,
    pdftitle={On the Density of Low Lying Zeros of a Large Family of Automorphic L-functions},
    pdfpagemode=FullScreen,
    citecolor=blue,
    pdfauthor={Timothy Cheek, Pico Gilman, Kareem Jaber, Steven J. Miller, Marie-H\'el\`ene Tom\'e}
}

\let\oldtocsection=\tocsection

\let\oldtocsubsection=\tocsubsection

\renewcommand{\tocsection}[2]{\hspace{0em}\oldtocsection{#1}{#2}}
\renewcommand{\tocsubsection}[2]{\hspace{2em}\oldtocsubsection{#1}{#2}}
\usepackage{xpatch}
\makeatletter   
\xpatchcmd{\@tocline}
{\hfil\hbox to\@pnumwidth{\@tocpagenum{#7}}\par}
{\ifnum#1<0\hfill\else\dotfill\fi\hbox to\@pnumwidth{\@tocpagenum{#7}}\par}
{}{}
\makeatother

\makeatletter
\edef\orig@output{\the\output}
\output{\setbox\@cclv\vbox{\unvbox\@cclv\vspace{0pt plus 20pt}}\orig@output}
\makeatother

\makeatletter
\renewcommand*\@maketitle{%
  \normalfont\normalsize
  \@adminfootnotes
  \@mkboth{\@nx\shortauthors}{\@nx\shorttitle}%
  \global\topskip42\p@\relax 
  \@settitle
  \ifx\@empty\authors \else \@setauthors \fi
  \ifx\@empty\@date \else {\vskip 1em \vtop{\centering\large\@date\@@par}}\fi
  \ifx\@empty\@dedicatory
  \else
    \baselineskip18\p@
    \vtop{\centering{\footnotesize\itshape\@dedicatory\@@par}%
      \global\dimen@i\prevdepth}\prevdepth\dimen@i
  \fi
  \@setabstract
  \normalsize
  \if@titlepage
    \newpage
  \else
    \dimen@34\p@ \advance\dimen@-\baselineskip
    \vskip\dimen@\relax
  \fi
} 
\makeatother

\bibliography{references.bib}

\title{On the density of low lying zeros of a large family of automorphic \texorpdfstring{$L$}{}-functions}

\author{Timothy Cheek}
\address{Department of Mathematics, University of Wisconsin}
\email{\href{mailto:tcheek@wisc.edu}{tcheek@wisc.edu}}

\author{Pico Gilman}
\address{Department of Mathematics, Massachusetts Institute of Technology}
\email{\href{mailto:pico@mit.edu}{pico@mit.edu}}

\author{Kareem Jaber}
\address{Department of Mathematics, Princeton University}
\email{\href{mailto:kj5388@princeton.edu}{kj5388@princeton.edu}}

\author{Steven J. Miller}
\address{Department of Mathematics, Williams College}
\email{\href{mailto:sjm1@williams.edu}{sjm1@williams.edu}}

\author{Marie-H\'el\`ene Tom\'e}
\address{Department of Pure Mathematics and Mathematical Statistics, University of Cambridge}
\email{\href{mailto:mht47@cam.ac.uk}{mht47@cam.ac.uk}}
\date{} 

\numberwithin{equation}{section}
\makeatletter
 \def\@textbottom{\vskip \z@ \@plus 3pt}
 \let\@texttop\relax
\makeatother

\begin{document}
\begingroup
\allowdisplaybreaks

\subjclass[2020]{11F11, 11M41 (primary), 11M50 (secondary)}
\keywords{Low-lying zeros, $n$-level densities, $n$-th centered moments, Katz-Sarnak Conjectures, Orthogonal family, $GL(2)$ $L$-functions.} 

\begin{abstract}
Under the generalized Riemann Hypothesis (GRH), Baluyot, Chandee, and Li nearly doubled the range in which the density of low lying zeros predicted by Katz and Sarnak is known to hold for a large family of automorphic $L$-functions. Generalizing their results, we prove the Katz-Sarnak density predictions hold for the $n$-th centered moments for test functions whose Fourier transform is compactly supported in $(-\sigma, \sigma)$ for $\sigma = \min\left\{3/2(n-1), 4/(2n-\mathbf{1}_{2\nmid n})\right\}$. For $n=3$, our results improve the previously best known $\sigma=2/3$ to $\sigma=3/4$. We also prove the two-level density agrees with Katz-Sarnak when $\sigma_1 = 3/2$ and $\sigma_2 = 5/6$, respectively, extending the previous best-known sum of supports $\sigma_1 + \sigma_2 = 2$. This work is the first evidence of a new phenomenon: by taking different test functions, we extend the range in which the Katz-Sarnak density predictions are known to hold. Our techniques have applications to understanding related quantities containing sums over multiple primes.
\end{abstract}

\maketitle
\tableofcontents

\section{Introduction}
The Katz-Sarnak philosophy \cite{katz1, katz2} predicts that the distribution of zeros near the central point of a family of $L$-functions ordered by their analytic conductor agrees with the distribution of eigenvalues near $1$ of a classical compact group depending on the symmetry type of the family. Katz and Sarnak \cite{katz1} introduced the one-level density to study these arithmetically important zeros near the central point using smooth test functions.

Iwaniec, Luo, and Sarnak \cite{ILS} studied the one-level density of the family of automorphic $L$-functions associated to forms $f$ lying in $\mcH_k(q)$, an orthogonal basis of Hecke eigenforms for the space of cuspidal newforms of weight $k$, level $q$, and trivial nebentypus. Assuming the generalized Riemann Hypothesis (GRH), they established agreement with the density predicted by Katz and Sarnak for test functions whose Fourier transform is supported in $(-2, 2)$ as $q \rightarrow \infty$ through squarefree integers. Later, \cite{barrett} removed the squarefree condition. Recently, Baluyot, Chandee, and Li \cite{BCL} studied a larger family associated to forms in $\mcH_k(q)$ averaged over all $q \asymp Q$, i.e., $q \ll Q$ and $Q \ll q$. For this family, \cite{BCL} developed novel techniques, allowing them to prove agreement with the density predicted by Katz and Sarnak as $Q \rightarrow \infty$ for test functions whose Fourier transform is supported in $(-4,4)$. This result almost doubles the previously best known support.

In addition to extending the support, more information about low lying zeros can be gained by studying higher level densities. While the two-level density is tractable, e.g., see \cite{MillerEC}, when $n>2$, the combinatorics involved require knowing the distribution of the sign of the functional equation, which is beyond current theory. To this end, Hughes and Rudnick \cite{hughesrudnick} introduced a related quantity, the $n$-th centered moment of the one-level density. Like the $n$-level density, this statistic studies $n$-tuples of zeros, yielding better bounds on the order of vanishing at the central point. Assuming GRH for $L(s,f)$ and Dirichlet $L$-functions, Hughes and Miller \cite{Hughes_2007} showed agreement between the Katz-Sarnak density predictions and the $n$-th centered moment of the one-level density for the family of $L$-functions associated to $f \in \mcH_k(q)$ for weight $k \geq n/2$ and test functions whose Fourier transform is suppported in $\sigma = 2/n$ as $q \rightarrow \infty$ through squarefree integers. When splitting this family by sign, \cite{Hughes_2007} obtained support $\sigma = \min\{1/(n-1), 2/n\}$. Subsequent work by \cite{10yearsmall} removed the $k\geq n/2$ condition and extended the support to $\sigma = 2/n$ when splitting by sign. Previous work shows that $\sigma = 2/n$ is a natural boundary. A natural question to ask is whether the methods of \cite{BCL} can be generalized to extend the support for the two-level density and the higher centered moments.

We study the family of $L$-functions associated with $f \in \mcH_k(q)$ averaged over $q \asymp Q$ and show that this is indeed the case. More precisely, we show the Katz-Sarnak philosophy holds for the $n$-th centered moment of the one-level density as $Q \rightarrow \infty$ for test functions whose Fourier transform is supported in $(-\sigma,\sigma)$ with $\sigma = \min\{3/2(n-1), 4/(2n-\mathbf{1}_{2\nmid n})\}$. Moreover, we show the two-level density of the low lying zeros for this family agrees with the predictions of Katz and Sarnak for test functions whose Fourier transform is supported in $\sigma_1 = 3/2$ and $\sigma_2 = 5/6$, respectively, which sum to support $7/3 > 2$, the previously best known result. Thus, we extend the range in which the Katz-Sarnak density conjecture is known to hold. In generalizing the methods of Baluyot, Chandee, and Li, we encounter several technical obstructions. Handling the spectral terms resulting from the Kuznetsov trace formula requires distinct primes, and we use a combinatorial argument to reduce our analysis to this case. To facilitate our application of the Kuznetsov trace formula, we must add back the primes dividing the level of our forms, i.e., $O(q^{(n-1)\sigma}\log(q))$ many terms. Consequently, when $n = 1$, bounding the contribution from these terms is trivial and follows from a direct application of the Petersson trace formula; however, when $n \geq 2$, this step creates a material technical obstruction. For test functions with suitably restricted support, we are able to add back these primes but fundamentally different techniques would need to be developed to remove this restriction. It is only to handle the terms arising from this step that we require $\sigma \leq 3/2(n-1)$; hence by adding back the primes dividing the level using another approach, the techniques we develop would allow one to prove that the density predicted by Katz and Sarnak holds in the extended range $\sigma = 4/(2n-\mathbf{1}_{2\nmid n})$. Extending the support beyond this point would require handling the novel terms that arise from Eisenstein series coefficients. Here, terms emerging only when $n \geq 2$ introduce competing sign influences that thwart a key contour shift. As we cannot move the contour past the critical line, this is a natural limit for these techniques.

To make our results precise, let $f(z) = \sum_{n=1}^\infty \psi_f(n)(4\pi n)^{k/2}\mathrm{e}(nz) \in S_k(q)$ be a holomorphic cuspform with trivial nebentypus, even integral weight $k \geq 4$, and level $q$. If $f$ is an eigenform of all the Hecke operators and the Atkin-Lehner involutions $|_kW(q)$ and $|_kW(Q_p)$ for all the primes $p \mid q$, then we say that $f$ is a newform and if, in addition, $f$ is normalized so that $\psi_f(1) = 1$, then we say that $f$ is primitive (see \cite[Section 2.5]{Ono}). The space $S_k^{\mathrm{new}}(q)$ of newforms has an orthogonal basis $\mcH_k(q)$ of primitive Hecke eigenforms; hence if we let $\lambda_f(n)n^{(k-1)/2}$ be the Hecke eigenvalue of $f$, then for $f \in \mcH_k(q)$ we have that $\lambda_f(n) = \sqrt{n}\psi_f(n)$. For each $f \in \mcH_k(q),$ the associated $L$-function $L(s,f)$ is defined for $\Re(s) > 1$ by
\begin{align}
    L(s,f) \ \coloneqq \ \sum_{n = 1}^\infty \frac{\lambda_f(n)}{n^{s}} \
    &= \ \prod_p \left(1-\frac{\lambda_f(p)}{p^{s}} + \frac{\chi_0(p)}{p^{2s}}\right)^{-1} \nonumber \\
    &= \ \prod_p \left(1-\frac{\alpha_f(p)}{p^s}\right)^{-1}\left(1 - \frac{\beta_f(p)}{p^s}\right)^{-1}\label{eqn:Lfunc},
\end{align}
where $\chi_0$ is the principal character mod $q$. Since $f$ is a newform, $L(s,f)$ can be completed and analytically continued to the entire complex plane. Further, it has a functional equation with root number $\epsilon_f = \pm1.$ Under GRH for $L(s,f)$, the nontrivial zeros of $L(s,f)$ are all of the form $\rho_f = 1/2 + i \gamma_f$ with $\gamma_f \in \rr$, and they occur in complex conjugate pairs. We let $\gamma_f(j)$ be an enumeration of the imaginary parts of the nontrivial zeros of $L(s,f)$ counted with multiplicity where $\gamma_f(j) = -\gamma_f(-j)$, and $\gamma_f(0)$ appears if and only if $\epsilon_f = -1$. See \cite{IK} for a general reference on automorphic forms.

Throughout, we let $\Phi$ be an even Schwartz function whose Fourier transform
    \begin{align}
        \widehat\Phi(x)\ =\ \int_{-\infty}^\infty \Phi(t) \mathrm e(-xt)\,dt\  =\ \int_{-\infty}^\infty \Phi(t) \mathrm e^{-2\pi ixt}\,dt
    \end{align}
is compactly supported. For the family of $L$-functions associated to forms $f \in \mcH_k(q)$, the number of zeros up to some fixed constant $T_0$ is $\Theta(\log q)$. As such, Katz and Sarnak defined the one-level density as (see \cite[p.~405]{katz1}) 
\begin{align}
    \mathscr{OD}(f; \Phi) \ \coloneqq \ \sum_{j} \Phi\Bigg( \frac{\gamma_f(j)}{2\pi}\log q \Bigg).
\end{align}
Due to the scaling by $\log q/2\pi$ and the decay of $\Phi$, the contribution of the nontrivial zeros of $L(s,f)$ whose distance from the central point $s = 1/2$ is $O(1/\log q)$ is weighted most heavily. Each $L$-function has only a small number of zeros within this range, and thus in order to evaluate the asymptotic behavior of such sums, we must introduce additional averaging over the family. For the orthogonal family we study, the Katz-Sarnak density conjecture predicts that
    \begin{align}
        \lim_{q \rightarrow \infty} \frac{1}{|\mcH_k(q)|} \sum_{f \in \mcH_k(q)} \mathscr{OD}(f; \Phi)\ =\ \int_{-\infty}^\infty \Phi(x)\left(1 + \frac{1}{2}\delta_0(x)\right)\,dx,
    \end{align}
where $W(O)(x) = 1 + \frac{1}{2}\delta_0(x)$ is the distribution for the scaling limit of orthogonal groups $O(N)$ calculated by Katz-Sarnak \cite[p. 409]{katz1} and $\delta_0(x)$ is the Dirac delta distribution at $x = 0$.

While proving the Katz-Sarnak conjecture is far beyond current techniques, results for certain families have been obtained by restricting the support and introducing additional averaging. To interface with spectral theory, it is useful to define the following harmonic average of a set of numbers $\eta_f$ over $f \in \mcH_k(q)$:
\begin{align}
    \hkq \eta_f\ \coloneqq\  \frac{\Gamma(k-1)}{(4\pi)^{k-1}} \sum_{f\in \mcH_k(q)} \frac{\eta_f}{||f||^2},\label{harmonicaverage}
\end{align}
where $||f||^2 \coloneqq \int_{\Gamma_0(q) \backslash \mathbb{H}} |f(z)|^2y^{k-2}\,dx\,dy$ denotes the squared Petersson norm of $f$. As in \cite{BCL}, we introduce additional averaging over $q \asymp Q$. Let $\Psi$ be a non-negative, smooth function compactly supported on $(0,\infty)$ with $\widehat\Psi(0) > 0$ and 
for any $X$, define the notation
\begin{align}
    \mathbb{E}_Q[X]\ :=\ \nonsense{0}\hkq X,
\quad\text{and} \quad
    \langle X \rangle_* \ \coloneqq \ \lim_{Q\to\infty}
    \mathbb{E}_Q[X]
\end{align}
where
\begin{align}
    N(Q)\ \coloneqq \ \sum_q \Psi\left(\frac qQ\right) \hkq 1.
\end{align}
By the work of \cite{Martin}, since $\widehat\Psi(0) \neq 0$, we have that $N(Q) \asymp Q$. Baluyot, Chandee, and Li (see Theorem 1.1 of \cite{BCL}) proved that for $\widehat\Phi$ compactly supported in $(-4,4)$,
\begin{align}
    \langle \mathscr{OD}(f; \Phi) \rangle_*\ =\ \int_{-\infty}^\infty \Phi(x)\left(1 + \frac{1}{2}\delta_0(x)\right)\,dx.
\end{align}
For $n \geq 2$, let $\Phi_i$ be even Schwartz functions whose Fourier transform is compactly supported in $(-\sigma_i, \sigma_i)$ and define the $n$-th centered moment of the one-level density to be
    \begin{align}
        \mathscr{M}(\Phi_1,\ldots,\Phi_n) \ \coloneqq \ \left\langle \prod_{i=1}^n \Bigg[\mathscr{OD}(f,\Phi_i) - \langle \mathscr{OD}(f,\Phi_i) \rangle_* \Bigg]\right\rangle_*.
    \end{align}
Our main result is as follows.
\begin{thm}\label{main thm}
    Assume the Riemann hypothesis, GRH for $L(s, f)$, $L(s, \mathrm{sym}^2 f)$, and Dirichlet $L$-functions, and the notation and hypotheses as above. Additionally, suppose that $\{\sigma_i \colon \ 1 \leq i \leq n\}$ is such that for any $1 \leq i \leq n$
    \begin{align}\label{cond:smallsigma}
        \sum_{j \neq i} \sigma_j\ \leq\ \frac{3}{2}
    \end{align}
    and for any partition of $\{1,2,\ldots, n\}$ into two sets $A,B$, either 
    \begin{align}\label{cond:bipart}
        \sum_{i \in A}\sigma_i+3\sum_{j\in B}\sigma_j\ \leq\ 4\hspace{1cm}\text{ or}\hspace{1cm} 3\sum_{i \in A}\sigma_i + \sum_{j \in B}\sigma_j\ \leq\ 4.
    \end{align} 
    For instance, one may take $\sigma_i = \sigma = \min\left\{\frac{3}{2(n-1)}, \frac{4}{2n-\mathbf{1}_{2\nmid n}}\right\}$ for each $i$. Then
    \begin{align}
        \mathscr{M}(\Phi_1,\ldots,\Phi_n) \ = \ \frac{\mathbf{1}_{2\mid n}}{(n/2)!} \sum_{\tau\in S_n} \prod_{i=1}^{n/2} \int_{-\infty}^\infty |u|\widehat\Phi_{\tau(2i-1)}(u)\widehat\Phi_{\tau(2i)}(u)\,du.
    \end{align}
\end{thm}
As a corollary, we obtain the following result toward the Katz-Sarnak density conjecture.
\begin{cor}\label{cor:twolevel}
     Assume the Riemann hypothesis, GRH for $L(s, f)$, $L(s, \mathrm{sym}^2 f)$, and Dirichlet $L$-functions, and the notation and hypotheses as above. Let $\sigma_1=3/2$ and $\sigma_2=5/6$. Then the two-level density satisfies
        \begin{align}
            \left\langle \mathop{\sum\sum}_{j_1\neq \pm j_2} \Phi_1\bigg(\frac{\gamma_f(j_1)\log q}{2\pi}\bigg)\Phi_2\bigg(\frac{\gamma_f(j_2)\log q}{2\pi}\bigg) \right\rangle_* \ &= \ 2\int_{-\infty}^\infty |u|\widehat\Phi_1(u)\widehat\Phi_2(u)\,du + \prod_{i=1}^2 \big( \frac12\Phi_i(0)+\widehat\Phi_i(0)\big) \nonumber \\
            &\hspace{.2cm} - \Phi_1\Phi_2(0)-2\widehat{\Phi_1\Phi_2}(0) + {\mathcal{ODD}}\ \Phi_1\Phi_2(0),
        \end{align}
    where $\mathcal{ODD}~\coloneqq~\langle (1-\epsilon_f)/2 \rangle_*$ denotes the proportion of forms with odd functional equation, agreeing with the Katz-Sarnak density conjecture.
\end{cor}

\begin{remark}
   The novelty of this result lies in that the sum of supports can only be extended beyond the known range by taking the test functions to be different. To the authors knowledge, this is the first time choosing different test functions improves a density result. 
\end{remark}

We now sketch the proof of Theorem \ref{main thm}. An explicit formula (Lemma \ref{lemma:their2.5}) allows us to convert the above sum over zeros into a sum over primes, i.e., to the limit in Theorem \ref{onylallpairedoralldifferent}. We then use a delicate combinatorial argument coupled with induction on $n$ to reduce to the case where all $n$ primes are distinct. Within this combinatorial argument, whenever the primes are not perfectly paired, we either have at least three primes equal, in which case we extract decay from these terms and reduce to a smaller case; or if some of the primes are distinct and the remaining ones appear exactly twice, then we use H\"older's inequality to split the sums over distinct primes from the other prime sums and appeal to a lower case. We require the primes to be distinct so that bounding the spectral terms resulting from the Kuznetsov trace formula is tractable.

The work of \cite{ng} allows us to convert a harmonic average over $\mcH_k(q)$ to a weighted sum of averages over an orthonormal basis of $S_k(d)$ for all $d \mid q$. Next, we apply the Petersson trace formula to remove the dependence of our prime sums on $q$. Then, we apply the Kuznetsov trace formula and individually bound the resulting spectral terms. We next generalize the method of \cite[Section 7]{BCL} to higher centered moments to bound the holomorphic and discrete parts of the spectrum. Our calculations in this section rely on choosing the orthonormal bases for the space of holomorphic cuspforms and the space of Maass cuspforms developed in \cite{Blomer}.

In contrast, the methods developed in \cite[Section 8]{BCL} to handle the contribution from the Eisenstein series part of the spectrum are more difficult to generalize. Handling novel terms that arise in the explicit calculations of the Fourier coefficients of Eisenstein series \cite{Kiral} when there are multiple primes requires detailed analysis. We split the primes into sets $\fp,\fo,\fq$ based on divisibility conditions. We observe that these conditions imply the primes in $\fp$ are all relatively small and hence easy to bound. The primes in $\fo$ produce novel terms arising only when $n \geq 3$. While they are also small, bounding them is more difficult and requires greater insight. The primes in $\fq$ do not divide the level and their contribution is a natural limit for the support we obtain. In the case $n = 1$, these terms are bounded using a contour shift. In the general setting, one cannot uniformly reduce the size of $\prod P_\alpha^{1/2\pm it}$ using the complex part of the exponent. New terms that emerge from the Fourier coefficients of Eisenstein series force us to handle all possible sign combinations on these exponents. Further reducing the size of the exponents would require moving the contour beyond the critical line, hence these competing signs are a clear limit for extending the method of \cite{BCL}. We leverage the lack of perfect sign pairings in these exponents to achieve an improvement for $n$ odd.

In Section \ref{section2}, we introduce our notation, the necessary background on automorphic forms, and several preliminary results. Then, in Section \ref{sec3}, we detail the combinatorial argument which allows us to reduce to bounding a sum over distinct primes. Next, in Section \ref{section4}, we apply the Petersson trace formula, truncate in the less important parameters arising from the conversion from $\mcH_k(q)$ to a basis for $S_k(d)$ for $d \mid q$, and add back the primes dividing the level. We then apply the Kuznetsov trace formula and bound the resulting sums via Propositions \ref{Proposition 6.2} and \ref{prop:their6.3}, whose proofs occur in Sections \ref{sec: DISC} and \ref{sec:EisensteinSeries}, respectively. In Section \ref{sec:EisensteinSeries}, we handle the continuous part of the spectrum, splitting into the case of principal and non-principal characters.

\begin{remark}
    When the authors first publically posted this paper, the results provided novel improvement over pre-existing results. However, in the time since when this paper was under review, \cite{chandee2025nthcenteredmomentslarge} improved upon several of this paper’s results, though there are still differences between the two works (for example, this paper provides the first instance where there is a benefit in having the test function factor as a product of non-identical test functions with different support).
\end{remark}

\section{Notation and Preliminary Results}\label{section2}
Throughout, $p_i$ always denotes a prime and $\fp, \fo, \fq, \ff, \fh$ and $\fg$ denote sets of primes satisfying certain conditions, which we specify when introducing these sets. For $f:\rr\to\cc$, we define

\[ \sum_\fp \prod_\fp f(\fp) \ \coloneqq \ \sum_{\substack{(p_1,\ldots, p_\ell) \in \fp^\ell \\ p_i\neq p_j}} \prod_{i=1}^\ell f(p_i)  \ = \ \mathop{\sum \cdots \sum}_{\substack{p_1, \ldots, p_{\ell} \\ p_i\neq p_j \\  \text{Conditions}}} \prod_{i=1}^{\ell} f(p_i).\]
All implied constants and Big Oh's depend only on $n$, $k$, $\e$, and the test functions $\Phi_i$ unless otherwise stated. We let $\e$ denote an arbitrary positive number and allow its value to change from line to line. 

\subsection{Automorphic Forms I}
Here, we introduce the spectral theory relevant to the Kuznetsov trace formula and follow the notations of \cite{BHM, Blomer}. Specifically, we introduce the three parts of the spectrum which arise from an application of the Kuznetsov trace formula and which we handle in the paper.

\emph{Holomorphic forms.} For any holomorphic cuspform $f$ of level $q$, weight $k$, and trivial nebentypus, we write its Fourier expansion as 
\begin{align}\label{eqn:FourierNormalization}
    f(z)\ =\ \sum_{n\geq 1} \psi_f(n)(4\pi n)^{k/2}\mathrm{e}(nz)
\end{align}
and define $\psi_f(n)$ to be the $n$-th Fourier coefficient of $f$.

If $f \in S_k(q)$ is an eigenfunction of all the Hecke operators $T_n$ for $(n,q)=1$, then we call $f$ a Hecke eigenform and write $\lambda_f(n)$ for the $n$-th Hecke eigenvalue of $f$. Since the operators $T_n$ are self-adjoint, for $(n,q) =1$ the eigenvalues $\lambda_f(n)$ are real. If $f$ is additionally a newform, then $f$ is an eigenform of $T_n$ for \textit{all} $n$ and $\lambda_f(p)^2 = p^{-1}$ for $p \mid q$ \cite[(2.24)]{ILS}.

When $f \in S_k(q)$ and $(n,q) = 1$, we have that
\begin{align}\label{eqn:FourierCoeffsHeckeEvalsRel}
    \lambda_f(n)\psi_f(1)\ =\ \sqrt{n}\psi_f(n) \quad \text{and} \quad
    \lambda_f(n)\ \ll_f\ \tau(n)\ \ll\ n^\e.
\end{align}
If $f$ is a newform, then the above holds for all $n$, we have the sharper bound $|\lambda_f(n)| \leq \tau(n)$, and the Hecke eigenvalues are multiplicative \cite[p.~68]{ILS}:
\begin{align}\label{eqn:Heckemult}
    \lambda_f(m)\lambda_f(n)\ =\ \sum_{\substack{d\mid (m,n)\\ (d,q)=1}}\lambda_f\left(\frac{mn}{d^2}\right).
\end{align}

\emph{Maass forms.} Let $f$ be a Maass form on $\Gamma_0(q)$ with spectral parameter $t\coloneqq \kappa_f\in \rr\cup[-i\theta,i\theta]$ (the best known value is $\theta=7/64$ due to Kim-Sarnak \cite{theta}). We write its Fourier expansion as
\begin{align}
    f(z)\ =\ \sum_{n\neq 0}\rho_f(n)W_{0,it}(4\pi|n|y)\mathrm{e}(nx)
\end{align}
for $z=x+iy$, where $W_{0,it}(y)=(y/\pi)^{\frac{1}{2}}K_{it}(y/2)$ is the Whittaker function and $K_{it}$ is the modified Bessel function of the second kind. We define $\rho_f(n)$ to be the $n$-th Fourier coefficient of $f$.
 
As in the holomorphic case, each Maass form is an eigenfunction of the Hecke operators $T_n$ for $(n,q) = 1$. When $(n,q) = 1$ or $f$ is a newform, we have
\begin{align}
    \lambda_f(n)\rho_f(1)\ =\ \sqrt{n}\rho_f(n).
\end{align}
We always have the bound
\begin{align}
    \lambda_f(n)\ \ll\ \tau(n)n^\theta\ \ll\ n^{\theta+\e}.
\end{align}

\emph{Eisenstein series.} Let $\fc$ be a cusp of $\Gamma_0(q)$. We write the Fourier expansion of the (real-analytic) Eisenstein series at $1/2+it$ as
\begin{align}
    E_\fc\left(z; \frac{1}{2}+it\right)\ =\ \delta_{\fc=\infty}y^{\frac{1}{2}+it}+\varphi_\fc(0,t)y^{\frac{1}{2}-it}+\sum_{n\neq 0}\varphi_\fc(n,t)W_{0,it}(4\pi |n|y)\mathrm{e}(nx)
\end{align}
for $z=x+iy$. We define $\varphi_\fc(n,t)$ to be the $n$-th Fourier coefficient of $E_\fc(z,1/2+it)$.

\emph{The Kuznetsov trace formula}. We use the version stated in Lemma 10 of \cite{Blomer}. 
\begin{lemma}\label{lemma:kuznetsov}
For $\phi:(0,\infty)\rightarrow \cc$ smooth and compactly supported, and $m,n,q\in \zz^+$, we have that
\begin{align}
    \sum_{\substack{c\geq 1\\ c\equiv 0\bmod q}}\frac{S(m,n;c)}{c}\phi\left(4\pi \frac{\sqrt{mn}}{c}\right)\ &=\ \sum_{\substack{\ell\geq 2\text{ even}\\1\leq j\leq \theta_\ell(q)}}(\ell-1)!\sqrt{mn} \ \ol{\psi_{j,\ell}}(m)\psi_{j,\ell}(n)\phi_h(\ell)\nonumber\\ 
    &\hspace{.5cm}+ \sum_{j=1}^\infty \frac{\ol{\rho_j(m)}\rho_j(n)\sqrt{mn}}{\cosh(\pi \kappa_j)}\phi_+(\kappa_j)\nonumber\\
    &\hspace{.5cm}+ \frac{1}{4\pi}\sum_{\fc}\int_{-\infty}^\infty \frac{\sqrt{mn}}{\cosh(\pi t)}\ol{\varphi_{\fc}(m,t)}\varphi_\fc(n,t)\phi_+(t)dt,
\end{align}
where $\theta_\ell(q)$ is the dimension of $S_\ell(q)$, the $j$-sum is over an orthonormal basis $\{u_j\}_{j=1}^\infty$ for the space of Maass cuspforms of level $q$, $\kappa_j$ is the spectral parameter of $u_j$, and the Bessel transforms $\phi_h$ and $\phi_+$ are defined by
\begin{align}
    \phi_h(\ell)\ \coloneqq\ 4i^k\int_0^\infty J_{\ell-1}(\xi)\phi(\xi)\frac{d\xi}{\xi}
\end{align}
and
\begin{align}
    \phi_+(r)\ \coloneqq\ \frac{2\pi i}{\sinh(\pi r)}\int_0^\infty (J_{2ir}(\xi)-J_{-2ir}(\xi))\phi(\xi)\frac{d\xi}{\xi}.
\end{align}
Here and throughout, $J_{\ell-1}$ is the Bessel function of the first kind.
\end{lemma}

To handle the Bessel transforms resulting from applying the Kuznetsov trace formula, we recall the following useful bounds.
\begin{lemma}[Adapted from Lemma 3.3 and Lemma 8.8 of \cite{BCL}]\label{lemma:their3.3}
Suppose that $W$ is a smooth function that is compactly supported on $(0,\infty)$. For real $X>0$ and $u,\xi\in \rr$, let 
\begin{align}
    h_u(\xi)\ =\ J_{k-1}(\xi)W\left(\frac{\xi}{X}\right)\mathrm{e}(u\xi).
\end{align}
Then the following estimates hold:
\begin{enumerate}[(a)]
    \item $h_+(r), h_h(r) \ \ll\ \frac{1+|\log X|}{F^{1-\e}}\left(\frac{F}{1+|r|}\right)^C\min\left\{X^{k-1},\frac{1}{\sqrt{X}}\right\}$ for some $F<(|u|+1)(1+X)$,
    \item if $r\in(-1/4, 1/4)$, then $h_+(ir)\ \ll\ \left(\frac{1}{\sqrt{X}}+(1+|u|)^{1/2 } \right)\min\left\{X^{k-1},\frac{1}{\sqrt{X}}\right\}$, and
    \item if $0 < |\Im(z)| < \frac12$, then $h_+(z)\ \ll\ (1+|u|)\min\left\{X^{k-1-2|\Im(z)|},\frac{1}{\sqrt{X}}\right\}.$
\end{enumerate}
\end{lemma}

\subsection{Quasi-orthogonality relations for \texorpdfstring{$\GL(2)$}{}}
In this subsection, we introduce the Petersson trace formula which allows us to average over an orthonormal basis of $S_k(q)$.
\begin{defn}
    Let $S_k(q)$ be the space of cuspforms of weight $k$, level $q$, and trivial nebentypus. Let $B_k(q)$ be an orthogonal basis of $S_k(q)$ under the Petersson inner product. Define 
    \begin{align} 
        \Delta_q(m,n)\ =\ \Delta_{k,q}(m,n)\ =\ \sum_{f \in B_k(q)}\mkern -17mu ^h \lambda_f(m)\lambda_f(n), \label{delta defn}
    \end{align}
    where $\sum \mkern-3mu ^h$ denotes the harmonic average defined as in \eqref{harmonicaverage} and when $(n,q) > 1$, $\lambda_f(n) \coloneqq \sqrt{n}\psi_f(n)$.
\end{defn}

\begin{lemma}[Petersson trace formula]\label{lemma:Peterssontraceformula}
    If $m,n,q$ are positive integers, then
    \begin{align} 
        \Delta_q(m,n)\ =\ \delta(m,n) + 2\pi i^{-k} \sum_{c\geq 1} \frac{S(m,n;cq)}{cq} J_{k-1}\left(\frac{4\pi \sqrt{mn}}{cq}\right)
    \end{align}
    where $\delta(m,n) = 1$ when $m=n$ and $0$ otherwise and $S$ denotes a Kloosterman sum.
\end{lemma}

The following lemma follows directly from Corollary 2.2 of \cite{ILS}.
\begin{lemma}\label{lemma:their2.2}
    If $m,n,q$ are positive integers, then
    \begin{align}
        \Delta_q(m,n)\ =\  \delta(m,n) + O\left( \frac{\tau(q)(m,n,q)(mn)^\epsilon}{q\Big((m,q)+(n,q)\Big)^{1/2}} \Bigg(\frac{mn}{\sqrt{mn}+q}\Bigg)^{1/2} \right).
    \end{align}
\end{lemma}

To convert from a sum over $\mcH_k(q)$ to a weighted sum over $B_k(d)$ for $d \mid q$, we state the form of Theorem 3.3.1 of \cite{ng} given in Lemma 2.3 of \cite{BCL}.

\begin{lemma}\label{lemma:lemma2.3ofbcl}
    Suppose that $m,n,q$ are positive integers such that $(mn,q) = 1,$ and let $q = q_1q_2$ where $q_1$ is the largest factor of $q$ satisfying $p\mid q_1\iff p^2\mid q$ for all primes $p.$ Then 
    \begin{align}
        \Delta_q^*(m,n)\ \coloneqq \ \hkq \lambda_f(m)\lambda_f(n)\ =\ \sum_{\substack{q = L_1L_2d\\L_1\mid q_1\\L_2\mid q_2}} \frac{\mu(L_1L_2)}{L_1L_2} \prod_{\substack{p\mid L_1\\ p^2\nmid d}} \Bigg(1-\frac{1}{p^2}\Bigg)^{-1} \sum_{e\mid L_2^\infty} \frac{\Delta_d(me^2,n)}{e}.
    \end{align}
\end{lemma}
 
\begin{remark}\label{remarkafterlemma2.3}
As in the remark in \cite[p.~6]{BCL}, from $\mu(L_1L_2)$ we deduce that $(L_2,d)=1$ and thus also $(e,d)=1$. This is crucial in our analysis of $\Delta_d$.
\end{remark}

\subsection{Automorphic Forms II}\label{sec:autformsII}
For holomorphic cuspforms, Atkin and Lehner \cite{AL} introduced the concept of newforms and oldforms. Lemma \ref{lemma:lemma2.3ofbcl} converts a harmonic average of Hecke eigenvalues over an orthogonal basis of the space of newforms to a sum over $d \mid q$ of harmonic averages over an orthogonal basis of $S_k(d)$. The Petersson and Kuznetsov trace formulae convert the resulting sum into sums over orthogonal bases for the space of holomorphic cuspforms and the space of Maass cuspforms, respectively. These bases, in general, contain oldforms. However, newforms have many nice properties. For example, Lemma \ref{lemma:mostbeautifullemmaintheworld}, which we use throughout to bound our sums over primes, crucially relies on GRH for Hecke $L$-functions. Hence, we must express our bases in terms of newforms in order to apply these results. To this end, we introduce orthonormal bases for the space $\mcS_k(\mcN)$ of holomorphic cuspforms of weight $k$ and level $\mcN$ and the space $\mcA(\mcN)$ of Maass cuspforms of level $\mcN$, which are expressed in terms of newforms. 

Here, we state newform theory for Maass cuspforms and detail the construction of an orthonormal basis $\mcB(\mcN)$ for $\mcA(\mcN)$ given by \cite[Section 5]{Blomer}; the construction for holomorphic cuspforms is analogous. Each Maass form is a real-valued eigenfunction of the Laplacian $\Delta \coloneqq -y^2(\frac{\partial^2}{\partial x^2}+\frac{\partial^2}{\partial y^2} )$ acting as a linear operator on the space of cuspforms in $L^2(\Gamma_0(q)\backslash \hh)$. If $\Delta f = \lambda f$, then we write $\lambda~\coloneqq~1/4 + \kappa_f^2$ and call $\kappa_f$ the spectral parameter of $f$. 

Let $\mcA(\mcN, \kappa)$ denote the space of Maass cuspforms of level $\mcN$ and spectral parameter $\kappa$ so that $\mcA(\mcN) = \coprod_{\kappa} \mcA(\mcN, \kappa)$. In general, an $L^2$-basis $\mcB(\mcN, \kappa)$ for $\mcA(\mcN, \kappa)$ also includes oldforms. Let $\mcB^*(\mcM)$ denote the set of level $\mcM$ newforms normalized as level $\mcN$ forms, i.e., such that their first Fourier coefficient satisfies (see \cite[Section 3.1]{BCL})
    \begin{align}
        |\rho_f(1)|^2 \ = \ \frac{(\mcN \kappa_f)^{o(1)}}{\mcN}.
    \end{align}
Let $\mcB^*(\mcM, \kappa) \subset \mcB^*(\mcM)$ denote the set of such forms with fixed spectral parameter $\kappa$. For each $f \in \mcB^*(\mcM, \kappa)$ and $d$ such that $d \mcM \mid \mcN$, define
    \begin{align}
        f|_d(z) \ \coloneqq \ f(d z).
    \end{align}
By newform theory,
    \begin{align}
        \mcA(\mcN, \kappa) \ = \ \bigoplus_{\mcM \mid \mcN} \bigoplus_{f \in \mcB^*(\mcM, \kappa)} \bigoplus_{d\mcM | \mcN} f|_d \cdot \cc. 
    \end{align}
While the first two sums are orthogonal, in general, the innermost is not. However, the forms $f|_d$ corresponding to a given $f \in \mcB^*(\mcM, \kappa)$ are linearly independent. Thus, for each fixed $f \in \mcB^*(\mcM, \kappa)$, Gram-Schmidt orthogonalization of the set $\{f|_d : d \mcM \mid \mcN\}$ yields an orthonormal basis $\mcB(\mcM, f)$ of $\bigoplus_{d\mcM | \mcN} f|_d \cdot \cc$. Then, taking the union over all levels $\mcM \mid \mcN$ and $f \in \mcB^*(\mcM, \kappa)$ gives an orthonormal basis $\mcB(\mcN, \kappa)$ of $\mcA(\mcN, \kappa)$. Finally, taking the union over all spectral parameters $\kappa$ yields an orthonormal basis $\mcB(\mcN) \coloneqq \coprod_{\kappa} \mcB(\mcN, \kappa)$ of $\mcA(\mcN)$. We thus have the orthogonal decomposition
    \begin{align}
        \mcA(\mcN) \ = \ \bigoplus_{\mcM \mid \mcN} \bigoplus_{f \in \mcB^*(\mcM)} \mcB(\mcM, f),
    \end{align} 
where we now take all newforms $f \in \mcB^*(\mcM)$ of lower level $\mcM \mid \mcN$ without fixing $\kappa_f$. Building on the work of \cite{ILS}, who proved the following for squarefree levels and Proposition 5 of \cite{Rouymi}, Blomer and Milicevic in Lemma 2 of \cite{Blomer} proved that the set
    \begin{align}\label{eqn:f^gdef}
        \mcB(\mcM, f) \ = \ \Bigg\{f^{(g)} \ \coloneqq \ \sum_{d \mid g} \xi_g(d) f|_d \ \Bigg| \ g\mcM \mid \mcN \Bigg\},
    \end{align}
where $\xi_g(d)$ is defined as in \cite[(5.6)]{Blomer}, forms an orthonormal basis for $\bigoplus_{d \mid \mcL} f|_d \cdot \cc$. The forms $f^{(g)}$ have the same spectral parameter as the underlying newform $f \in \mcB^*(\mcM)$. Taking the union over forms $f \in \mcB^*(\mcM)$ and levels $\mcM \mid \mcN$ yields an orthonormal basis of $\mcA(\mcN)$ given by
    \begin{align}\label{eqn:Heckebasis}
        \mcB(\mcN) \ = \ \left\{f^{(g)} \ \Bigg| \ f \in \mcB^*(\mcM), \ \mcM \ | \ \mcN, \text{ and } g\mcM \mid \mcN\right\}.
    \end{align}
Following Section $3.1$ of \cite{BCL}, we call this basis the Hecke basis of level $\mcN$, and throughout the paper, we take our orthonormal basis of $\mcA(\mcN)$ to be $\mcB(\mcN)$.
    
The Hecke basis of level $\mcN$ consists of Hecke-Maass eigenforms, that is, the $f^{(g)}$ are real-valued joint eigenfunctions of both the Laplacian and all the Hecke operators for $(n, \mcN) = 1$. More precisely, let $0=\lambda_0\leq \lambda_1\leq \cdots$ denote the eigenvalues (counted with multiplicity) of the Laplacian acting on the space of cuspforms in $L^2(\Gamma_0(q)\backslash \hh)$. As above, we write $\lambda_j= 1/4+\kappa_j^2$ and we follow the convention of \cite{BCL}: we choose the sign of $\kappa_j$ such that $\kappa_j\geq 0$ if $\lambda_j\geq 1/4$ and such that $i\kappa_j>0$ otherwise. For each $\lambda_j>0$, we may choose an eigenvector $u_j$ such that $\{u_j\}_{j=1}^\infty$ forms an orthonormal basis of $\mcA(\mcN)$. Moreover, $\{u_j\}_{j=1}^\infty$ coincides with the Hecke basis $\mcB(\mcN)$ for $\mcA(\mcN)$ defined in \eqref{eqn:Heckebasis}. Hence, each element $u_j$ is of the form $f^{(g)}$ and both $f$ and $u_j = f^{(g)}$ have the same spectral parameter.

The $n$-th Fourier coefficient of $f|_d$ is nonzero only if $d \mid n$ (see \cite[Section 5]{Blomer}). Thus, we write
    \begin{align}
        \rho_{f^{(g)}}(n) \ = \ \sum_{d \mid g} \xi_g(d)\rho_{f}(n/d)
    \end{align}
with the convention that $\rho_f(x) = 0$ when $x \not\in \mathbb{Z}$. Since $g \mid \mcN$ in \eqref{eqn:f^gdef}, when $(n, \mcN) = 1$, only the $d = 1$ term is nonzero and thus
    \begin{align}
        \sqrt{n}\rho_{f^{(g)}}(n) \ = \ \xi_g(1)\sqrt{n}\rho_f(n) \ = \ \xi_g(1)\rho_f(1)\lambda_f(n) \ = \ \rho_{f^{(g)}}(1)\lambda_{f}(n).
    \end{align}
Therefore, when $(n, \mcN) = 1$ we have that
    \begin{align}\label{eqn:equaltonewformevals}
        \lambda_{f^{(g)}}(n) \ = \ \lambda_f(n).
    \end{align}
Moreover, by Lemma 2 of \cite{Blomer}, we have the following bound:
    \begin{align}\label{eqn:rhobound}
        \sqrt{n}\rho_{f^{(g)}}(n) \ \ll \ (n\mcN)^\e n^\theta(\mcN, n)^{1/2 - \theta}|\rho_f(1)| \ \ll \ \mcN^\e n^{1/2 +\e}|\rho_f(1)|.
    \end{align}
The Fourier coefficients of the forms in the Hecke basis $\mcB(\mcN)$ are quasi-multiplicative \cite[p.~74]{BHM}: if $u_j = f^{(g)}$ and $m = qm'$ with $(m', q)=1$, then
    \begin{align}\label{eqn:Coeffmult}
        \sqrt{m}\rho_j(m) \ = \ \sum_{d \mid (\mcN, q/(q, \mcN))} \mu(d)\chi_0(d) \lambda_f\left(\frac{q}{d(q, \mcN)}\right)\left(\frac{(\mcN, q)m'}{d}\right)^{1/2}\rho_j\left(\frac{(\mcN, q)m'}{d}\right).
    \end{align}
In particular, when $(q, \mcN) = 1$, by \cite[(5.3)]{Blomer}, we have that
    \begin{align}\label{eqn:FourierCoeffMult}
        \sqrt{m}\rho_j(m) \ = \ \lambda_f(q) \sqrt{m'} \rho_j(m').
    \end{align}
This gives us a way to split $\rho_j(m)$ into the part coprime to $\mcN$ and the part dividing $\mcN$, where the part coprime to $\mcN$ becomes a Hecke eigenvalue of the underlying newform so that we can use Hecke multiplicativity and Lemma \ref{lemma:mostbeautifullemmaintheworld}.
\begin{remark}\label{rem:ALforholforms}
     Analogous results hold for holomorphic cuspforms with slight changes. Specifically, the same construction works for holomorphic cuspforms of weight $k$ and level $\mcN$ if we instead define $f|_d(z) \coloneqq d^{k/2}f(d z)$ (see \cite[p.~15]{Blomer}). As in \cite{BCL}, we denote the Hecke basis of level $\mcN$ we thus obtain by $B_k(\mcN)$. Additionally, for holomorphic cuspforms, we may take $\theta = 0$ in \eqref{eqn:rhobound}. For the forms $f$ in $\mcB_k(\mcN)$, \eqref{eqn:Coeffmult} and \eqref{eqn:FourierCoeffMult} hold with $\rho_j(m)$ replaced by $\psi_f(m)$.
\end{remark}

\subsection{Explicit formula and several consequences of GRH}
We state several consequences of GRH which are necessary in later sections. 

\begin{lemma}
\label{ten min int}
    Assume the Riemann Hypothesis. Suppose $\Phi$ is an even Schwartz function, then
    \begin{align}
        \sum_{p\nmid q} \frac{\log^2(p)}{p \log^2(q)}\widehat\Phi\left(\frac{\log p}{\log q}\right)\ =\ \frac12 \int_{-\infty}^\infty |u|\widehat\Phi(u)du + O\left(\frac{\log\log q}{\log q}\right).
    \end{align}
\end{lemma}
\begin{proof}
This follow from a calculation similar to \cite[(3.5)]{jackmiller}.
\end{proof}

We require an explicit formula relating sums over zeros to sums over primes.
\begin{lemma}[Lemma 2.5 of \cite{BCL}]\label{lemma:their2.5}
Let $\Phi$ be an even Schwartz function whose Fourier transform is compactly supported. We have
\begin{align}
\sum_{\gamma_f}\Phi\left(\frac{\gamma_f}{2\pi}\log q\right)\ =\ -\frac{1}{\log q}\sum_{m=1}^\infty \frac{\Lambda(m)[c_f(m)+c_{\ol f}(m)]}{\sqrt{m}}\widehat\Phi\left(\frac{\log m}{\log q}\right)+ \widehat\Phi(0) + O\left(\frac{1}{\log q}\right),
\end{align}
where $\Lambda$ is the von Mangoldt function and
\begin{align}
    c_f(n) \ = \ \begin{cases}
        \alpha_f^{\ell}(p) + \beta_f^{\ell}(p) &\textrm{ if } n = p^{\ell} \\
        0 &\textrm{ otherwise,}\label{cdefn}
    \end{cases}
\end{align}
for $\alpha_f(p)$ and $\beta_f(p)$ defined as in \eqref{eqn:Lfunc}.
\end{lemma}
The following lemma is a minor generalization of \cite[(4.23)]{ILS}.
\begin{lemma}\label{lemma:boundingpsq}
    Let $f$ be a primitive holomorphic cuspidal newform of weight $k$ and level $d \mid q$ with trivial nebentypus. Assume GRH for $L(s,\mathrm{sym}^2 f)$ and let $F(t)$ be a smooth compactly supported function on $(-\sigma,\sigma)$. Then
        \begin{align}
            \mcS\ \coloneqq \ \frac{1}{(\log q)^2}\sum_{(p,q) = 1} \frac{(\log p)^2 \lambda_f(p^2)}{p} F\left(\frac{\log p}{\log q}\right) \ \ll \ \frac{\log \log q}{\log q}.
        \end{align}
\end{lemma}
\begin{proof}
     Define $g(t\log q) \coloneqq t F(t)$ so that
    \begin{align}
        \mcS\ =\ \frac{1}{\log q} \sum_{n} \frac{\Lambda(n) \lambda_{f}(n^2)}{n} g\left( \log n \right) + O\left(\frac{1}{\log q}\right).
    \end{align}
    Fourier inversion of $g$ and the ideas of Proposition 7 of \cite{Tao} yield
        \begin{align}
            \mcS + O\left(\frac{1}{\log q}\right)\ &=\ \frac{1}{\log q} \int_{-\infty}^\infty \sum_{n} \frac{\Lambda(n)\lambda_{f}(n^2)}{n^{1 + 1/\log q + it}} \widehat g\left(t-\frac{i}{\log q}\right)\, dt.
        \end{align}
    The order of vanishing of $L(s, \mathrm{sym}^2 f)$ at $s = 1$ is $0$ since Shimura \cite{S1} proved that $L(s,\mathrm{sym}^2f)$ is entire and $L(1, \mathrm{sym}^2f) \neq 0$, so by Theorem 5.17 of \cite{IK},
        \begin{align}
            -\frac{L'}{L}\left(1 + \frac{1}{\log q} + it,\mathrm{sym}^2f\right) \ \ll \ \log\log (q + |t| + 2),
        \end{align}
    which completes the proof. 
\end{proof}
\noindent
Next, we generalize Lemma 2.7 of \cite{BCL} to sums over multiple primes.
\begin{lemma} \label{lemma:mostbeautifullemmaintheworld}
    Assume GRH for $L(s,\chi)$ with $\chi$ mod $q$ and for $L(s,f)$ where $f$ is a primitive holomorphic Hecke eigenform or a primitive Maass Hecke eigenform of weight $k$ and level $q$. For $1 \leq i \leq n$, let $\Psi_i$ be smooth functions compactly supported on $[0,X]$ with $X\geq 2$. Write $z_i = 1/2 + it_i$ with $t_i \in \rr$ and let $N\geq 1$ be an integer. Suppose that for each $i$ there is some $A_{i}\in \rr$ with
\begin{align}
    \left| \Psi_i^{(3)}(x) \right|\ \leq \ \frac{A_{i}}{x^3 \min\{\log(X+3), X/x\}}
\end{align} 
for all $x \in \rr^+.$ Then,
    \begin{align}
        &\mathop{\sum \cdots \sum}_{\substack{(p_i, N) = 1 \\ p_i \neq p_j}} \prod_{i=1}^n \frac{\lambda_f(p_i)\log p_i\Psi_i(p_i)}{p_i^{z_i}}
        \ \ll \ \prod_{i=1}^n\Bigg[A_{i} \log^{1+\e}(X) \log(q  + k+ |t_i|) + \log(NX)\sup|\Psi_i(t)|\Bigg].
    \end{align}
    If $\chi$ is a nonprincipal character, the above holds with $\lambda_f(p_i)$ replaced with $\chi(p_i)$ and $k=0.$
\end{lemma}
\begin{proof}
    We observe that
    \begin{align}
        &\mathop{\sum \cdots \sum}_{\substack{(p_i, N) \ =  \ 1 \\ p_i \neq p_j, \ 1 \leq i \neq j \leq n+1}} \prod_{i=1}^{n+1} \frac{\lambda_f(p_i)\log p_i}{p_i^{z_i}}\Psi_i\left(p_i\right) \nonumber\\
        &\hspace{1cm}= \ \left(\sum_{p_1\nmid N} \frac{\lambda_f(p_1)\log p_1}{p_1^{z_1}} \Psi_1\left(p_1\right)\right) \left(\mathop{\sum \cdots \sum}_{\substack{(p_i, N) \ =  \ 1 \\ p_i \neq p_j, \ 2 \leq i \neq j \leq n+1}} \prod_{i=2}^{n+1} \frac{\lambda_f(p_i)\log p_i}{p_i^{z_i}}\Psi_i\left(p_i\right)\right)\nonumber\\
        &\hspace{1.5cm} - n\sum_{p_1\nmid N} \frac{\lambda_f^2(p_1)\log^2 p_1}{p_1^{z_1+z_2}} \Psi_1(p_1)\Psi_2(p_1) \mathop{\sum \cdots \sum}_{\substack{(p_i, p_1N) \ =  \ 1 \\ p_i \neq p_j, \ 3 \leq i \neq j \leq n+1}} \prod_{i=3}^{n+1} \frac{\lambda_f(p_i)\log p_i}{p_i^{z_i}}\Psi_i\left(p_i\right).\label{notdoingxinonsense}
    \end{align}
    The claim now follows by strong induction. The non-principal character case is analogous.
\end{proof}
\begin{remark}\label{rem:evalsums}
    By the remark following Lemma 2.7 of \cite{BCL}, when we apply the previous lemma to smooth functions of the form $\Psi(x) = V(cx/X)$ for $V$ supported on $[0, c]$ or $\Psi(x) = \widehat{\Phi}\left(\frac{c\log x}{\log X}\right)$ for $\widehat\Phi$ compactly supported on $(-\infty, c]$, we may take $A_i \ll 1$.
\end{remark}

\section{Two-level Density and Centered Moments}\label{sec3}
In this section, we prove Theorem \ref{main thm}. We will reduce the computation of the two-level density and the $n$-th centered moment of the one-level density to a sum over $n$ distinct primes not dividing the level. This entire section proceeds by induction on $n$.
Define 
\begin{align}
    \mathscr{P}(f,\Phi) \ \coloneqq \ \frac{1}{\log q} \sum_{p\nmid q} \lambda_f(p) \frac{\log p}{\sqrt p} \hat\Phi\left(\frac{\log p}{\log q}\right)
\end{align}
and
\begin{align}
    R_q(f; \Phi_i)\ \coloneqq \  \mathscr{OD}(f; \Phi_i)-\widehat \Phi_i(0)-\frac{1}{2}\Phi_i(0) +2 \mathscr{P}(f,\Phi_i).
\end{align}

\subsection{Bounding error terms}\label{subsec: removing errors}
We follow the idea of \cite{rubinstein} to show that error terms can be bounded by using a non-negative test function $H$ with sufficiently nice properties. Specifically, choose $\rho$ such that $0<\rho<\frac{1}{n}$. By the proof of Lemma 2 of \cite{rubinstein}, there exists some real, even, nonnegative Schwartz function $H$ whose Fourier transform is supported in $(-\rho,\rho)$ so that $H(x)\geq\max_i|\Phi_i(x)|.$ Note that $H$ is dependently only on $n$ and the $\Phi_i.$
\begin{lemma}\label{lemma: bounding abs curly P}
    For each $1\leq i\leq n$, we have $|R_q(f; \Phi_i)| = O\left( \frac{\log \log q}{\log q}\right)$ uniform in $f\in \mcH_k(q)$. Further, if Theorem \ref{onylallpairedoralldifferent} holds for every $j\leq n-1$, then
    for any $T \subsetneq \{1,2,\ldots,n\}$,
    \begin{align}
        \mathbb{E}_Q\left[\prod_{i\in T}|\mathscr{P}(f;\Phi_i)| \right]\ = \ O_T(1). \label{uniform bound on abs curly p}
    \end{align}
\end{lemma}
Note we prove Theorem \ref{onylallpairedoralldifferent} by induction on $n$, while here we only use it on $j\leq n-1$.
\begin{proof}
By Lemma \ref{lemma:their2.5}, Lemma 4.1 of \cite{BCL}, GRH for $L(s,\mathrm{sym}^2f)$ and \cite[(4.23)]{ILS}, Theorem 5.17 of \cite{IK}, and noting that $c_f = \ol{c_f} = c_{\ol f}$, we conclude $R_q$ is uniformly $O\left(\frac{\log \log q}{\log q}\right)$ for $f\in \mcH_k(q)$.

For each $1\leq i \leq n,$ we have
\begin{align}
    |\mathscr{OD}(f; \Phi_i)|\ \leq\ \sum_j \left| \Phi_i\left(\frac{\gamma_f(j)\log q}{2\pi} \right) \right|\ \leq\ \sum_j H\left( \frac{\gamma_f(j)\log q}{2\pi}\right)\ =\ \mathscr{OD}(f; H).
\end{align}
Using this and the exact explicit formulas, we have
\begin{align}
    2|\mathscr{P}(f,\Phi_i)|\ 
    &\leq\ |\widehat \Phi_i(0)+\frac{1}{2}\Phi_i(0)|+ |R_q(f; \Phi_i)|+ \mathscr{OD}(f; H) \nonumber \\
    &=\ |\widehat \Phi_i(0)+\frac{1}{2}\Phi_i(0)|+ |R_q(f; \Phi_i)|+\widehat H(0)+\frac{1}{2}H(0)+R_q(f;H)-2\mathscr{P}(f, H) \nonumber \\
    &\leq\ |\widehat \Phi_i(0)+\frac{1}{2}\Phi_i(0)|+ |R_q(f; \Phi_i)|+\widehat H(0)+\frac{1}{2}H(0)+|R_q(f;H)|-2\mathscr{P}(f, H).
    \label{eqn:expansion}
\end{align}
Note that $\widehat H(0)+\frac{1}{2}H(0)\geq 0$ because $H\geq 0$. Further, for sufficiently large $q$, we have $|R_q(f; \Phi_i)| \leq 1$ for all $f\in \mcH_k(q)$ and $|R_q(f;H)|\leq 1$. Define the fixed constant
\begin{align}
    K_i\ \coloneqq \ \frac{|\widehat \Phi_i(0)+\frac{1}{2}\Phi_i(0)|+\widehat H(0)+\frac{1}{2}H(0)+2}{2}, 
\end{align}
so that \eqref{eqn:expansion} gives the pointwise bound $|\mathscr{P}(f,\Phi_i)| \leq K_i-\mathscr{P}(f, H).$ 
In particular, the right-hand side is non-negative. 
Positivity of $\mathbb{E}_Q$ gives
\begin{align}
    \mathbb{E}_Q\left[\prod_{i\in T}|\mathscr{P}(f;\Phi_i)| \right]\ &\leq\ \mathbb{E}_Q\left[\prod_{i\in T}(K_i-\mathscr{P}(f, H)) \right] \ =\ \sum_{S\subset T}(-1)^{|S|}\mathbb{E}_Q\left[\mathscr{P}(f, H)^{|S|}\right] \prod_{i\in T\setminus S}K_i.
\end{align}

For every $2\leq j\leq n-1$, the tuple consisting of $j$ copies of $H$ satisfies the support condition of Theorem \ref{onylallpairedoralldifferent}; namely, $(j-1)\rho<1<\frac{3}{2}$, and, for every partition with $a+b=j$, we have $(a+3b)\rho \leq 3j\rho<3<4$ and $(3a+b)\rho<3j\rho<3<4$. 
Thus, the right side is $O(1)$.
\end{proof}
Note the $H$ being nonnegative is crucial as it allows us to extract the error terms from the average and consequently reduce to a lower case.
Now, by Lemma 4.1 of \cite{BCL}, using the linearity of $\mathbb{E}_Q$ to expand, and using Lemma \ref{lemma: bounding abs curly P} to bound any term containing an $R_q(f,\Phi_i)$, we conclude
\begin{align}
    \mathscr{M}(\Phi_1,\ldots, \Phi_n)\ &=\ \lim_{Q\rightarrow \infty}\mathbb{E}_Q\left[\prod_{i=1}^n\left( \mathscr{OD}(f; \Phi_i)- \widehat \Phi_i(0)-\frac{1}{2}\Phi_i(0) \right)\right]\nonumber \\
    &=\ \lim_{Q\rightarrow \infty}\mathbb{E}_Q\left[\prod_{i=1}^n\left( -2\mathscr{P}(f,\Phi_i)+R_q(f; \Phi_i) \right)\right] \ = \ (-2)^n \Bigg\langle \prod_{i=1}^n\mathscr{P}(f; \Phi_i)\Bigg\rangle_*. \label{eqn3.8}
\end{align}
The following theorem now immediately implies Theorem \ref{main thm}.

\begin{thm}\label{onylallpairedoralldifferent}
    Assume RH and GRH for $L(s, f)$, $L(s, \mathrm{sym}^2f)$, and Dirichlet $L$-functions. Let $\Psi$ be a nonnegative, smooth function compactly supported on $(0, \infty)$ for which $\widehat\Psi(0)> 0$. For $1 \leq i \leq n$, let $\Phi_i$ be an even Schwartz function whose Fourier transform is compactly supported in $(-\sigma_i, \sigma_i)$, where the $\sigma_i$ satisfy Conditions \eqref{cond:smallsigma} and \eqref{cond:bipart}. Then we have that
        \begin{align} 
            \Bigg\langle\prod_{i=1}^n\mathscr{P}(f; \Phi_i)\Bigg\rangle_* \ 
            &=\ \frac{\mathbf{1}_{2|n}}{2^n(n/2)!} \sum_{\tau\in S_n} \prod_{i=1}^{n/2} \int_{-\infty}^\infty |u|\widehat\Phi_{\tau(2i-1)}(u)\widehat\Phi_{\tau(2i)}(u)\,du.
        \end{align}
\end{thm}
To compute the two-level density, by a similar calculation to that in \cite[(4.5)]{MillerEC}, we have
\begin{align}
    \sum_{j_1\neq \pm j_2} \prod_{i=1}^2 \Phi_i\Bigg(\frac{\gamma_f(j_i)}{2\pi}\log q\Bigg) \ =\ \prod_{i=1}^2 \mathscr{OD}(f; \Phi_i) - 2\mathscr{OD}(f; \Phi_1\Phi_2) + \frac{1 - \epsilon_f}{2}\Phi_1(0)\Phi_2(0).
\end{align}
By averaging each term over $f \in \mcH_k(q)$ and $q \asymp Q$ and 
using Lemma \ref{lemma: bounding abs curly P}, Corollary \ref{cor:twolevel} now follows from Theorem \ref{onylallpairedoralldifferent}.


\subsection{Proof of Theorem \texorpdfstring{\ref{onylallpairedoralldifferent}}{}}\label{subsec3.1}
\noindent 
This subsection is dedicated to the proof of Theorem \ref{onylallpairedoralldifferent}. 
For ease of notation, apart from Lemma \ref{lemmaforcase1ofcombo}, we assume that all the functions $\Phi_i$ are equal to $\Phi$. Unlike related papers, e.g., \cite{10yearsmall, Hughes_2007}, we must reduce to the case of distinct primes to perform the manipulations needed to apply the Kuznetsov trace formula and to handle the resulting terms.

By splitting the sum into sums over powers of distinct primes, it suffices to bound each
\begin{align}
    \mcS\ =\ \nonsense{0} \hkq \mathop{\sum\cdots\sum}_{\substack{p_1,\ldots, p_\ell \nmid q\\ p_i\neq p_j}} \prod_{i=1}^\ell \mcF(p_i, a_i),
\end{align}
where $a_1 \geq a_2\geq \cdots \geq a_\ell > 0$ is a partition of $n$ and
\begin{align}
    \mcF(p_i, a_i)\ \coloneqq\ \left(\frac{\lambda_f(p_i) \log p_i}{\sqrt{p_i}\log q}\widehat \Phi\left(\frac{\log p_i}{\log q}\right)\right)^{a_i}.
\end{align} 
We split into five disjoint cases based on the partition of $n$ as delineated below.
\setlist[enumerate, 1] 
{1., 
leftmargin  = 2em,
itemindent  = 0pt,
labelwidth  = 2em,
labelsep    = 0pt,
font        = \bfseries,
align       = left,
itemsep     = 1.5mm,
ref         = \mbox{\textbf{\arabic*}}
}
\begin{enumerate}
    \setlength\itemsep{0em}
    \item Each part is $2.$ This term contributes, so we compute it in Lemma \ref{lemmaforcase1ofcombo}. \label{case1ofcombo}
    \item All parts are $\geq 2$, and at least one part is $>2$.\label{case2ofcombo}
    \item There is exactly one $1$ in the partition (and $n\geq 2$).\label{case3ofcombo}
    \item There are at least two $1$'s in the partition, but not all $1$'s.\label{case4ofcombo}
    \item Each part is $1.$
    \label{case5ofcombo}
\end{enumerate}
We proceed by induction on $n$ with base cases $n=1,2,3$. We first prove Cases \ref{case1ofcombo} and \ref{case2ofcombo}. We prove Case \ref{case5ofcombo} in Proposition \ref{prop42}. Note that all three of these cases are proven independently of $n$. We again use induction to prove Cases \ref{case3ofcombo} and \ref{case4ofcombo}. Case \ref{case3ofcombo} and Case \ref{case4ofcombo} are vacuously true for $n=1,2$ and the base case $n = 3$ is proven in Appendix \ref{subsec:base}. 

Before proving Cases \ref{case1ofcombo}, \ref{case2ofcombo}, \ref{case3ofcombo}, and \ref{case4ofcombo}, we prove several auxiliary lemmas.

\begin{lemma}\label{bounding a>2}
    If $a_i \geq 3$, then
    \begin{align}
        \sum_{p_i \nmid N} \mcF(p_i, a_i)\ \ll\ \frac{1}{\log^{a_i}(q)},
    \end{align}
    where the bound is independent of $f$ and $N \mid qp_1p_2\cdots p_n.$
\end{lemma}
\begin{proof}
    We use that $|\lambda_f(p_i)| \leq \tau(p_i) = 2$ and that the sum over $p_i$ converges absolutely. 
\end{proof}

\begin{lemma}\label{bounding a=2}
    For any $N \mid qp_1p_2\cdots p_n,$ we have that
    \begin{align}
        \sum_{p_i \nmid N} \mcF(p_i, 2)\ =\ \frac12 \int_{-\infty}^\infty |u|\widehat\Phi(u)^2du + O\left(\frac{\log\log q}{\log q}\right),
    \end{align}
    with error term bounded uniformly in $f$.
\end{lemma}
\begin{proof}
    We write $\lambda_f(p_i)^2 = \lambda_f(p_i^2)+1$ and use Lemma \ref{lemma:boundingpsq} with $F = \widehat\Phi^2$ to deduce that the $\lambda_f(p_i^2)$ term is $\ll \log\log q(\log q)^{-1}$ and Lemma \ref{ten min int} to evaluate the contribution of the $1$ term.
\end{proof}

\begin{lemma}\label{lemmaforcase1ofcombo}
    All of the sums corresponding to partitions in Case \ref{case1ofcombo} contribute
    \begin{align}
        \frac{1}{2^n(n/2)!} \sum_{\tau\in S_n} \prod_{i=1}^{n/2} \int_{-\infty}^\infty |u|\widehat\Phi_{\tau(2i-1)}(u)\widehat\Phi_{\tau(2i)}(u)\,du.
    \end{align}
\end{lemma}
\begin{proof}
    For each partition of the set $\{1,2,\ldots,n\}$ into $n/2$ disjoint sets of size $2$, say $\{\{a_1,b_1\}, \{a_2,b_2\},\\ \ldots, \{a_{n/2}, b_{n/2}\}\},$ we have a sum
    \begin{align}
        \mcS \ = \ \nonsense{0}\hkq \sum_{\substack{p_{a_i} = p_{b_i} \nmid q\\ p_{a_i}\neq p_{a_j}}} \prod_{i=1}^{n/2} \mcF(p_{a_i}, 2).
    \end{align}
    By applying Lemma \ref{bounding a=2} to each $\mcF,$ we conclude that
    \begin{align}
        \mcS \ = \ \nonsense{0}\hkq \prod_{i=1}^{n/2}\left[\frac12 \int_{-\infty}^\infty |u|\widehat\Phi_{a_i}(u)\widehat\Phi_{b_i}(u)\,du + O\left(\frac{\log\log q}{\log q}\right)\right].
    \end{align}
    Expanding this out, any term with at least one $O(\log\log q/\log q)$ term is $\ll \log\log Q/\log Q$ as $Q\to \infty.$ Hence,
    \begin{align}
        \lim_{Q \to \infty} \mcS \ = \ \prod_{i=1}^{n/2}\frac12 \int_{-\infty}^\infty |u|\widehat\Phi_{a_i}(u)\widehat\Phi_{b_i}(u)\, du.
    \end{align}
    Now, for each $\tau \in S_n$, from the partition $\{\{1, 2\},\{2,3\}, \ldots, \{n-1,n\}\}$ we obtain another partition by letting $\tau$ act on this partition via $\{ \{ \tau(2i-1), \tau(2i)\}: 1\leq i\leq n/2\}.$ The stabilizer of this group action has order $2^{n/2}(n/2)!$ and this completes the proof.
\end{proof}

\begin{lemma}\label{lemmaforcase2ofcombo}
    The sums corresponding to partitions in Case \ref{case2ofcombo} are $O\left((\log Q)^{-3}\right).$
\end{lemma}
\begin{proof}
    This follows from taking absolute values and using Lemmas \ref{bounding a>2} and \ref{bounding a=2}.
\end{proof}

\begin{lemma}\label{lemmaforcase3ofcombo}
    The sums corresponding to partitions in Case \ref{case3ofcombo} are $O\left((\log Q)^{\e-1}\right).$
\end{lemma}

\begin{lemma}\label{lemmaforcase4ofcombo}
    The sums corresponding to partitions in Case \ref{case4ofcombo} are $O\left((\log Q)^{\e-1}\right).$
\end{lemma}

The proof of Lemmas \ref{lemmaforcase3ofcombo} and \ref{lemmaforcase4ofcombo} can be found in Appendix \ref{app:combo}. All that remains to be proved to finish the induction is Case \ref{case5ofcombo}, i.e., where $p_1,\ldots,p_n$ are distinct and coprime to $q$. This comprises the remainder of the paper.

\begin{remark}
    In comparison to other cases, Case \ref{case5ofcombo} is by far the most difficult, and we use the Kuznetsov trace formula to handle this case. 
\end{remark}

\section{Converting Sums over Distinct Primes to Spectral Terms}\label{section4}
In this section, we prove the following proposition to show that the sum over distinct primes contributes negligibly as $Q \rightarrow \infty$. We follow Section 6 of \cite{BCL}: we apply the Petersson trace formula, truncate in $L_1L_2,$ and then again in $e,$ add back in primes that divide the level, and perform several substitutions to apply the Kuznetsov trace formula.

\begin{prop}\label{prop42}
Fix $n>1$ and assume GRH for $L(s, f)$ and for Dirichlet $L$-functions. Let $\Phi_i$ be even Schwartz functions with $\widehat\Phi_i$ compactly supported in $(-\sigma_i,\sigma_i)$, where the $\sigma_i$'s satisfy Conditions \eqref{cond:smallsigma} and \eqref{cond:bipart}.
Then
\begin{align} 
    \Sigma_1\ \coloneqq \ \nonsense{n} \hkq \sum_{\substack{p_1,\ldots,p_n\nmid q\\ p_i\neq p_j}} \prod_{i=1}^n \phinonsenselambda{i} \ \ll\ \frac{1}{\log Q}.
\end{align}
\end{prop}

We use Hecke multiplicativity to combine the $\lambda_f(p_i)$ since $(p_i, p_j) = (p_i,q) = 1$ for $i\neq j$. Then, we apply Lemma \ref{lemma:lemma2.3ofbcl} to $\Delta_d^*(1,\prod p_i)$ and use Hecke multiplicativity once more to obtain
\begin{align}
    \Sigma_1\ &=\ \nonsense{n} \sum_{\substack{q = L_1L_2d\\L_1\mid q_1\\L_2\mid q_2}} \frac{\mu(L_1L_2)}{L_1L_2} \prod_{\substack{p\mid L_1\\ p^2\nmid d}} \Bigg(1-\frac{1}{p^2}\Bigg)^{-1}\sum_{e\mid L_2^\infty} \frac{1}{e}  \nonumber \\
    &\hspace{.5cm} \times \bkd \lambda_f(e^2) \sum_{\substack{p_1,\ldots,p_n\nmid q\\ p_i\neq p_j}}  \prod_{i=1}^n\phinonsenselambda{i}.
\end{align}

\subsection{Truncation in \texorpdfstring{$L_1L_2$}{} and \texorpdfstring{$e$}{}}
We next bound the contribution from $L_1L_2 \geq \mcL_0$ and $e\geq E$ for
\begin{equation}
    \mcL_0\ =\ (\log Q)^{n+5} \hspace{1cm} \text{ and } \hspace{1cm} E \ =\  (\log Q)^{n+2}.
\end{equation}
We first truncate in $L_1L_2$ and then show that for $L_1L_2 \leq \mcL_0$, the sum over $e > E$ contributes negligibly as $Q \rightarrow \infty$. 

By Lemma \ref{lemma:mostbeautifullemmaintheworld} with $N=q$, each $z_i = 1/2$, and $X=q^{\max \sigma_i}$, we obtain 
\begin{align}
    \sumoverpdistinct \prod_{i=1}^n \phinonsenselambda{i} \ \ll \ (\log Q)^{2n+\e}.
\end{align}
We then proceed as in \cite[p.~20]{BCL}, mutatis mutandis.

\subsection{Filling in primes}\label{section:filling in primes}
We now remove the dependence of the prime sums on $q$ in order to apply the Kuznetsov trace formula. To do this, we add back the sums over distinct primes dividing $q$ and show that this changes the value of $\Sigma_1$ by a negligible amount for test functions whose Fourier transform has sufficiently small support. For $\ell \geq 1$, it suffices to bound
    \begin{align} 
        \mathcal{E}\ &\coloneqq\ \nonsensen \sum_{\substack{p_1,\ldots, p_{\ell}\mid q\\ p_i\neq p_j}}\ \sum_{\substack{p_{\ell+1},\ldots, p_{n}\nmid q\\ p_i\neq p_j}} \prod_{i=1}^n \phinonsense{i} \sum_{\substack{q = L_1L_2d \\ L_1 \mid q_1 \\ L_2 \mid q_2 \\ L_1L_2 \leq \mcL_0}} \frac{\mu(L_1L_2)}{L_1L_2} \nonumber \\
        &\hspace{+0.5cm} \times\prod_{\substack{p | L_1 \\ p^2 \nmid d}}\left(1-\frac{1}{p^2}\right)^{-1} \sum_{\substack{e \mid L_2^\infty \\ e<E}}\frac{1}{e} \sum_{f \in B_k(d)}\mkern -17mu ^h \lambda_f(e^2)\lambda_f\left(\prod p_i\right).
    \end{align}
From $e \mid L_2^\infty$ and $(L_2, d) = 1$, we conclude $(e^2, \prod p_i, d) = 1$. Thus, by Lemma \ref{lemma:their2.2}, 
\begin{align}
    \Delta_d\left(e^2,\prod p_i\right)\ \ll\ \frac{\left(e^2d \prod p_i\right)^\e}{d^{3/2}}\sqrt{e^2\prod p_i} \label{peterssonbound}.
\end{align}
Taking absolute values, using \eqref{peterssonbound}, and partial summing, we obtain
\begin{align}
    \mcE\ &\ll\ Q^{\sum_{i=\ell +1}^n (\sigma_i-\delta_i)-3/2 + \e}.
\end{align}
Since $\ell \geq 1$, by Condition \eqref{cond:smallsigma} we have $\Sigma_1 = \Sigma_2 + O(Q^{-\e}) + O\left((\log Q)^{-1}\right)$,
where
\begin{align}
    \Sigma_2\ &\coloneqq\ \nonsense{n} \sum_{\substack{p_1, \ldots, p_n\\ p_i\neq p_j }} \prod_{i=1}^n \phinonsense{i} \sum_{\substack{q = L_1L_2d \\ L_1 \mid q_1 \\ L_2 \mid q_2 \\ L_1L_2 \leq \mcL_0}} \frac{\mu(L_1L_2)}{L_1L_2}\nonumber \\
    &\hspace{+0.5cm} \times \prod_{\substack{p | L_1 \\ p^2 \nmid d}}\left(1-\frac{1}{p^2}\right)^{-1}\sum_{\substack{e \mid L_2^\infty \\ e<E}}\frac{1}{e}\sum_{f \in B_k(d)}\mkern-17mu^h  \lambda_f(e^2)\lambda_f\left(\prod p_i \right).
\end{align}
\begin{remark}
    Since $f \in B_k(d)$ and we are no longer guaranteed $(p_i, d) = 1$, we cannot use Hecke multiplicativity. Splitting the eigenvalues before dropping this condition is not viable since we seek to apply the Petersson trace formula. With this in mind, there are two natural methods for removing this restriction on the supports, which we only require to add back the primes dividing $q$. 

    The first is to improve bounds on Kloosterman-Bessel sums (see Corollary 2.2 of \cite{ILS}). However, analysis of these terms is difficult and notably, if one can achieve support $a/(n-1)$ for this step, then one can similarly attain support $a/n$ for the main result.
    
    The second is to prove a stronger variant of Lemma \ref{lemma:mostbeautifullemmaintheworld} using an improved bound for $L'/L(s, f)$ on the line $\Re(s) = 1/2 + 1/\log q$ (e.g., see Theorem 5.17 of \cite{IK} for a classical bound). Improving the bound to $\ll \log(q+|s|)^{1+\e}$ would give arbitrary support for filling in primes. However, obtaining better bounds requires control over the order of vanishing at $s=1/2$ and the vertical distribution of zeros near the central point---precisely the question we wish to study.
\end{remark}

\subsection{Applying the Kuznetsov trace formula} \label{sec:Kuznetsov}
The goal of this subsection is to apply the Kuznetsov trace formula to $\Sigma_2$ to change the level of our forms. By the Petersson trace formula (Lemma \ref{lemma:Peterssontraceformula}),
\begin{align}
    \Sigma_2\ &=\ \nonsensen \sumprimesdistinct{1}{n} \prod_{i=1}^n \phinonsense{i} \sum_{\substack{q = L_1L_2d \\ L_1 \mid q_1 \\ L_2 \mid q_2 \\ L_1L_2 \leq \mcL_0}} \frac{\mu(L_1L_2)}{L_1L_2}  \nonumber \\
    &\hspace{+0.5cm} \times \prod_{\substack{p \mid L_1 \\ p^2 \nmid d}}\left(1-\frac{1}{p^2}\right)^{-1} \esum \frac{2\pi  i^{-k}}{e} \sum_{c \geq 1} \frac{S(e^2, \prod p_i; cd)}{cd}J_{k-1}\left(\frac{4\pi\sqrt{e^2\prod p_i}}{cd}\right).
\end{align}
Performing the substitutions in \cite[p.~22-23]{BCL} where we use a sum over $s$ rather than $n$, yields
    \begin{align}\label{eqnforKF}
        \Sigma_2\ &=\ \frac{2\pi i^{-k}}{N(Q)} \sum_{\substack{L_1,L_2 \\ (L_1,L_2) = 1 \\ L_1L_2 \leq \mcL_0}}\frac{\mu(L_1L_2)}{L_1L_2}\prod_{p \mid L_1} \left(1-\frac{1}{p^2}\right)^{-1}\sum_{r \mid L_1} \frac{\mu(r)}{r^2}\sum_{d \mid L_2} \mu(d) \sumprimesdistinct{1}{n} \prod_{i=1}^n \frac{\log p_i}{\sqrt{p_i}}\esum\frac{1}{e}\nonumber \\
        &\hspace{0.5cm} \times \sum_{c \geq 1}\sum_{s\geq 1}\frac{S(e^2, \prod p_i; cL_1rds)}{cL_1rds}\Psi\left(\frac{L_1^2L_2rds}{Q}\right)&\nonumber \\
        &\hspace{+0.5cm}\times\frac{1}{\log^n\left(L_1^2L_2rds\right)}\prod_{i=1}^n\widehat\Phi_i\left(\frac{\log p_i}{\log \left(L_1^2L_2rds\right)}\right)J_{k-1}\left(\frac{4\pi\sqrt{e^2\prod p_i}}{cL_1rds}\right).
    \end{align}
We put $\mfm = cL_1rds$ so that our sum over $s$ becomes
    \begin{align}
        &\sum_{\substack{\mfm \geq 1 \\ \mfm \equiv 0 \ \mathrm{mod} \ cL_1rd}}\frac{S(e^2, \prod p_i; \mfm)}{\mfm\log^n\left(L_1^2L_2rds\right)} \Psi\left(\frac{L_1L_2\mfm}{cQ}\right)\prod_{i=1}^n \widehat{\Phi}_i\left(\frac{\log p_i}{\log \left(L_1L_2\mfm/c\right)}\right)J_{k-1}\left(\frac{4\pi \sqrt{e^2\prod p_i}}{\mfm}\right) \nonumber\\
        &=\ \sum_{\substack{\mfm \geq 1 \\ \mfm \equiv 0 \ \mathrm{mod} \ cL_1rd}} \frac{S(e^2, \prod p_i; \mfm)}{\mfm}f\left(4 \pi\frac{\sqrt{\prod p_i e^2}}{\mfm}\right),
    \end{align}
where
    \begin{align} \label{fdef}
        f(\xi)\ \coloneqq \ \frac{\Psi\left(4\pi L_1L_2\sqrt{e^2\prod p_i}(cQ\xi)^{-1}\right)}{\log^n\left(4\pi L_1L_2\sqrt{e^2\prod p_i}(c\xi)^{-1}\right)}\prod_{i=1}^n \widehat\Phi_i\left(\frac{\log p_i}{\log \left(4\pi L_1L_2\sqrt{e^2\prod p_i}(c\xi)^{-1}\right)}\right)J_{k-1}(\xi).
    \end{align}
We now introduce $n$ smooth partitions of unity for the sum over primes. This tells us that
\begin{align}
\sumprimesdistinct{1}{n}\cdots\ =\  \sum_{P_1}\mkern-1mu ^d\cdots \sum_{P_n}\mkern-1mu ^d \sumprimesdistinct{1}{n} \prod_{i=1}^n V\left(\frac{p_i}{P_i}\right) \cdots,
\end{align}
where $\sum_{P_i}^d$ denotes a sum over $P_i = 2^j$ for $j \geq 0$, and $V$ is a smooth function supported on $[1/2, 3]$ with $\sum_{P = 2^j} V\left( x/P \right) = 1$
for all real $x \geq 2$. We separate the dependence of $f(\xi)$ on the $p_i$ and put
    \begin{align}\label{eqn:Xdef}
        X\ \coloneqq\  \frac{4 \pi L_1L_2 \sqrt{e^2 \prod P_i}}{cQ}
    \end{align}
and
    \begin{align} \label{Heq}
        H\left(\xi, \lambda_1,\ldots, \lambda_n\right)\ \coloneqq \ \Psi\left(\frac{X}{\xi}\sqrt{\prod \lambda_i}\right) \frac{(\log Q)^n}{ \log ^n\left(\frac{X}{\xi}\sqrt{\prod \lambda_i}Q\right)}\prod_{i=1}^n\widehat\Phi_i\left(\frac{\log \left(\lambda_iP_i\right)}{\log\left(\frac{X}{\xi}\sqrt{\prod \lambda_i}Q\right)}\right).
    \end{align}
Following the proof of Lemma 6.1 of \cite{BCL} mutatis mutandis, we obtain the following bound.
\begin{lemma}\label{lemma:bound on H}
    Let $\widehat{H}$ be the usual Fourier transform of $H$ defined as above. For any $A>0,$
    \begin{align}
        \widehat{H}(u, v_1,\ldots, v_n)\ \ll_A\ \left(\frac{1}{(1 + |u|)\prod (1+|v_i|)}\right)^A.
    \end{align}
\end{lemma}
Since $\Psi$ is compactly supported, say in $(a_1,b_1)$, and $1/2 \leq p_i/P_i \leq 3$ for all $1 \leq i \leq n$, we have that $H\left(\xi, \frac{p_1}{P_1},\ldots, \frac{p_n}{P_n}\right)$ is compactly supported in $\xi$. Thus, from \eqref{fdef} and \eqref{Heq} we write
    \begin{align}
        f\left(\xi\right)\ &=\ \frac{1}{(\log Q)^n}H\left(\xi, \frac{p_1}{P_1},\ldots, \frac{p_n}{P_n}\right)J_{k-1}\left(\xi\right)W\left(\frac{\xi}{X}\right),
    \end{align}
where $W$ is smooth, compactly supported, and satisfies $W(\alpha) = 1$ whenever $\alpha \in [2^{-n}/b_1, 3^n/a_1]$. Fourier inversion yields
    \begin{align}
        f(\xi)\ =\ \frac{J_{k-1}(\xi)}{(\log Q)^n} W\left(\frac{\xi}{X}\right) \int_{-\infty}^\infty\cdots \int_{-\infty}^\infty\widehat{H} (u, v_1,\ldots, v_n)\mathrm{e}\left(u\xi + \sum_{i=1}^n v_i\frac{p_i}{P_i}\right)\,du\,dv_1\cdots\,dv_n.
    \end{align}
Let $h_u(\xi) = J_{k-1}(\xi) W\left(\xi X^{-1}\right) \mathrm{e} \left(u\xi\right)$ and let $x = \prod p_i$. Then
    \begin{align}
        \Sigma_2\ &=\ \frac{2\pi i^{-k}}{N(Q)} \sum_{\substack{L_1, L_2 \\ (L_1, L_2) = 1 \\ L_1L_2 < \mathcal{L}_0}} \frac{\mu(L_1L_2)}{L_1L_2}\prod_{p \mid L_1}\left(1-\frac{1}{p^2}\right)^{-1}\sum_{r \mid L_1} \frac{\mu(r)}{r^2} \sum_{d \mid L_2} \mu(d) \sum_{P_1}\mkern-1mu ^d\cdots \sum_{P_n}\mkern-1mu ^d \frac{1}{(\log Q)^n}\nonumber \\
        &\hspace{+0.5cm} \times \esum \frac{1}{e}  \int_{-\infty}^\infty \cdots \int_{-\infty}^\infty \widehat{H}(u, v_1,\ldots ,v_n)\sum_{c\geq 1} \sumprimesdistinct{1}{n} \prod_{i=1}^n \frac{\log p_i}{\sqrt{p_i}}V\left(\frac{p_i}{P_i}\right)\mathrm{e}\left(v_i\frac{p_i}{P_i}\right)\nonumber\\
        &\hspace{+.5cm} \times \sum_s \frac{S(e^2, x; cL_1rds)}{cL_1rds}h_u\left(\frac{4 \pi \sqrt{e^2x}}{cL_1rds}\right) \,du\,dv_1\cdots \,dv_n.
    \end{align}
By the Kuznetsov trace formula (Lemma \ref{lemma:kuznetsov}), we have that
    \begin{align}
        \sum_s \frac{S(e^2, x; cL_1rds)}{cL_1rds}h_u\left(\frac{4 \pi \sqrt{e^2x}}{cL_1rds}\right)
        \ =\ \mathrm{Hol}(c, x; u) + \mathrm{Dis}(c, x; u) + \mathrm{Ctn}(c, x; u)
    \end{align}
where
    \begin{align}
        \mathrm{Hol}(c, x; u)\  &\coloneqq\  \frac{1}{2\pi} \sum_{\substack{\ell \geq 2 \ \mathrm{even} \\ 1 \leq j \leq \theta_\ell(cL_1rd)}} (\ell-1)!\sqrt{xe^2} \ \overline{\psi_{j, \ell}}(e^2)\psi_{j, \ell}(x)h_h(\ell), \\
        \mathrm{Dis}(c, x; u)\ &\coloneqq \ \sum_{j=1}^\infty \frac{\overline{\rho_j}(e^2)\rho_j(x)\sqrt{xe^2}}{\cosh \left(\pi \kappa_j\right)} h_{+}\left(\kappa_j\right),\label{disdefn} \\
        \mathrm{Ctn}(c, x; u)\ &\coloneqq\  \frac{1}{\pi}\sum_{\mathfrak{c}} \int_{-\infty}^\infty \frac{\sqrt{xe^2}}{\cosh(\pi t)}\overline{\varphi_{\mathfrak{c}}}(e^2,t)\varphi_{\mathfrak{c}}(x, t)h_+(t)\,dt
    \end{align}
and $e, r, d, L_1,$ and $L_2$ are fixed. As such, the forms we work with are now of level $N \coloneqq cL_1rd$.

We now prove Proposition \ref{prop42}. It suffices to bound $\Sigma_2$, which follows from the next propositions.
\begin{prop}\label{Proposition 6.2}
With notation as above, we have that
    \begin{align} 
        \sum_{c \geq 1} \sumprimesdistinct{1}{n} &\prod_{i=1}^n \frac{\log p_i}{\sqrt{p_i}}\mathrm{e} \left(v_i\frac{p_i}{P_i}\right)V\left(\frac{p_i}{P_i}\right)\left[\mathrm{Hol}\left(c, \prod p_i; u\right) + \mathrm{Dis}\left(c, \prod p_i; u\right)\right]\\
        \ &\ll\ Q^{\e-1} (1 + |u|)^2\prod_{i=1}^n(1 + |v_i|)^2 \sqrt{P_i}.
    \end{align}
\end{prop}

\begin{prop}
\label{prop:their6.3}
With notation as above, we have that 
    \begin{align}
        \sum_{c \geq 1} \sumprimesdistinct{1}{n} &\prod_{i=1}^n \frac{\log p_i}{\sqrt{p_i}}\mathrm{e}\left(v_i\frac{p_i}{P_i}\right)V\left(\frac{p_i}{P_i}\right)\mathrm{Ctn}\left(c, \prod p_i; u\right)\nonumber \\
        &\ll\   (1+|u|)^2 Q^\e \left(Q^{1-\delta^*}+\frac{\sqrt{\prod P_i}}{Q}\right) \prod_{i=1}^n (1+|v_i|)^3.
    \end{align}
\end{prop}

Taking absolute values, applying Lemma \ref{lemma:bound on H} with $A = 9$, and using $\sum \sigma_i\leq 4$ from Condition \eqref{cond:bipart} yields
\begin{align}
    \Sigma_2\ &\ll\ \frac{1}{N(Q)} \sum_{\substack{L_1, L_2 \\ (L_1, L_2) = 1 \\ L_1L_2 < \mathcal{L}_0}} \frac{1}{L_1L_2}\prod_{p \mid L_1}\left(1-\frac{1}{p^2}\right)^{-1}\sum_{r \mid L_1} \frac{1}{r^2} \sum_{d \mid L_2} \sum_{P_1}\mkern-1mu ^d \cdots \sum_{P_n}\mkern-1mu ^d \frac{1}{(\log Q)^n}\esum \frac{1}{e}\nonumber \\
    &\hspace{+0.5cm} \times \int_{-\infty}^\infty\cdots\int_{-\infty}^\infty \left|\widehat{H}(u, v_1,\ldots, v_n)\right|\cdot\left|\frac{S(e^2, \prod p_i; cL_1rds)}{cL_1rds}h_u\left(\frac{4 \pi \sqrt{e^2 \prod p_i}}{cL_1rds}\right)\right|\,du\,dv_1\cdots\,dv_n\nonumber\\
    &\ll\ \frac{1}{N(Q)} \sum_{L_1,L_2} \frac{\tau(L_1)\tau(L_2)^2}{L_1L_2}\sum_{P_1}\mkern-1mu ^d \cdots \sum_{P_n}\mkern-1mu ^d \frac{Q^\e}{\log^nQ}\left(Q^{1-\delta^*}+\frac{\sqrt{\prod P_i}}{Q}\right) \nonumber \\
    &\hspace{+0.5cm} \times \int_{-\infty}^\infty \cdots \int_{-\infty}^\infty \frac{(1+|u|)^2\prod(1+|v_i|)^3}{(1+|u|)^9\prod(1+|v_i|)^9} \,du\,dv_1\cdots\,dv_n\nonumber\\
    &\ll\ \frac{Q^\epsilon}{N(Q)} \sum_{P_1}\mkern-1mu ^d \cdots \sum_{P_n}\mkern-1mu ^d \left(Q^{1-\delta^*}+\frac{\sqrt{\prod P_i}}{Q}\right) \ \ll \ \frac{Q^\epsilon}{Q} \left(Q^{1-\delta^*}+\frac{Q^{\frac12 \sum 
    (\sigma_i - \delta_i)}}{Q}\right) \ \ll \ Q^{-\e}.
\end{align}

\section{Discrete and Holomorphic Terms: Proof of Proposition \ref{Proposition 6.2}}\label{sec: DISC}
In this section, we take the index for our prime sums to be $\alpha$ instead of $i$. We bound
\begin{equation}\label{disco}
\mathcal{DISC}\ \coloneqq \ \sum_{c \geq 1} \sum_{\substack{p_1,\ldots, p_n \\ p_\alpha \neq p_\beta}} \mathrm{Dis}\left(c, \prod p_\alpha; u\right)\prod_{\alpha=1}^n\frac{\log p_\alpha}{\sqrt{p_\alpha}}\mathrm{e}\left(v_\alpha\frac{p_\alpha}{P_\alpha}\right)V\left(\frac{p_\alpha}{P_\alpha}\right),
\end{equation}
since bounding $\mathrm{Hol}\left(c,\prod p_\alpha; u\right)$ follows analogously. By \eqref{disdefn}, we write
\begin{align}\label{discosplit}
\mathcal{DISC}\ &=\ \sum_{c\geq 1}\sum_{j=1}^\infty \frac{e\overline{\rho_j}(e^2)}{\cosh(\pi \kappa_j)} h_+(\kappa_j) \sum_{\substack{p_1,\ldots, p_n \\ p_\alpha \neq p_\beta}}  \rho\left( \prod_{\alpha=1}^n p_\alpha \right) \prod_{\alpha=1}^n \frac{\sqrt{p_\alpha}\log(p_\alpha)}{\sqrt{p_\alpha}}\mathrm{e}\Bigg(v_\alpha\frac{p_\alpha}{P_\alpha}\Bigg)V\Bigg(\frac{p_\alpha}{P_\alpha}\Bigg),
\end{align}
where the $j$-sum is over the orthonormal Hecke basis detailed in Section \ref{sec:autformsII}. We split the sum into the contribution from $\kappa_j \in \rr$ and from $\kappa_j \in i \rr$, which we denote $S_1$ and $S_2$, respectively. We bound $S_1$ as bounding $S_2$ follows similarly. We begin by bounding the innermost sum.

\begin{lemma}\label{boundsonmel}
With notation as above, we have
\begin{align}
\mcT\ =\ \sumprimesdistinctdos{1}{n} \rho_j\left(\prod p_\alpha\right) \prod_{\alpha=1}^n \frac{\log (p_\alpha)\sqrt{ p_\alpha}}{\sqrt{p_\alpha}}\mathrm{e}\left(v_\alpha\frac{p_\alpha}{P_\alpha}\right)V\left(\frac{p_\alpha}{P_\alpha}\right) \ \ll\ |\rho_f(1)|(NQ)^\e \prod_{\alpha=1}^n (1 + |v_\alpha|)^2.
\end{align}
\end{lemma}
\begin{proof}
    We follow the ideas in the proof of Lemma 7.1 of \cite{BCL}. By Mellin inversion,
    \begin{align}
    \mcT\ &=\ \left(\frac{1}{2\pi i}\right)^n
    \sum_{\substack{p_1,\ldots, p_n \\ p_\alpha \neq p_\beta}} 
    \rho_j\left(\prod p_\alpha\right) \prod_{\alpha=1}^n \frac{\sqrt{p_\alpha}\log (p_\alpha)}{\sqrt{p_\alpha}}V_{0}\left(\frac{p_\alpha}{P_\alpha}\right) \nonumber \\
    &\hspace{.5cm}\times \int_{(0)} \cdots \int_{(0)} \prod_{\alpha=1}^n p_\alpha^{-s_\alpha}\widetilde{\mcW}_{v_\alpha}(s_\alpha)P_\alpha^{s_\alpha}\,ds_1\cdots \,ds_n,
    \end{align}
    where $\widetilde{\mcW}_{v_\alpha}$ is the Mellin transform of $\mcW_{v_\alpha} = \mathrm{e}(v_\alpha x)V(x),$ and $V_0$ is a smooth function compactly supported on $(0,\infty)$ such that $V_0(x) = 1$ whenever $V(x)\neq 0.$ By the proof of Lemma 7.1 in \cite{BCL}, for any $A > 0$
    \begin{align}
        \widetilde{\mcW}_{v_\alpha}(it)\ \ll_A\ \left(\frac{1+|v_\alpha|}{1+|t|} \right)^A. \label{eqn:boundonW}
    \end{align}
    We split into $n+1$ cases based on the number of primes dividing $N$. Without loss of generality, we assume that $p_1,\ldots, p_\ell \mid N$ while $p_{\ell+1},\ldots, p_n\nmid N$ for $0\leq \ell \leq n.$ Therefore it suffices to bound
    \begin{align}
    \mcT_{\ell}\ &=\ \left(\frac{1}{2\pi i}\right)^n
    \sum_{\substack{p_1,\ldots, p_\ell \mid N \\ p_\alpha \neq p_\beta}} 
    \sum_{\substack{p_{\ell +1},\ldots, p_n \nmid N \\ p_\alpha \neq p_\beta}} \rho_j\left(\prod_{\alpha=1}^n p_\alpha\right) \prod_{\alpha=1}^n \frac{\sqrt{p_\alpha}\log (p_\alpha)}{\sqrt{p_\alpha}} V_{0}\left(\frac{p_\alpha}{P_\alpha}\right) \nonumber \\
    &\hspace{+0.5cm} \times \int_{(0)} \cdots \int_{(0)}\prod_{\alpha=1}^n p_\alpha^{-s_\alpha}\widetilde{\mcW}_{v_\alpha}(s_\alpha)P_\alpha^{s_\alpha}\,ds_1\cdots \,ds_n
    \end{align}
    for $0 \leq \ell \leq n.$

    We separate the primes $p_1, \ldots, p_\ell$ dividing the level from the primes $p_{\ell + 1}, \ldots, p_n$ not dividing the level using the quasi-multiplicativity relation satisfied by the Fourier coefficients of Hecke-Maass cuspforms in the basis $\mcB_k(N)$. Since $(\prod_{\alpha = \ell+1}^n p_\alpha, N) = 1$, by \eqref{eqn:FourierCoeffMult} with $q = \prod_{\alpha = \ell + 1}^n p_\alpha$ and $m' = \prod_{\alpha = 1}^\ell p_\alpha$, we have that
    \begin{align}
             \rho_j\left(\prod_{\alpha = 1}^n p_\alpha\right) \prod_{\alpha = 1}^n \sqrt{p_\alpha} \ = \ \lambda_f\left(\prod_{\alpha = \ell + 1}^n p_\alpha \right)\rho_j\left(\prod_{\alpha=1}^\ell p_\alpha\right)\prod_{\alpha = 1}^\ell \sqrt{p_\alpha}.
        \end{align}
    Thus, using Hecke multiplicativity since the primes not dividing the level are distinct, writing $s_\alpha = it_\alpha$, and taking absolute values, we conclude that
    \begin{align}
        \mcT_{\ell}\ &\ll \ \int_{-\infty}^\infty \cdots \int_{-\infty}^\infty
        \left|\sum_{\substack{p_1,\ldots, p_\ell \mid N \\ p_\alpha \neq p_\beta}} 
        \rho_j\left(\prod_{\alpha=1}^\ell p_\alpha\right) \prod_{\alpha=1}^\ell\log (p_\alpha) V_{0}\left(\frac{p_\alpha}{P_\alpha}\right) p_\alpha^{-it_\alpha}\right| 
        \left|\prod_{\alpha=1}^n \widetilde{\mcW}_{v_\alpha}(it_\alpha)P_\alpha^{it_\alpha}\right|
         \nonumber \\
        &\hspace{+0.5cm} \times \left| \sum_{\substack{p_{\ell +1},\ldots, p_n \nmid N \\ p_\alpha \neq p_\beta}} \prod_{\alpha=\ell+1}^n \frac{\lambda_f(p_\alpha)\log (p_\alpha)}{p_\alpha^{1/2 +it_\alpha}}V_{0}\left(\frac{p_\alpha}{P_\alpha}\right)\right| \,dt_1\cdots \,dt_n.
    \end{align}
    We apply Lemma \ref{lemma:mostbeautifullemmaintheworld} to the sums over $p_\alpha \nmid N$ to obtain
    \begin{align}
        \mcT_\ell &\ll \ \int_{-\infty}^\infty \cdots \int_{-\infty}^\infty
        \sum_{\substack{p_1,\ldots, p_\ell \mid N \\ p_\alpha \neq p_j}}\left| \rho_j\left(\prod_{\alpha=1}^\ell p_\alpha\right) \prod_{\alpha=1}^\ell\log (p_\alpha) V_{0}\left(\frac{p_\alpha}{P_\alpha}\right) p_\alpha^{-it_\alpha}\right| \\
        &\hspace{+0.5cm} \times \left| \prod_{\alpha=\ell+1}^n \log^{1+\e} (P_\alpha + 2) \log(N+k+|t_\alpha|) + \log\left(N P_\alpha\right) \right|\left|\prod_{\alpha = 1}^n  \widetilde{\mcW}_{v_\alpha}(it_\alpha)P_\alpha^{it_\alpha}\right| \,dt_1\cdots \,dt_n. \nonumber
    \end{align}
    Now, we use that each $\log(P_\alpha) \ll \log Q\ll  Q^\e$ and take \eqref{eqn:boundonW} with $A = 2$, to conclude that
    \begin{align}
        \mcT_\ell \ &\ll \ \int_{-\infty}^\infty \cdots \int_{-\infty}^\infty
        \sum_{\substack{p_1,\ldots, p_\ell \mid N \\ p_\alpha \neq p_\beta}} 
        \left|\rho_j\left(\prod_{\alpha=1}^\ell p_\alpha\right)\prod_{\alpha = 1}^\ell \log (p_\alpha) V_{0}\left(\frac{p_\alpha}{P_\alpha}\right) p_\alpha^{-it_\alpha}\right| \nonumber \\
        &\hspace{+0.5cm} \times  \prod_{\alpha=\ell+1}^n  Q^\e (\log N) (1+|t_\alpha|)^\e \prod_{\alpha=1}^n \left(\frac{1+|v_\alpha|}{1+|t_\alpha|} \right)^2 \,dt_1\cdots \,dt_n.
    \end{align}
    By positivity, we drop the condition $p_\alpha \neq p_\beta$ on the sums over primes $p_\alpha \mid N$ and from \eqref{eqn:rhobound}, we have that $|\rho_j(p_1\cdots p_\alpha)| \ll N^\e (p_1 \cdots p_\alpha)^\e |\rho_f(1)|$. Hence,
    \begin{align}
        \mcT_\ell \ &\ll \ (NQ)^\e \int_{-\infty}^\infty \cdots \int_{-\infty}^\infty
        \hspace{+0.1cm} \sum_{p_1,\ldots, p_\ell \mid N} N^\e |\rho_f(1)| 
        \prod_{\alpha = 1}^\ell p_\alpha^\e  \prod_{\alpha=1}^n \frac{(1+|v_\alpha|)^2}{(1 + |t_\alpha|)^{2-\e}}  \,dt_1\cdots \,dt_n\\
        &\ll \ |\rho_f(1)| (NQ)^\e \int_{-\infty}^\infty \cdots \int_{-\infty}^\infty 
        \prod_{\alpha=1}^n \frac{(1+|v_\alpha|)^2}{(1 + |t_\alpha|)^{2-\e}}  \,dt_1\cdots \,dt_n \ \ll \ |\rho_f(1)| (NQ)^\e \prod_{\alpha=1}^n (1+|v_\alpha|)^2,\nonumber
    \end{align}
    where the sum over $p_1,\ldots,p_\ell \mid N$ occurs at most $\tau(N)^\ell$ many times, and each term is $\ll N^{\e}.$
\end{proof}

We now bound $S_1$ using the above lemma. By \eqref{eqn:rhobound} and using that $e \ll Q^\e$, we obtain
\begin{align} \label{ecoeffs}
e\overline{\rho_j}(e^2)\ \ll\  N^\e e^{1 + \e}|\rho_f(1)| \ \ll \ (NQ)^\e |\rho_f(1)|.
\end{align}
Combining Lemma \ref{lemma:their3.3}(a), \eqref{discosplit}, Lemma \ref{boundsonmel}, and \eqref{ecoeffs} yields
\begin{align}
S_1 \ &\ll \ \sum_{c\geq 1}\sum_{j=1}^\infty \frac{e\overline{\rho_j}(e^2)}{\cosh(\pi \kappa_j)} h_+(\kappa_j)|\rho_f(1)|(NQ)^\e \prod_{\alpha = 1}^n (1+|v_\alpha|)^2 \label{sqrt X garbage} \\
&\ll\ \sum_{c\geq1} \sum_{j=1}^\infty \frac{(NQ)^\e|\rho_f(1)|}{\cosh\left(\pi \kappa_j\right)}\frac{(1 + |\log X|)}{F^{1-\e}}\left(\frac{F}{1 + |\kappa_j|}\right)^C\min\left\{X^{k-1}, \frac{1}{\sqrt{X}}\right\} |\rho_f(1)|  \prod_{\alpha = 1}^n (1+|v_\alpha|)^2\nonumber \\
&\ll\ \sum_{c\geq1} \min\left\{X^{k-1}, \frac{1}{\sqrt{X}}\right\}(NQ)^\e \prod_{\alpha = 1}^n (1+|v_\alpha|)^2\sum_{j=1}^\infty \frac{|\rho_f(1)|^2}{\cosh(\pi\kappa_j)}\frac{\left(1 + |\log X|\right)}{F^{1-\e}}\left(\frac{F}{1 + |\kappa_j|}\right)^C, \nonumber
\end{align}
where $f$ is the Hecke newform of level $M$ such that $u_j = f^{(g)}$ with $M \mid N$ and $gM \mid N$. We now proceed analogously to \cite[Section 7]{BCL} to obtain Proposition \ref{Proposition 6.2}.

\section{Eisenstein Series: Proof of Proposition \ref{prop:their6.3}}\label{sec:EisensteinSeries}
The parameterization given by \cite{Kiral} is crucial to our treatment here. In the following lemmas, we recall their parameterization and specialize the results of \cite[Section 3.3]{Kiral} to our setting.

\begin{lemma}[Corollary 3.2 of \cite{Kiral}]\label{lemma:their8.1}
For $N\in \zz^+$, $(ab)^{-1}$ forms a complete set of representatives for the set of inequivalent cusps of $\Gamma_0(N)$, where $b$ runs over divisors of $N$ and $a$ runs over elements of $(\zz/(b,N/b)\zz)^\times$. We can choose $a$ so that $(a,N)=1$ after adding some multiple of $(b,N/b)$. 
\end{lemma}

\begin{lemma}\label{lemma:their8.2}
Let $\fc$ be a cusp of $\Gamma_0(N)$, and let $\fc=(ab)^{-1}$ be its representative given by Lemma \ref{lemma:their8.1}. Suppose that $b=b_0b'$, where $b_0$ is the largest factor of $b$ that is relatively prime to $N/b$. Write $\varphi$ for the Euler totient function, $\tau$ for the Gauss sum, and $\chi_0$ for the principal character modulo $N/b$.
\begin{enumerate}[(a)]
\item Let $\fn$ be a squarefree integer. Then $\varphi_{\fc}(\fn, t)=0$ unless 
    \begin{align}\label{eqn:8.2nmequality}
        \fn\ =\ \frac{b'}{(b,N/b)}m.
    \end{align}
In this case, for some (potentially trivial) decomposition $\fn=\prod_{j}\fp_j \prod_{\iota}\fq_\iota$ we have that $m = \prod_{\iota}\fq_\iota$ and $b' = (b,N/b)\prod_{j}\fp_j$. Then
\begin{align}
    \varphi_\fc(\fn, t)\ &=\ \frac{\pi^{1/2+it}}{\Gamma(1/2+it)}\fn^{-1/2+it}\left( \frac{b'}{\prod_j \fp_j} \right)^{1/2+it}\frac{\prod_{j}\fp_j}{(Nb)^{1/2+it}} \eightKloos{\prod_\iota\fq_\iota} \\
    &\hspace{.5cm} \times \dnonsense{\prod_\iota \fq_\iota}{\chi} \nonumber.
\end{align}

\item Fix $e\in \zz^+$. Suppose $\varphi_{\fc}(\fn,t)\ol \varphi_{\fc}(e^2,t) \neq 0$ with $\fn$ and $b'/(b,N/b)$ as above. Then $\prod_j \fp_j\mid e$ and
\begin{align}
    \varphi_\fc(e^2,t)\ &=\ \frac{\pi^{1/2+it}}{\Gamma(1/2+it)}e^{-1+2it} \left( \frac{b'}{\prod_j \fp_j} \right)^{1/2+it}\frac{\prod_{j}\fp_j}{(Nb)^{1/2+it}} \eightKloos{e^2} \nonumber\\
    &\hspace{.5cm} \times \dnonsense{m}{\chi}.
\end{align}
\end{enumerate}
\end{lemma}

\begin{proof}
For $n$ a positive integer, Theorem 3.4 in \cite{Kiral} with $\fa=1/N$ gives $\varphi_{\fc}(n,t)=0$ unless
\begin{align}
    n\ =\ \frac{b'}{(b,N/b)}m,\label{eqn:8.2nmequality2}
\end{align}
for some integer $m$. This, together with \cite[(3.3)]{Kiral}, implies that
\begin{align}
    \varphi_{\fc}(n,t)\ &=\ \frac{\pi^{1/2+it}}{\Gamma{(1/2+it)}}n^{-1/2+it}\frac{(b,N/b)^{1/2+it}}{(Nb)^{1/2+it}}\frac{b'}{(b,N/b)}\eightKloos{m}\nonumber \\
    &\hspace{.5cm} \times \dnonsense{m}{\chi}.\label{eqn:kygeneral}
\end{align}
Since $\fn$ is squarefree, for any prime $p\mid \fn$, we either have $p\mid\mid b'/(b,N/b)$ or $p\mid\mid m$. Thus $b'=(b,N/b)\prod_j \fp_j$ and $m=\prod_\iota \fq_\iota$ for some decomposition $\fn=\prod_{j}\fp_j \prod_{\iota}\fq_\iota$. Substituting into \eqref{eqn:kygeneral} and simplifying gives (a).

Now, assume $\varphi_{\fc}(\fn,t)\ol \varphi_{\fc}(e^2,t) \neq 0$. Then $\varphi_{\fc}(\fn, t)\neq 0$ implies that $b'/(b,N/b) = \prod_j \fp_j$. Since $\varphi_{\fc}(e^2, t)\neq 0$, we must have $ e^2 = m\prod_j{\fp_j}$. This implies $\prod_j \fp_j\mid e$. Finally, $b' = (b,N/b)\prod_j \fp_j$ gives (b) after substituting and simplifying.
\end{proof}

\begin{remark}
For notation's sake, we write $d,m$ for the $d \mid m$ term in the expansion of $\varphi_\fc(e^2,t)$, and write $\ol d,\ol m$ for the analogous terms in the expansion of $\varphi_\fc(\fn,t)$. 
\end{remark}

From the following two propositions, we obtain Proposition \ref{prop:their6.3}.

\begin{prop}\label{prop:8alldivide}
Write $\fn=\prod p_i$; then
\begin{align}
    \mathcal{C}\mathcal{T}\mathcal{N}_{\fn\mid N}\ &\coloneqq\ \sum_{c\geq 1}\mathop{\sum\cdots \sum}_{\substack{p_i\mid N\\ p_i\neq p_j}}\eightlog{i}\sum_{\fc}\int_{-\infty}^\infty \frac{\sqrt{\fn e^2}}{\cosh(\pi t)} \ol{\varphi_{\fc}}(e^2,t){\varphi_{\fc}}(\fn,t)h_+(t)dt\nonumber\\
    &\ll\  Q^\e(1+|u|)^2\frac{\sqrt{\prod P_i}}{Q}.
\end{align}
\end{prop}

\begin{prop}\label{prop:8somedivide}
Write $\fn=\prod p_i$; then
\begin{align}
    \mathcal{C}\mathcal{T}\mathcal{N}_{\fn\nmid N}\ &\coloneqq\ \sum_{c\geq 1}\mathop{\sum\cdots \sum}_{\substack{p_1,\ldots, p_\ell \mid N\\ p_{\ell+1},\ldots, p_n\nmid N\\ p_i\neq p_j}}\eightlog{i}\sum_{\fc}\int_{-\infty}^\infty \frac{\sqrt{\fn e^2}}{\cosh(\pi t)} \ol{\varphi_{\fc}}(e^2,t){\varphi_{\fc}}(\fn,t)h_+(t)dt\nonumber\\
    &\ll\ Q^\e(1+|u|)^2\left(Q^{1-\delta^*} + \frac{\sqrt{\prod P_i}}{Q}\right)\prod_{i=1}^n (1+|v_i|)^3
\end{align}
when $\ell < n$, i.e., at least one sum is over primes not dividing $N$.
\end{prop}

We now recall the following bounds which we use throughout:
\begin{align}
    S(\alpha,0;\gamma)\ \ll\  (\alpha,\gamma)\ &\leq \ \alpha\nonumber;\\
    \tau(\chi)\ &\ll\ \sqrt{c};\text{ where $\chi$ is a character mod $c$};\nonumber\\
    \sum_{d\mid \ell}1\ &\ll\ \ell^\e;\nonumber\\
    \left(L(1+2it,\ol{\chi^2}\chi_0)\right)^{-1} \ &\ll\ (N(1+|t|))^\e.\label{eqn:8manybounds}
\end{align}

When $(a,b_0) = 1,$ by Kluyver's formula we have that
\begin{align}
    S(aa', 0; b_0) = S(a', 0;b_0) \hspace{1cm} \text{ and } \hspace{1cm} S(a, 0; b_0) = \mu(b_0).\label{eqn:kluyvver}
\end{align}

The proof of Proposition \ref{prop:8alldivide} is straightforward by taking absolute values and taking naive bounds, similar to Section 8.2 of \cite{BCL}. See \ref{appen b1} for the details.

\subsection{Proof of Proposition \texorpdfstring{\ref{prop:8somedivide}}{} -- The case where at least one \texorpdfstring{$p_i\nmid N$}{}}
Using Lemma \ref{lemma:their8.1}, we write each cusp $\fc$ as $(ab)^{-1}$, and thus the sum over cusps $\fc$ becomes 
\begin{align}
\mcC\mcT\mcN_{\fn\nmid N}\ &= \ \sum_{c\geq 1}\sum_{b\mid N}\amod\mathop{\sum\cdots \sum}_{\substack{p_1,\ldots, p_\ell \mid  N\\ p_{\ell+1},\ldots, p_n\nmid N\\ p_i\neq p_j}}\eightlog{i}
\nonumber\\
&\hspace{.5cm} \times\int_{-\infty}^\infty \frac{\sqrt{\fn e^2}}{\cosh(\pi t)} \ol{\varphi_{\fc}}(e^2,t){\varphi_{\fc}}(\fn,t)h_+(t)dt,
\end{align}
where $\fn=\prod p_i$. We look to apply Lemma \ref{lemma:their8.2}, as $\fn$ is squarefree because the primes are distinct.

When $\varphi_\fc(\fn,t)\ol\varphi_\fc(e^2,t)\neq 0$, \eqref{eqn:8.2nmequality2} holds; hence without loss of generality, we may add any implications of this as conditions. For any index $\ell +1\leq i\leq n$, the corresponding sum is over $(p_i,N)=1$, hence $(p_i, b') = 1$ and so $p_i \mid \ol m$. We thus split the initial sum into $2^{\ell}$ sums where $b'/(b,N/b)$ ranges over all of the $2^\ell$ possible products of the primes $p_1, \ldots, p_\ell \mid N$, including the case where $b' = (b,N/b)$ and none of these primes divide $b'/(b,N/b)$. Up to reordering, we may write 
    \begin{align}\label{eqn:8dumbdivsplitting}
        \mathop{\sum\cdots \sum}_{\substack{p_1,\ldots, p_\ell \mid  N}}
        \ =\ 
        \mathop{\sum\cdots \sum}_{\substack{p_1,\ldots, p_\ell \mid  N\\p_1,\ldots, p_\ell\nmid \frac{b'}{(b,N/b)}}}
        +
        \mathop{\sum\cdots \sum}_{\substack{p_1,\ldots, p_\ell \mid  N\\p_1\mid \frac{b'}{(b,N/b)}\\ p_2,\ldots, p_\ell\nmid \frac{b'}{(b,N/b)}}}
        +\cdots +
        \mathop{\sum\cdots \sum}_{\substack{p_1,\ldots, p_\ell \mid  N\\p_1\cdots p_\ell\mid \frac{b'}{(b,N/b)}}}.
    \end{align}
We split each sum over $p_i\mid N$ into $p_i \mid b_0$ and $p_i \nmid b_0$, and bound each of the resulting cases individually. We write 
    \begin{align}
        \prod_\fp \fp\ \coloneqq\ \prod_{i=1}^j p_i
        \qquad\text{ and }\qquad
        \mathop{\sum_{\fp}\sum_{\fo}\sum_{\fq}}_{p_i\neq p_j}\ \coloneqq \mathop{\sum\cdots \sum}_{(\text{Conditions }\ref{eqn:8primecondition a}-\ref{eqn:8primecondition e})},
    \end{align}
where the conditions are
\setlist[enumerate, 1] 
{1., 
leftmargin  = 2em,
itemindent  = 0pt,
labelwidth  = 2em,
labelsep    = 0pt,
font        = \bfseries,
align       = left,
itemsep     = 1.5mm,
ref         = \mbox{\textbf{\arabic*}}}
\begin{enumerate}[(a)]
    \setlength\itemsep{0em}
    \item for $p\in \{p_1,\ldots, p_\ell\}$, either we have that $p\mid b_0$ or we have that $p\mid N$ and $p \nmid b_0$, \label{eqn:8primecondition a}
    \item $p_1\cdots p_j\mid \frac{b'}{(b,N/b)}$ for some $0 \leq j\leq \ell$,\label{eqn:8primecondition b}
    \item $p_{j+1},\ldots, p_\ell \nmid \frac{b'}{(b,N/b)}$ for $j+1\leq \ell$, using the $j$ from the previous line, \label{eqn:8primecondition c}
    \item $p_{\ell+1},\ldots, p_n\nmid N$, \label{eqn:8primecondition d}
    \item $p_i\neq p_j$ for $i\neq j$, and \label{eqn:8primecondition e}
    \item $\frac{b'}{(b,N/b)}$ is squarefree and has exactly $j$ prime factors, \label{eqn:8primecondition f}
\end{enumerate}
and the $\sum_{b| N}$ sum is restricted to $b$ satisfying Condition \ref{eqn:8primecondition f}. Here, the $\fp$'s are the primes $p_1,\ldots, p_j$ whose product is $b'/(b, N/b)$; the $\fo$'s are the primes $p_{j+1},\ldots, p_\ell$ which divide the level $N$ but do not divide $b'/(b,N/b)$ and hence divide $\ol m$; and the $\fq$'s are the primes $p_{\ell+1},\ldots, p_n$ which do not divide the level and thus also divide $\ol m$. Conditions \ref{eqn:8primecondition b}-\ref{eqn:8primecondition f}
force $b'/(b,N/b)=\prod_{\fp} \fp$ and $\ol m = \prod_\fo \fo \prod_\fq \fq$. 
Since \eqref{eqn:8.2nmequality2} must also hold for $e^2$, we may additionally restrict the sum over primes in $\fp$ to those dividing $e$.

We now apply Lemma \ref{lemma:their8.2}; using the identity $\Gamma(1/2+it)\Gamma(1/2-it)=\pi/\cosh(\pi t)$, we obtain 
\begin{align*}
    \mcC\mcT\mcN_{\fn\nmid N}\ =\ \sum_\alpha E_\alpha,
\end{align*}
where 
\begin{align}
E_1\ &= \ \sum_{c\geq 1}\bsum \amod\mathop{\sum_{\fp}\sum_{\fo}\sum_{\fq}}_{p_i\neq p_j}\eightlog{i} \frac{1}{Nb_0} \eightKloos{m}
\nonumber\\
&\hspace{.5cm} \times \int_{-\infty}^\infty e^{-2it}\left( \prod_\fp \fp \prod_\fo \fo \prod_\fq \fq \right)^{it} \left( \prod_\fp \fp \right)\eightKloos{\prod_\fo \fo \prod_\fq \fq} \nonumber \\
&\hspace{.5cm} \times \conjdnonsense{m}{\chi}\nonumber \\
&\hspace{.5cm} \times \chinonsense{\prod_\fo \fo \prod_\fq \fq }{\psi}{1} h_+(t)dt.
\end{align}
Here, we take $\prod_i$ to be understood as iterating over the primes in the left sums, e.g., over the primes in $\fp,\fo$, and $\fq$. We take $\ol d=1$ for notational simplicity. For any other $\ol d$ satisfying the conditions for $\varphi_\fc(\fn,t)$ in Lemma \ref{lemma:their8.2}, the corresponding sums merely change the input to $\psi$ by various factors of $p_i$ (for some $j+1 \leq i \leq n$) and switch the exponent on some subset of primes from $(-)^{it}$ to $(-)^{-it}$. This can be handled similarly, see Lemma \ref{lemma:their8.9} and Remark \ref{remark:8bad2}. 

\begin{remark}\label{remark:8bad1}
However, crucially, due to the values that $\ol d$ can take, all possible sign combinations of $\pm \emph{it}$ on the $\fq$'s occur. Given the competing natures of these signs in the contour shift in the proof of Lemma \ref{lemma:their8.9}, a new argument is required to get an bound like that of \cite{BCL}.
\end{remark}

The other sums $E_\alpha$ are similar in structure; they differ only in the divisibility conditions implied by \eqref{eqn:8dumbdivsplitting} (and thus in the value of $b'/(b,N/b)$, which determines $\ol m$ and $m$) and the value of $\ol d$. From now on, we only focus on $E_1$ as the other cases can be treated similarly. We split into the principal and non-principal cases for $\psi$ and write 
\begin{align*}
    E_1\ =\ E_{princ}+E_{non-princ}.
\end{align*}
We explicitly recall their definitions in \eqref{eprincdefn} and \eqref{enonprincdefn} respectively and omit their definition here due to their length.
Proposition \ref{prop:8somedivide} now follows immediately from the following two lemmas.

\begin{lemma}\label{lemma:8princ}
With the notation as above, and same assumptions as in Proposition \ref{prop:8somedivide}, we have
\begin{align}
    E_{princ}\ \ll\ Q^\e(1+|u|)^2\left(Q^{1-\delta^*} + \frac{\sqrt{\prod P_i}}{Q}\right)\prod_i (1+|v_i|)^2.
\end{align}
\end{lemma}

\begin{lemma}\label{lemma:8nonprinc}
With the notation as above, and the same assumptions as in Proposition \ref{prop:8somedivide}, we have
\begin{align}
    E_{non-princ}\ \ll \ Q^\e (1+|u|)^2 \frac{\sqrt{\prod P_i}}{Q} \prod_i(1+|v_i|)^3.
\end{align}   
\end{lemma}

\subsection{Proof of Lemma \ref{lemma:8princ} -- The principal character}
Explicitly, we look to bound 
\begin{align}
E_{princ}\ &= \ \sum_{c\geq 1}\bsum\amod\mathop{\sum_{\fp}\sum_{\fo}\sum_{\fq}}_{p_i\neq p_j}\eightlog{i} \frac{1}{Nb_0} \eightKloos{m}
\nonumber\\
&\hspace{.5cm} \times \int_{-\infty}^\infty e^{-2it}\left( \prod_\fp \fp \prod_\fo \fo \prod_\fq \fq \right)^{it} \left( \prod_\fp \fp \right)\eightKloos{\prod_\fo \fo \prod_\fq \fq}  \nonumber \\
&\hspace{.5cm} \times \conjdnonsense{m}{\chi}\nonumber \\
&\hspace{.5cm} \times \princnonsense{\prod_\fo \fo \prod_\fq \fq}{\psi_0} h_+(t)dt, \label{eprincdefn}
\end{align}
where $\psi_0$ is the principal character modulo $(b,N/b)/(\prod_\fo \fo \prod_\fq \fq , (b,N/b))$.

We prove a generalization of Lemma 8.7 of \cite{BCL}. 
\begin{lemma}\label{lemma:8their8.7}
Fix $A > 0.$ We have that
\begin{align}
    \mathop{\sum\cdots \sum}_{\substack{p_1,\ldots, p_\kappa \nmid N\\ p_i \neq p_j}}
    &\prod_{\alpha} \frac{\log p_{\alpha}}{p_{\alpha}^{1/2\pm_{\alpha} it}}V\left( \frac{p_\alpha}{P_\alpha} \right)\mathrm{e}\left(v_\alpha \frac{p_\alpha}{P_\alpha}\right)\nonumber \\
    &=\ \prod_{\alpha} \left(P_\alpha^{1/2\mp_\alpha it}\widetilde V_\alpha\left(\frac{1}{2}\mp_{\alpha} it\right)+O(P_{\alpha}^\e(1+|v_{\alpha}|)^2\log N)\right),
\end{align}
where the $\pm_\alpha$ are independent of one another and based on $\alpha$, and $\widetilde{V}_{\alpha}$ 
satisfies
\begin{align}
    \widetilde{V}_{\alpha}(s)\ &\ll\ Q^\e\left(\frac{1+|v_\alpha|}{1+|s|}\right)^A.
\end{align}
\end{lemma}
\begin{proof}
This follows by induction, using Lemma 8.7 of \cite{BCL} as the base case.
For details, see Appendix \ref{appen b2}.
\end{proof}

To bound the integral over $t$, we prove a generalization of Lemma 8.9 of \cite{BCL}.

\begin{lemma}\label{lemma:their8.9}
Let
\begin{align}
    I\ &=\ \int_{-\infty}^\infty \frac{e^{-2it}d^{2it}h_+(t)}{L(1-2it, \chi^2\chi_0,) L(1+2it,\psi_0^2\chi_0)} \prod_{\alpha^+} P_{\alpha^+}^{1/2+it}\widetilde V_{\alpha^+}\left(\frac{1}{2}+it\right)\prod_{\alpha^-} P_{\alpha^-}^{1/2-it}\widetilde V_{\alpha^-}\left(\frac{1}{2}-it\right)\nonumber \\
    &\hspace{+0.5cm} \times \prod_\beta O(P_\beta^\e(1+|v_\beta|)^2\log N)\,dt,
\end{align}
where $\prod_{\alpha^*},\prod_\beta$ are finite products over at most $n$ indices in total, and at least one of $\alpha^+$ or $\alpha^-$ is a non-trivial product. Suppose $P_i \ll Q^{\sigma_i - \delta_i}$ for $1 \leq i \leq n$ and $\{\sigma_i \colon 1 \leq i \leq n\}$ satisfies Condition \eqref{cond:bipart}. Then for $\delta^* = \frac14 \sum_i \delta_i$, we have that 
\begin{align}
    I \ &\ll \ Q^{1 - \delta^*}q_0^\e \left(\frac{e}{d}\right)^{-1/2+2\e}(1+|u|)\min\left\{X^{k-3/2+2\e},\frac{1}{\sqrt{X}}\right\} \prod_{\alpha}(1+|v_\alpha|)^2\prod_{\beta}[(1+|v_\beta|)^2\log N],
\end{align}
where $q_0$ is the conductor of $\chi^2\chi_0$.
\end{lemma}
\begin{proof}
Place the index $i$ in $A$ if the sign of the imaginary part of the exponent of $P_i$ is negative; otherwise, place the index $i$ in $B$ and add the remaining indices to $A$. Since the case in which all primes are in $\prod_{\alpha^*}$ occurs in our application of this lemma, and since this term dominates, we reduce to this case.

Since Condition \eqref{cond:bipart} holds for this partition, for simplicity we assume
\begin{align}
    \sum_{i\in A}\sigma_i+3\sum_{j\in B}\sigma_j\ \leq\ 4.\nonumber
\end{align}
Thus, making the change of variables $s=1/2-it$ (if the other case of Condition \eqref{cond:bipart} holds, take $s=1/2+it$), we obtain
\begin{align}
    I\ &=\ i\int_{(1/2)} \left(\frac{e}{d}\right)^{2s-1}\frac{1}{L(2s,\chi^2\chi_0) L(2-2s,\psi_0^2\chi_0)}
    \prod_{\alpha^-} P_{\alpha^-}^{s}\widetilde V_{\alpha^-}\left(s\right)\prod_{\alpha^+} P_{\alpha^+}^{1-s}\widetilde V_{\alpha^+}\left(1-s\right)\nonumber \\
    &\hspace{+0.5cm} \times h_+(i(s-1/2))\prod_\beta O(P_\beta^\e(1+|v_\beta|)^2\log N)\,ds.
\end{align}
We move the line of integration to $\Re(s)=1/4+\e$, and then proceed analogously as in the proof of Lemma 8.9 of \cite{BCL}.
\end{proof}
\begin{remark}
    The condition on the supports $\sigma_i$ is necessary to handle the novel terms emerging from Lemma \ref{lemma:their8.2} when $n \geq 2$. Lemma \ref{lemma:their8.9} must handle all possible combinations of signs on the exponents of $\prod P_i^{\pm it}$. Write $\prod_{a \in A} P_a^{-it}\prod_{b \in B} P_b^{it}$. Based on our choice of change of variables, upon shifting the contour to $\mathrm{Re}(s) = 1/4 + \e$ we are left with either 
        \begin{align}
            \prod_{i \in A}P_i^{1/4+\e}\prod_{j \in B}P_j^{3/4 + \e}\hspace{1cm} \text{ or }\hspace{1cm} \prod_{i \in A}P_i^{3/4+\e}\prod_{j \in B}P_j^{1/4 + \e}.\nonumber
        \end{align}
    If Condition \eqref{cond:bipart} fails to hold, since each $P_i\ll Q^{\sigma_i - \delta_i}$, either choice of change of variables, $s = 1/2 - it$ or $s = 1/2 + it$, yields a factor of size $\gg Q^{1+\e}$.
\end{remark}
\begin{remark}\label{remark:8bad2}
    The argument in the above lemma cannot be pushed any further to increase the support; the contour shift would otherwise cross the critical line, and is thus a natural boundary for our treatment of Eisenstein series. See Remark \ref{remark:8bad1} as well.
\end{remark}
To bound $E_{princ}$, we first look to apply Lemma \ref{lemma:8their8.7} to the sum over primes $\fq$ not dividing $N$. To do so, we first note that $(\prod_\fq \fq, b_0) = 1$; by \eqref{eqn:kluyvver}, we may remove $\prod_\fq \fq$ from the Kloosterman, and similarly from the corresponding Euler totient function.

We move the sum over $\fq$ inside the integral and apply Lemma \ref{lemma:8their8.7}. Next, we note that for every $m$, the modulus $(b,N/b)/(m,(b,N/b))$ of $\chi$ is a factor of $(b, N/b)$. Thus, moving the sum $\sum_{a\bmod (b,N/b)}\mkern -90mu^*\mkern85mu$ inside the integral, character orthogonality allows us to reduce to the case where $\chi = \chi_\circ$ is the principal character modulo $(b,N/b)/(m, (b,N/b))$. This gives
\begin{align}
    E_{princ}\ &=\ \sum_{c\geq 1}\bsum\mathop{\sum_\fp\sum_\fo}_{p_i\neq p_j}\eightlog{i} \frac{1}{Nb_0}\int_{-\infty}^\infty e^{-2it}\left(\prod_\fp \fp \prod_\fo \fo \right)^{it} \nonumber \\
    &\hspace{.5cm} \times  \left( \prod_\fp \fp \right) \eightKloos{\prod_\fo \fo} \eightKloos{m} \nonumber \\
    &\hspace{.5cm} \times \conjdnonsenseprinc{m}{\chi}\nonumber \princnonsense{\prod_\fo \fo}{\psi_0}\nonumber\\
    &\hspace{.5cm} \times h_+(t) \prod_{\alpha^*}\left(P_{\alpha^*}^{1/2\pm_* it}\widetilde V_{\alpha^*}\left(\frac{1}{2}\pm_* it\right)+O(P_{\alpha^*}^\e(1+|v_{\alpha^*}|)^2\log N)\right) dt,\label{eqn:8beforesplt}
\end{align}
where $\alpha^*$ iterates over $\alpha^+$ and $\alpha^-$. We split into cases of all error and at least one main term.

\begin{remark} 
The application of character orthogonality of $\chi$ eliminates the contribution from the Gauss sums $\tau(\chi)$ over non-principal characters modulo ${(b,N/b)}/{(m, (b, N/b))}$, which would otherwise contribute a detrimental additional factor of magnitude $\sqrt{{(b,N/b)}/{(m, (b, N/b))}}.$
\end{remark}

\subsubsection{All error terms}
We consider the all error term of \eqref{eqn:8beforesplt}.
We move the remaining prime sums into the integral and apply absolute values. By Condition \ref{eqn:8primecondition a}, we have some combination of the sums $\fp,\fo\mid b_0$, or $\fp,\fo\mid N,\nmid b_0$. We split the sum over primes in $\fo$ into $\ff \mid b'$, $\fg \mid b_0$, and $\fh \mid N, \fh \nmid b$ as follows:
    \begin{align}
        \prod_{\fo} \fo\  =\  \prod_{\substack{\ff \mid b \\ \ff \nmid b_0}} \ff \prod_{\fg \mid b_0} \fg \prod_{\substack{\fh \mid N \\ \fh \nmid b}} \fh \ =\  \prod_{\substack{\ff \mid b'}} \ff \prod_{\fg \mid b_0} \fg \prod_{\substack{\fh \mid N \\ \fh \nmid b}} \fh.\nonumber
    \end{align}
Since $b_0$ is the largest factor of $b$ coprime to $N/b$, primes in $\fo$ which divide $b'$ cannot be coprime to $N/b$ so $\left(\prod_{\fo \mid b'} \fo, (b, N/b)\right) = \prod_{\fo \mid b'} \fo$. If $\fo \mid N$ but $\fo\nmid b$, then $\left(\prod_{\fo \nmid b} \fo, (b, N/b)\right) = 1.$ Hence
\begin{align}
    \eightKloos{\prod_\fo \fo}\ = \ S\left(\prod_{\fg} \fg, 0; b_0\right) \ \ll \ \sqrt{b_0} \prod_{\fg} \sqrt{\fg}\label{eqn:kloosshift},
\end{align}
where we used \eqref{eqn:kluyvver} and to obtain the last bound, we have used that the primes $\fo \mid b_0$ are distinct so their product is at most $b_0$. Thus, taking absolute values, bounding the Kloosterman sum, and dropping the condition $p_i \neq p_j$ by positivity yields
\begin{align}
    \sum_{\substack{\fo\\ p_i\neq p_j}}&\eightlog{i}\left(\prod_\fo \fo \right)^{it}S\left(\prod_{\fg} \fg, 0; b_0\right)\nonumber\\
    &\ll \ \sqrt{b_0} \sum_{\fg \mid b_0} \prod_\fg \log \fg \sum_{\ff | b'} \prod_\ff \frac{\log \ff}{\sqrt{\ff}} \sum_{\substack{\fh \mid N \\ \fh \nmid b}} \prod_\fh \frac{\log \fh}{\sqrt{\fh}}\
    \ll\ (b_0)^{1/2+\e}N^\e.\label{eqn8:osumdiescase1}
\end{align}
We now deal with the sum over $\fp$. Recall that we have $\prod_\fp \fp \mid e$. Since $e \leq  E \ll  Q^\e$, we have that
\begin{align}
    \sum_{\substack{\fp \\(\text{Conditions }\ref{eqn:8primecondition a},\ref{eqn:8primecondition b},\ref{eqn:8primecondition e})}} 
    \prod_i 
    \frac{\log p_i}{\sqrt{p_i}} V\left(\frac{p_i}{P_i}\right)\mathrm{e}\left(v_i\frac{p_i}{P_i}\right)p_i^{it}p_i\ \ll\ Q^\e\label{eqn:8pesumdies}.
\end{align}
We are now ready to bound $E_{princ,err}$. Using the bounds on $h_+(t)$ from Lemma \ref{lemma:their3.3} and the work above (after dropping the condition on $\sum_{b|N}$ by positivity), we have that
\begin{align}
    E_{princ,err}\ &\ll\ Q^\e(1+|u|) \prod_{i=1}^n (1+|v_i|)^2 \sum_{c\geq 1}\sum_{b\mid N}\varphi((b,N/b)) \nonumber \\
    &\hspace{.5cm} \times \frac{e^{2+\e}(b_0)^{1/2+\e}}{N^{1-\e}b_0} \sqrt{(b,N/b)}F^{1+\e}\min\left\{X^{k-1},\frac{1}{\sqrt{X}}\right\}(1+|\log X|),
\end{align}
where $F<(1+|u|)(1+X)$ and we use the bounds in \eqref{eqn:8manybounds}. By $\varphi((b,N/b)) \leq N/b$, we conclude
\begin{align}
\sum_{b\mid N}\varphi((b,N/b))\frac{e^{2+\e}(b_0)^{1/2+\e}}{N^{1-\e}b_0} \sqrt{(b,N/b)}\ \ll\ \sum_{b\mid N} N^\e \frac{\sqrt{(b,N/b)}} {b(b_0)^{1/2-\e}}\ \ll\ \tau(N) N^\e\ \ll\ N^\e.
\end{align}
As such, arguing as in \cite[p.~30]{BCL} to the sum over $c$, we obtain
\begin{align}
    E_{princ,err}\ &\ll\ Q^\e(1+|u|)^2\frac{\sqrt{\prod P_i}}{Q}\prod_{i=1}^n (1+|v_i|)^2.
\end{align}

\subsubsection{At least one main term}
We now look to bound
\begin{align}
    E_{princ,main}\ &\coloneqq\ \sum_{c\geq 1}\bsum \mathop{\sum_\fp\sum_\fo}_{p_i\neq p_j}\eightlog{i} \frac{1}{Nb_0}\int_{-\infty}^\infty e^{-2it}\left(\prod_\fp \fp \prod_\fo \fo \right)^{it} \nonumber \\
    &\hspace{.5cm} \times  \left( \prod_\fp \fp \right) \eightKloos{\prod_\fo \fo} \eightKloos{m} \nonumber \\
    &\hspace{.5cm} \times \conjdnonsenseprinc{m}{\chi}h_+(t) \prod_{\alpha^*} P_{\alpha^*}^{1/2\pm_* it}\widetilde V_{\alpha^*}\left(\frac{1}{2}\pm_* it\right) \nonumber \\
    &\hspace{.5cm} \times \princnonsense{\prod_\fo \fo}{\psi_0}  \prod_\beta O(P_\beta^\e(1+|v_\beta|)^2\log N)\,dt,
\end{align}
where the product over $\alpha^*$ is not empty.

We look to argue as in Lemma \ref{lemma:their8.9}. We move the sums over $\fp,\fo$ into the integral (including the Kloosterman sum with $\fo$ dependency) and obtain the integral in the statement of Lemma \ref{lemma:their8.9} with an additional factor of
\begin{align}
    T\ &\coloneqq\ \frac{\varphi{((b,N/b))}}{Nb_0}\mathop{\sum_\fp\sum_\fo}_{p_i\neq p_j}\eightlog{i}\left(\prod_\fp \fp \prod_\fo \fo\right)^{\pm it}\left(\prod_\fp \fp \right)\nonumber \\
    &\hspace{.5cm} \times \frac{1}{\varphi\left( \frac{(b,N/b)}{(\prod_\fo \fo, (b,N/b))} \right)}\eightKloos{\prod_\fo \fo}.
\end{align}
In the worst case, after shifting the contour as in Lemma \ref{lemma:their8.9}, we are left with $\prod_\fp \fp^{1/4 - \e} \prod_\fo \fo^{1/4 - \e}$. We take absolute values to obtain
\begin{align}
    T\ &\ll\ \frac{\varphi{((b,N/b))}}{Nb_0}\mathop{\sum_\fp\sum_\fo}_{p_i\neq p_j}\prod_i \frac{\log p_i}{p_i^{1/4+\e}} \left(\prod_\fp \fp \right) \frac{1}{\varphi\left( \frac{(b,N/b)}{(\prod_\fo \fo,(b,N/b))} \right)}\left|\eightKloos{\prod_\fo \fo}\right|.
\end{align}
We again split the sum over primes in $\fo$ into sums over $\ff \mid b'$, $\fg \mid b_0$, and $\fh \mid N, \fh \nmid b$ to arrive at
\begin{align}
    T\ &\ll\ \frac{\varphi{((b,N/b))}}{Nb_0}\sum_\fp \prod_\fp \frac{\log \fp}{\fp^{1/4+\e}}\fp \sum_{\substack{\ff \mid b'}} \prod_\ff \frac{\log \ff}{\ff^{1/4+\e}} \sum_{\fg \mid b_0} \prod_\fg \frac{\log \fg}{\fg^{1/4+\e}} \sum_{\substack{\fh \mid N \\ \fh \nmid b}} \prod_\fh \frac{\log \fh}{\fh^{1/4+\e}}\nonumber \\
    &\hspace{+0.5cm} \times \frac{1}{\varphi\left( \frac{(b,N/b)}{(\prod \ff, (b,N/b))} \right)}\left|S\left(\prod \fg, 0; b_0\right)\right|,
\end{align}
where we have simplified the second Euler totient function using that none of the primes $\fg$ or $\fh$ divide $(b, N/b)$ and we have simplified the Kloosterman sum using that none of the primes $\ff$ or $\fh$ divide $b_0$. For notational convenience, we do not write the requirement that all primes are distinct here, yet we still assume this condition. By Theorem 2.9 of \cite{mont}, $\varphi(n) \gg n^{1-\e}$ thus,
    \begin{align}
        \frac{1}{\varphi\left( \frac{(b,N/b)}{(\prod \ff, (b,N/b))} \right)}\ \ll\ \frac{N^\e(\prod \ff, (b,N/b))}{(b,N/b)}\ \ll\ \frac{N^\e \prod \ff}{(b,N/b)}.
    \end{align}
    As such, using $b = b'b_0$ and $b'=(b,N/b)\prod_\fp \fp$, we conclude 
    \begin{align}
        \frac{1}{\varphi\left( \frac{(b,N/b)}{(\prod \ff, (b,N/b))} \right)}\frac{\varphi((b,N/b))}{Nb_0}\
        \ll\ \frac{N^\e \prod \ff}{(b,N/b)}\frac{(b,N/b)b'}{Nb}\ 
        \ll\ \frac{(b,N/b)}{N^{1-\e}b}\prod_\fp \fp\prod_\ff \ff.
    \end{align}
    We then use that the primes $\fg$ are distinct together with the argument in \eqref{eqn:kloosshift} to deduce
    \begin{align}
        \left|S\left(\prod_\fg \fg, 0; b_0\right)\right|\ &\ll\ (b_0)^{3/4}\left(\prod_\fg \fg \right)^{1/4}.
    \end{align}
    Therefore,
    \begin{align}
        T\ &\ll\ Q^\e\frac{(b,N/b)(b_0)^{3/4}}{N^{1-\e}b}\sum_\fp \prod_\fp \frac{\log \fp}{\fp^{1/4+\e}} \fp^2 \sum_{\ff\mid b'}\prod_\ff \ff^{3/4-\e}\log \ff \sum_{\fg\mid b_0}\prod_\fg \log \fg \sum_{\substack{\fh\mid N\\ \fh\nmid b}} \prod_\fh \frac{\log \fh}{\fh^{1/4+\e}}.
    \end{align}
    Thus using that the $\ff$ are distinct,
        \begin{align} \label{eqn:8introb'inden}
            \sum_{\ff \mid b'} \ff^{3/4 - \e}\log \ff \ &=\ \mathop{\sum \cdots \sum}_{p_i \in \ff} \prod_i p_i^{3/4-\e} \log p_i \ \ll \  (b')^{3/4-\e}\log^{2n}(b').
        \end{align}
    By similar arguments to those in \eqref{eqn:8pesumdies}, the sums over primes in $\fp$ and $\fh$ are $\ll (QN)^\e$. Similarly, the sum over primes in $\fg$ is $\ll b_0^\e$. As such,
    \begin{align}
        T\ &\ll\ Q^\e\frac{(b,N/b)(b_0)^{3/4+\e}(b')^{3/4-\e}}{N^{1-\e}b}\ \ll\ Q^\e\frac{(b,N/b)}{N^{1-\e}b^{1/4-\e}}.
    \end{align}
We recall that $q_0^\e \ll N^\e$. Therefore, we have that
\begin{align}
E_{princ,main}\ &\ll\ (1+|u|)Q^{1-\delta^*}\prod_{i=1}^{n}(1+|v_i|)^2 \sum_{c\geq 1}\sum_{b\mid N} \frac{(b,N/b)}{N^{1-\e}b^{1/4-\e}}\min\left\{X^{k-3/2+2\e},\frac{1}{\sqrt{X}}\right\},
\end{align}
after dropping the condition on $\sum_{b|N}$ by positivity. We now bound the sum over $c$ and $b \mid N$. Let $\mcF(c) \coloneqq \min\left\{X^{k-3/2+2\e},X^{-1/2}\right\}$ and put $N  = cL_1rd \coloneqq c\alpha$ where $\alpha = L_1rd \ll Q^\e$. We then have
\begin{align}
    \sum_{c\geq 1}\sum_{b\mid c\alpha}\frac{(b,N/b)}{(c\alpha)^{1-\e} b^{1/4-\e}}\mcF(c)\ &\leq\ \sum_{c \geq 1}\sum_{b_1\mid \alpha}\sum_{b_2\mid c}\frac{(b_1b_2,\frac{c\alpha}{b_1b_2})}{(c\alpha)^{1-\e}(b_1b_2)^{1/4-\e}}\mcF(c)\nonumber\\
    &\ll\ \sum_{c\geq 1} \sum_{b_2\mid c}\frac{(b_2,c/b_2)}{c^{1-\e}(b_2)^{1/4-\e}}\mcF(c)\sum_{b_1\mid \alpha}\frac{b_1(\alpha/b_1)}{(b_1)^{1/4-\e}\alpha^{1-\e}}\nonumber\\
    &\ll\ \alpha^\e \sum_{b_2} \frac{1}{(b_2)^{5/4-\e}}\sum_{c' \geq 1}\frac{(b_2,c')}{c'^{1-\e}}\mcF(c'b_2)\nonumber\\
    &\ll\ \alpha^\e \sum_{b_2} \frac{1}{(b_2)^{5/4-\e}}\sum_{\ell\mid b_2}\ell \sum_{c' \geq 1}\frac{1}{(c'\ell)^{1-\e}}\mcF(c'\ell b_2)\nonumber\\
    &\ll\ Q^\e \alpha^\e \sum_{b_2} \frac{1}{(b_2)^{5/4-\e}}\sum_{\ell\mid b_2} \frac{\ell}{\ell^{1-\e}} \nonumber \\
    &\ll\ Q^\e \sum_{b_2}\frac{(b_2)^\e}{(b_2)^{5/4-\e}}\ \ll Q^\e.
\end{align}
Hence, we conclude that
    \begin{align}
        E_{princ,main}\ \ll\ Q^{1-\delta^*+\e}(1 + |u|)\prod_{j=1}^n(1 + |v_j|)^2.
    \end{align}

The proof of Lemma \ref{lemma:8nonprinc} is straightforward and follows by applying Lemma \ref{lemma:mostbeautifullemmaintheworld} to the sum over $\fq$, taking absolute values, and taking naive bounds. The idea is similar to Section 8.5 of \cite{BCL}. See Appendix \ref{appen b3} for details.

\addtocontents{toc}{\protect\setcounter{tocdepth}{0}} 
\section*{Acknowledgements}
\addtocontents{toc}{\protect\setcounter{tocdepth}{3}} 
This research was completed at the Williams College SMALL REU Program and was supported by Williams College and the National Science Foundation (grant DMS2241623). The authors are grateful for the support of Duke University, Princeton University, and the University of Michigan. The authors appreciate Arijit Paul and Zijie Zhou for helpful preliminary discussions. The authors would like to thank Siegfried Baluyot, Vorrapan Chandee, Xiannian Li, Wenzhi Luo, and Jesse Thorner for helpful discussions. The authors also thank the anonymous referee for their comments which greatly improved the exposition and pointed out an error in an earlier draft.
\appendix

\section{Proof of Lemmas in Theorem \ref{onylallpairedoralldifferent}}
\label{app:combo}
In this section, we prove the base case of the induction and Lemmas \ref{lemmaforcase3ofcombo} and \ref{lemmaforcase4ofcombo}. Recall that to do this, we split into five disjoint cases based on the partition of $n$ as follows.
\setlist[enumerate, 1] 
{1., 
leftmargin  = 2em,
itemindent  = 0pt,
labelwidth  = 2em,
labelsep    = 0pt,
font        = \bfseries,
align       = left,
itemsep     = 1.5mm,
ref         = \mbox{\textbf{\arabic*}}
}
\begin{enumerate}
    \setlength\itemsep{0em}
    \item Each part is $2.$ This term contributes and as such we compute it in Lemma \ref{lemmaforcase1ofcombo}.
    \item All parts are $\geq 2$, and at least one part is $>2$.
    \item There is exactly one $1$ in the partition (and $n\geq 2$).
    \item There are at least two $1$'s in the partition, but not all $1$'s.
    \item Each part is $1.$
\end{enumerate}

\subsection{Base Cases} \label{subsec:base}
When $n=2$, we are either in Case \ref{case1ofcombo} or Case \ref{case5ofcombo}, and since these cases have been established independently of this induction, this completes the case $n = 2$. When $n=3$, we are either in Case \ref{case2ofcombo} or Case \ref{case5ofcombo}, or we have the partition $2+1=3$.

We begin with
\begin{align}
\mcS_1 \ \coloneqq\ \nonsense{0}\hkq \sum_{p_1=p_2\nmid q} \mcF(p_1, 2)
\sum_{\substack{p_3\nmid q \\ p_1\neq p_3}} \mcF(p_3, 1).
\end{align}
When $p_1=p_2=p_3$, $\mcS_1$ becomes
\begin{align}
\nonsense{0}\hkq \sum_{p\nmid q} \mcF(p, 3) \ \ll \ \frac{1}{(\log Q)^3}
\end{align}
by Lemma \ref{lemmaforcase2ofcombo}. Hence, $\mcS_1 = \mcS_2 + O((\log Q)^{-3})$, where
\begin{align}
    \mcS_2 \ &\coloneqq \ \nonsense{0} \hkq  \sum_{p\nmid q} \sum_{p'\nmid q} \frac{\lambda_f^2(p)\lambda_f(p')\log^2(p)\log(p')}{p\sqrt{p'} \log^3(q)} \widehat\Phi^2\left(\frac{\log p}{\log q}\right)
    \widehat\Phi\left(\frac{\log p'}{\log q}\right).
\end{align}
Since $f$ is a newform, $\lambda_f(p)^2 = \lambda_f(p^2) + 1$, and thus
\begin{align}
    \mcS_2 \ &=\  \nonsense{0}\hkq \sum_{p\nmid q} \sum_{p'\nmid q} \frac{\lambda_f(p^2)\lambda_f(p')\log^2(p)\log(p')}{p\sqrt{p'} \log^3(q)} \widehat\Phi^2\left(\frac{\log p}{\log q}\right)
    \widehat\Phi\left(\frac{\log p'}{\log q}\right)\nonumber\\
    &\hspace{.5cm}+\nonsense{0}\hkq  \sum_{p\nmid q} \sum_{p'\nmid q} \frac{\lambda_f(p')\log^2(p)\log(p')}{p\sqrt{p'} \log^3(q)} \widehat\Phi^2\left(\frac{\log p}{\log q}\right)
    \widehat\Phi\left(\frac{\log p'}{\log q}\right).
\end{align}
We name these sums $\mcS_3$ and $\mcS_4$, respectively. 

To bound $\mcS_4,$ we evaluate the sum over $p$ using Lemma \ref{ten min int} and then apply the case $n=1$, i.e., Proposition 4.2 of \cite{BCL}, to the resulting sum. We conclude that
\begin{align}
    \mcS_4\ &\ll\ \left(\int_{-\infty}^\infty |u|\Phi^2(u)\,du\right) \nonsense{0}\hkq  \sum_{p'\nmid q} \frac{\lambda_f(p') \log(p')}{\sqrt{p'} \log q} \nonumber
    \widehat\Phi\left(\frac{\log p'}{\log q}\right)\\
    &\ll \ \left(\int_{-\infty}^\infty |u|\Phi^2(u)\,du\right)\frac{1}{\log Q} \ \ll \ \frac{1}{\log Q}.
\end{align}
We now bound $\mcS_3.$ Taking absolute values yields
\begin{align}
    \mcS_3 \ &\leq \   \sum_{q=1}^{aQ}  \hkq \left| \frac{1}{\sqrt{N(Q)}}\sqrt{\Psi\left(\frac qQ\right)} \sum_{p\nmid q} 
    \frac{\lambda_f(p^2)\log^2(p)}{p \log^2(q)}\widehat\Phi^2\left(\frac{\log p}{\log q}\right)\right|\nonumber\\
    &\hspace{.5cm}\times \left| \frac{1}{\sqrt{N(Q)}}\sqrt{\Psi\left(\frac qQ\right)}\sum_{p'\nmid q} \frac{\lambda_f(p')\log(p')}{\sqrt{p'} \log(q)} 
    \widehat\Phi\left(\frac{\log p'}{\log q}\right)\right|\label{eqn:n=3case},
\end{align}
where $\Psi$ is supported on $(0, a)$. Viewing the sum over $q$ and the sum over $f$ as a single sum, we now apply Cauchy-Schwarz to obtain
\begin{align}
    \mcS_3 \ &\leq \ \left| \sum_{q=1}^{aQ}  \hkq\frac{1}{N(Q)}\Psi\left(\frac qQ\right) \left(\sum_{p\nmid q} 
    \frac{\lambda_f(p^2)\log^2(p)}{p \log^2(q)}\widehat\Phi\left(\frac{\log p}{\log q}\right)^2\right)^2\right|^{1/2}\nonumber\\
    &\hspace{.5cm}\times \left| \sum_{q=1}^{aQ}  \hkq \frac{1}{N(Q)}\Psi\left(\frac qQ\right)\left(\sum_{p'\nmid q} \frac{\lambda_f(p')\log(p')}{\sqrt{p'} \log(q)}
    \widehat\Phi\left(\frac{\log p'}{\log q}\right)\right)^2\right|^{1/2}.
\end{align}
After removing the diagonal piece whose contribution is $O(1)$ by Lemma \ref{ten min int}, the second sum reduces to Proposition \ref{prop42} with $n=2$. Therefore,
\begin{align}
    \mcS_3 \ &\ll \ \left| \sum_{q=1}^{aQ}  \hkq \frac{1}{N(Q)}\Psi\left(\frac qQ\right) \left(\sum_{p\nmid q} 
    \frac{\lambda_f(p^2)\log^2(p)}{p \log^2(q)}\widehat\Phi\left(\frac{\log p}{\log q}\right)^2\right)^2\right|^{1/2}.\label{boundinglambda_f p^2 term}
\end{align}
We apply Lemma \ref{lemma:boundingpsq} with $F(t) = \widehat\Phi^2(t)$ to conclude that
\begin{align}
    \mcS_3 \ &\ll \ \left| \sum_{q=1}^{aQ}  \hkq \frac{1}{N(Q)}\Psi\left(\frac qQ\right) \left(\frac{\log \log q}{\log q}\right)^2\right|^{1/2}\ \ll\ \frac{\log\log Q}{\log Q}.
\end{align}
\hfill $\Box$

We now assume $n\geq 4$ and for the purpose of strong induction assume Theorem \ref{onylallpairedoralldifferent} holds for $1, 2, \ldots, n-1$. We complete the proof in Lemmas \ref{lemmaforcase3ofcombo} and \ref{lemmaforcase4ofcombo}.

\subsection{Proof of Lemma \texorpdfstring{\ref{lemmaforcase3ofcombo}}{}}
We recall that we assume our partition is in nonincreasing order. Because $n\geq 4$ and there is exactly one 1 in the partition, there is at least one part $\geq 3$, or at least two parts of~2. 

Suppose there is a part $\geq 3$. Then by Lemma \ref{lemma:mostbeautifullemmaintheworld} and Lemma \ref{bounding a=2}, we deduce
\begin{align}
    \mcS\ &= \ \nonsense{0} \hkq\mathop{\sum\cdots\sum}_{\substack{p_1,\ldots, p_\ell \nmid q\\ p_i\neq p_j}} \prod_{i=1}^\ell \mcF(p_i, a_i) \nonumber \\
    &\ll\ \nonsense{0} \hkq\sum_{p_1\nmid q} \left|\mcF(p_1, a_1)\right| \sum_{p_2\nmid qp_1} \left|\mcF(p_2,a_2)\right|\cdots \left|\sum_{p_\ell \nmid qp_1\cdots p_{\ell-1}} \mcF(p_\ell, a_\ell)\right|\nonumber\\
    &\ll\ \nonsense{0} \hkq\frac{1}{(\log q)^3} \frac{(\log q)^{2+\e}}{\log q} \ = \ O\left(\frac{1}{\log Q}\right).
\end{align}

Now, suppose there is exactly one part 1 and every other part is 2. Explicitly, 
\begin{align}
    \mcS\ &=\ \nonsense{0} \hkq
    \sum_{p_1\nmid q}\sum_{p_2\nmid p_1q} \mathop{\sum\cdots\sum}_{\substack{p_3,\ldots, p_{\ell-1} \nmid p_1p_2q\\ p_i\neq p_j}}\prod_{i=1}^{\ell-1} \mcF(p_i, 2) \sum_{p_\ell \nmid qp_1\cdots p_{\ell-1}} \mcF(p_\ell, 1).
\end{align}
We now induct on $\ell.$ The base case $\ell = 2$ is handled in the $n=3$ case (see Appendix \ref{subsec:base}).

Now, we assume the statement holds for $\ell-1\geq 2$. If $p_1 = p_\ell$, then this reduces to Case \ref{case2ofcombo}. If $p_1 = p_i$ for some $1 < i < \ell,$ then this reduces to at least one $a_i \geq 3$. Hence, we may add back when $p_1$ equals any other $p_i$. The same idea holds with $p_2$. Since $\ell > 2$, we now have 
\begin{align}
    \mcS\ &=\ o(1)+ \nonsense{0} \hkq
    \sum_{p_1\nmid q}\sum_{p_2\nmid q}
    \mathop{\sum\cdots\sum}_{\substack{p_3,\ldots, p_{\ell-1} \nmid q\\ p_i\neq p_j}} \prod_{i=1}^{\ell-1} \mcF(p_i, 2) \sum_{p_\ell \nmid qp_3\cdots p_{\ell-1}} \mcF(p_\ell, 1).   
\end{align}
We apply Lemma \ref{bounding a=2} for $\mcF(p_1,2)$ and $\mcF(p_2,2)$ and write $\mcF(p_i,2)=C_i+G_i$, where $C_i$ is the integral and $G_i \ll \log\log q/\log q$. Expanding out $(C_1+G_1)(C_2+G_2)$ splits $\mcS$ into 4 sums: $C_1C_2$; $C_1G_2$ and $C_2G_1$; and $G_1G_2$. The first case is trivial; since $C_1C_2$ is a constant independent of $f$, this reduces to Case \ref{case5ofcombo} or the inductive hypothesis. 

To bound the $C_1G_2$ term, we add back the $C_1C_2$ term so that it suffices to show the $C_1(C_2 + G_2)$ term is $o(1)$. By a similar argument as before, we may add back the condition that $p_i\neq p_2$ for $3\leq i\leq \ell$. 
Since this difference is $o(1)$, we appeal to the inductive hypothesis, completing this case. The same idea holds for $C_2G_1.$

Finally, we bound the $G_1G_2$ sum. By  Lemma \ref{lemma:mostbeautifullemmaintheworld} and Lemma \ref{bounding a=2}, we have that
\begin{align}
    \nonsense{0} \hkq
    G_1G_2 &\mathop{ \sum\cdots \sum}_{\substack{p_3,\ldots, p_{\ell-1} \nmid q\\ p_i\neq p_j}} \prod_{i=3}^{\ell-1} \mcF(p_i, 2) \sum_{p_\ell \nmid qp_3\cdots p_{\ell-1}} \mcF(p_\ell, 1) \nonumber \\
    \ll\ \nonsense{0} \hkq
    &\frac{(\log\log q)^2}{(\log q)^2} \mathop{\sum\cdots\sum}_{\substack{p_3,\ldots, p_{\ell-1} \nmid q\\ p_i\neq p_j}} \prod_{i=3}^{\ell-1} \left|\mcF(p_i, 2) \right| \frac{\log^{2+\e}(q)}{\log q} \label{abs in lemma2.9}\nonumber\\
    \ll\ \nonsense{0} \hkq
    &\frac{(\log\log q)^2}{(\log q)^2} \log^{1+\e}(q) \ \ll \ \frac{(\log q)^\e}{\log q},
\end{align}
completing the induction and the proof of Lemma \ref{lemmaforcase3ofcombo}. \hfill $\Box$
 
\subsection{Proof of Lemma \texorpdfstring{\ref{lemmaforcase4ofcombo}}{}}
Recall that this lemma bounds terms where the partition has at least two $1$'s in the partition, but it is not all $1$. We will need the following generalization of H\"older's inequality, which is essential for our analysis of these terms, and a lemma which allows us to apply this.
\begin{lemma}\label{holder}
Suppose $\lambda_i \in \R^+$ with $\lambda_a+\cdots +\lambda_z=1$. Then 
\begin{align}
    \sum_{k=1}^n |a_k|^{\lambda_a}\cdots |z_k|^{\lambda_z}\ \leq\ \left(\sum_{k=1}^n |a_k|\right)^{\lambda_a}\cdots \left(\sum_{k=1}^n |z_k|\right)^{\lambda_z}.
\end{align}
\end{lemma}
\begin{lemma}\label{gimme xis but good this time}
    Let $m > 1.$ Then there exist  $\xi_i^{-1} \in 2\zz\cap \{1,2,\ldots, m+1\}$ with 
    $\sum_{i=1}^m \xi_i = 1$.
\end{lemma}
\begin{proof}
    Take $\lfloor m/2\rfloor$-many $\xi_i=(2\lfloor m/2\rfloor)^{-1}$ and $\lceil m/2\rceil$-many $\xi_i=(2\lceil m/2\rceil)^{-1}.$ 
\end{proof}

We induct on $\ell$, the number of parts in our partition. Our base case is $\ell = 1$ and is vacuous. Let there be $k \geq 2$ many $1$'s in our partition and $k'-k$ many $2$'s.

We add back in the primes for $p_i = p_j$ when $a_i + a_j > 2$. 
If there is only one part of $1$ remaining, we are in Case \ref{case3ofcombo} and if there are two parts of $1$ remaining, we have $\ell-1$ parts and we apply the inductive hypothesis. 
Hence, it suffices to bound
\begin{align}
    \nonsense{0} \hkq\mathop{\sum\cdots\sum}_{\substack{p_1,\ldots, p_\ell \nmid q\\ i,j > \ell - k \implies p_i\neq p_j}} \prod_{i=1}^\ell \mcF(p_i, a_i).
\end{align}
Similar to the previous subsection, we apply Lemma \ref{bounding a=2} to each of the $(k'-k)$-many $a_i = 2$. As before, we first handle the all-integral terms which have a sum over at most $n-2$ primes using the inductive hypothesis. For each of the remaining sums, we denote by $\mcD(Q)$ the product of the sum of the not-all-integral terms from Lemma \ref{bounding a=2} and the sums with $a_i > 2$. Since $k<\ell$,
\begin{align}
    \mcD(Q) \ \ll \ \left(\frac{\log\log Q}{\log Q}\right)^{\mathbf{1}_{k'>k}} (\log Q)^{-\sum_{i=1}^{\ell-k'} a_i} \ \ll \ \frac{1}{(\log Q)^{1-\e}}.
\end{align} 
It thus suffices to bound
\begin{align}
    \mcS'\ &\coloneqq \ \nonsense{0} \hkq \mcD(Q) \mathop{\sum\cdots\sum}_{\substack{p_{\ell-k+1},\ldots, p_\ell \nmid q \\ p_i\neq p_j}} \prod_{i=\ell-k+1}^\ell \mcF(p_i, 1).
\end{align}
For a given $p_i$ and each $\ell - k < j \leq \ell$ and $j \neq i$, we add back the sum over $p_j = p_i$. Morally, to show the terms we are adding are negligible, it suffices to show the innermost sum when averaged is $O(1)$ since the $\mcD(Q)$ term gives us decay once we take absolute values. This is formalized in Lemma \ref{lemma:removingdistinctness}. From there, we bound the resulting sum in Lemma \ref{lemma:boundingwhenindep}.

\begin{lemma} \label{lemma:boundingwhenindep}
    For each $m\geq 2,$ we have
    \begin{align}
        \mcS''\ \coloneqq \ \frac1{N(Q)}\sum_{q=1}^{aQ}  \Psi\left(\frac qQ\right) \hkq\mcD(Q) \mathop{\sum\cdots\sum}_{p_1,\ldots, p_m \nmid q} \prod_{i=1}^m \frac{\lambda_f(p_i)\log p_i}{\sqrt {p_i}\log q}\widehat\Phi\left(\frac{\log p_i}{\log q}\right)\ =\ o(1). 
    \end{align}
\end{lemma}
\begin{proof}
    Taking the $\xi_i$ as given by Lemma \ref{gimme xis but good this time}, we have 
\begin{align}
    \mcS''\ =\ \sum_{q=1}^{aQ}\hkq\mcD(Q) \mathop{\sum\cdots\sum}_{p_1,\ldots, p_m \nmid q} \prod_{i=1}^m \frac{\lambda_f(p_i)\log p_i}{\sqrt {p_i}\log q}\widehat\Phi\left(\frac{\log p_i}{\log q}\right) \left(N(Q)^{-1}\Psi\left(\frac qQ\right)\right)^{\xi_i}.
\end{align}
Taking absolute values yields
\begin{align}
    \mcS''\ &\leq \ \sum_{q=1}^{aQ}\hkq|\mcD(Q)| \prod_{i=1}^m \left|\sum_{p_i\nmid q}\frac{\lambda_f(p_i)\log p_i}{\sqrt {p_i}\log q}\widehat\Phi\left(\frac{\log p_i}{\log q}\right)\right|\left(N(Q)^{-1} \Psi\left(\frac qQ\right)\right)^{\xi_i}\nonumber\\
    &\ll \ \frac{1}{(\log Q)^{1-\e}}\sum_{q=1}^{aQ}\hkq\prod_{i=1}^m \left|\sum_{p_i\nmid q}\frac{\lambda_f(p_i)\log p_i}{\sqrt {p_i}\log q}\widehat \Phi\left(\frac{\log p_i}{\log q}\right)\right|\left(N(Q)^{-1}\Psi\left(\frac qQ\right)\right)^{\xi_i}.
\end{align}
We now raise the prime sums to the $\xi_i$ and $\xi_i^{-1} \in 2\zz^+$ powers. Since the $\lambda_f(p_i)$ are real we obtain
\begin{align}
    \mcS''\ &\ll\ \frac{1}{(\log Q)^{1-\e}}\sum_{q=1}^{aQ}\hkq\prod_{i=1}^m \left(\left(\sum_{p_i\nmid q}\frac{\lambda_f(p_i)\log p_i}{\sqrt {p_i}\log q}\widehat\Phi\left(\frac{\log p_i}{\log q}\right)\right)^{\xi_i^{-1}} N(Q)^{-1}\Psi\left(\frac qQ\right)\right)^{\xi_i}\nonumber\\
    &\ll\ \frac{1}{(\log Q)^{1-\e}}\sum_{q=1}^{aQ}\hkq\prod_{i=1}^m \left(\left|\sum_{p_i\nmid q}\frac{\lambda_f(p_i)\log p_i}{\sqrt {p_i}\log q}\widehat\Phi\left(\frac{\log p_i}{\log q}\right)\right|^{\xi_i^{-1}} N(Q)^{-1}\Psi\left(\frac qQ\right)\right)^{\xi_i}.
\end{align}
We now apply Lemma \ref{holder} to the sum over $q$ and $f$ to conclude that
\begin{align}
    \mcS''\ &\ll \ \frac{1}{(\log Q)^{1-\e}} \prod_{i=1}^m   \left(\sum_{q=1}^{aQ}\hkq\left(\sum_{p_i\nmid q}\frac{\lambda_f(p_i)\log p_i}{\sqrt {p_i}\log q}\widehat\Phi\left(\frac{\log p_i}{\log q}\right)\right)^{\xi_i^{-1}} \hspace{-.1cm}N(Q)^{-1}\Psi\left(\frac qQ\right)\right)^{\xi_i}.
\end{align}
Now, each term in the product is a $\xi_i^{-1}\leq k+1 \leq (n-2)+1 < n$-th centered moment. By the inductive hypothesis, each of these is finite. Hence $\mcS'' = o(1).$
\end{proof}

\begin{lemma}\label{lemma:removingdistinctness}
    Let $IH_1(s)$ be the statement: for any $k \geq 0,$ removing the distinctness condition from 
    \begin{align}
        \mcS_5 \ \coloneqq \ \frac{1}{N(Q)} \sum_{q=1}^{aQ}\Psi\left(\frac qQ\right)\hkq\mcD(Q) \sum_{p_1,\ldots,p_k \nmid q} \sum_{\substack{p_{k+1},\lodts, p_{k+s} \nmid q\\ p_i\neq p_j}} \prod_{i=1}^{k+s} \mcF(p_i,1)
    \end{align}
    changes the sum by $o(1)$. Let $IH_2(s)$ be the statement: for any $k \geq 0$ and $j\geq 1$ we have
    \begin{align}
        \mcS_6 \ \coloneqq \ \frac{1}{N(Q)} \sum_{q=1}^{aQ}\Psi\left(\frac qQ\right) \hkq\mcD(Q) \sum_{p_1,\ldots,p_k \nmid q} \sum_{p' \nmid q} \mcF(p',j) \sum_{\substack{p_{k+1},\lodts, p_{k+s} \nmid p'q\\ p_i\neq p_j}} \prod_{i=1}^{k+s} \mcF(p_i,1)
    \end{align}
    is $o(1)$. Suppose $\mcD(Q)\to 0$ as $Q \rightarrow \infty$ independent of $f \in \mcH_k(q)$. Then $IH_1(s)$ and $IH_2(s)$ are true for all $s\geq 1.$
\end{lemma}
\begin{proof}
The statement $IH_1(1)$ is trivially true. Next, to prove $IH_2(1)$, we observe
    \begin{align}
        \mcS_6 \ &=\ \frac{1}{N(Q)} \sum_{q=1}^{aQ}\Psi\left(\frac qQ\right) \hkq\mcD(Q) \sum_{p_1,\ldots,p_k \nmid q} \sum_{p' \nmid q} \mcF(p',j) \sum_{p_{k+1}\nmid q} \prod_{i=1}^{k+1} \mcF(p_i,1)\nonumber\\
        &\hspace{.5cm} - \frac{1}{N(Q)} \sum_{q=1}^{aQ}\Psi\left(\frac qQ\right) \hkq\mcD(Q) \sum_{p_1,\ldots,p_k \nmid q} \sum_{p' \nmid q} \mcF(p',j+1) \prod_{i=1}^{k} \mcF(p_i,1).
    \end{align}
Since $j+1\geq2,$ by Lemma \ref{bounding a>2} or \ref{bounding a=2}, we have that $\sum_{p'\nmid q}\mcF(p',j+1) = O(1)$ independent of $f$, and thus Lemma \ref{lemma:boundingwhenindep} bounds the second sum. If $j=1,$ then the first sum is $o(1)$ by Lemma \ref{lemma:boundingwhenindep} with $k+2$. Otherwise, if $j>1,$ then the sum over $p'$ is $O(1)$, so we incorporate it into $\mcD(Q)$ and apply Lemma \ref{lemma:boundingwhenindep}. We thus have shown $IH_2(1).$

We assume $IH_1$ and $IH_2$ for all $1,2,\ldots, s-1.$ We first prove $IH_1(s).$ By the same idea as \eqref{notdoingxinonsense}, we observe that 
\begin{align}
    \mcS_5\ &=\ \frac{1}{N(Q)} \sum_{q=1}^{aQ}\Psi\left(\frac qQ\right)\hkq\mcD(Q) \sum_{p_1,\ldots,p_k \nmid q}\  \sum_{p_{k+1}\nmid q}\  \sum_{\substack{p_{k+2},\lodts, p_{k+s} \nmid q\\ p_i\neq p_j}} \prod_{i=1}^{k+s} \mcF(p_i,1)\\
    &\hspace{.5cm} - (s-1) \frac{1}{N(Q)} \sum_{q=1}^{aQ}\Psi\left(\frac qQ\right)\hkq\mcD(Q) \sum_{p_1,\ldots,p_k \nmid q} 
    \sum_{\substack{p_{k+1}\nmid q\\ p_{k+1}=p_{k+2}} }
    \sum_{\substack{p_{k+3}, \lodts, p_{k+s} \nmid p_{k+1}q\\ p_i\neq p_j}} \prod_{i=1}^{k+s} \mcF(p_i,1)\nonumber.
\end{align}
Applying $IH_1(s-1)$ to the first sum removes the distinctness condition and applying $IH_2(s-2)$ with $j=2, \ p'=p_{k+1}$ shows the second is $o(1)$ (note that $IH_2(0)$ follows from Lemma \ref{lemma:boundingwhenindep}). Hence
\begin{align}
    \mcS_5\ =\ \frac{1}{N(Q)} \sum_{q=1}^{aQ}\Psi\left(\frac qQ\right)\hkq\mcD(Q) \sum_{p_1,\ldots,p_{k+1} \nmid q}  \sum_{\substack{p_{k+2},\lodts, p_{k+s} \nmid q}} \prod_{i=1}^{k+s} \mcF(p_i,1) + o(1) \ = \ o(1),
\end{align}
where the last equality follows by Lemma \ref{lemma:boundingwhenindep}.

We now prove $IH_2(s)$, using the same idea as \eqref{notdoingxinonsense}. We add back when some $p_\ell = p'$ for $\ell \geq k + 1$. By renaming $\ell$ to $k+1,$ it suffices to bound
\begin{align}
    \frac{1}{N(Q)} \sum_{q=1}^{aQ}\Psi\left(\frac qQ\right) \hkq\mcD(Q) \sum_{p_1,\ldots,p_k \nmid q} \sum_{p' \nmid q} \mcF(p',j+1) \sum_{\substack{p_{k+2},\lodts, p_{k+s} \nmid p'q\\ p_i\neq p_j}} \ \prod_{\substack{i=1\\ i\neq k+1}}^{k+s} \mcF(p_i,1), 
\end{align}
which is $o(1)$ by $IH_2(s-1)$. It now suffices to bound
\begin{align}
    \mcS_6' \ = \ \frac{1}{N(Q)} \sum_{q=1}^{aQ}\Psi\left(\frac qQ\right) \hkq\mcD(Q) \sum_{p_1,\ldots,p_k \nmid q} \sum_{p' \nmid q} \mcF(p',j) \sum_{\substack{p_{k+1},\ldots, p_{k+s} \nmid q\\ p_i\neq p_j}} \prod_{i=1}^{k+s} \mcF(p_i,1).
\end{align}
If $j=1,$ then this is $o(1)$ by $IH_1(s)$ and Lemma \ref{lemma:boundingwhenindep}, and if $j>1,$ then $\sum_{p'\nmid q}\mcF(p',j) = O(1)$ and we incorporate it into the $\mcD(Q)$ term. The resulting sum is $o(1)$ by the same reasoning.
\end{proof}

\section{Proof of Results in Section \ref{sec:EisensteinSeries}}\label{app:8}

\subsection{Proof of Proposition \texorpdfstring{\ref{prop:8alldivide}}{} -- Contribution from the primes \texorpdfstring{$p_i\mid N$}{}}\label{appen b1}
We now prove Proposition \ref{prop:8alldivide}.
After using the parameterization given by Lemma \ref{lemma:their8.1}, we take absolute values to obtain
    \begin{align}
        \mathcal{CTN}_{\fn \mid  N} \ &\ll \ \sum_{c\geq 1}\sum_{b\mid N}\amod \mathop{\sum \cdots \sum}_{\substack{p_i \mid  N \\ p_i \neq p_j}} \prod_i \frac{\log p_i}{\sqrt{p_i}}
        \int_{-\infty}^\infty \frac{\sqrt{e^2 \prod p_i}}{\cosh(\pi t)}|\overline{\varphi}_\fc(e^2,t)||\varphi_\fc(\fn, t)||h_+(t)|\,dt \label{eqn:CTNp|Nexp}.
    \end{align}
We use \eqref{eqn:kygeneral} to bound the Fourier coefficients $\ol\varphi_\fc(e^2,t)$ and $\varphi_\fc(\fn, t)$ and obtain that
    \begin{align}
        \overline{\varphi}_\fc(e^2, t) \ &\ll \ \frac{\sqrt{\pi}e^{-1}}{|\Gamma(1/2 + it)|}\frac{\sqrt{(b, N/b)}}{\sqrt{Nb}}\frac{b'}{(b, N/b)}\left|\eightKloos{m}\right|\nonumber \\
        &\hspace{+0.5cm} \times \sum_{\substack{d \mid  m \\ (d, N/b) = 1}}\chinonsenseabs{m}{\chi}{d}.
    \end{align}
We bound the Kloosterman sum as $\leq b_0$ and use the bounds in \eqref{eqn:8manybounds} on the $L$-function and the Gauss sum. Recalling that $b = b'b_0$ and $m \leq e^2$ so that the sum over $d\mid m$ occurs $\ll \tau(m) \ll \tau(e^2) \ll e^\e$ times, we obtain that
    \begin{align}
        \overline{\varphi}_\fc(e^2, t) \ &\ll\ \frac{\sqrt{\pi} e^{-1}}{|\Gamma(1/2 - it)|}\frac{\sqrt{(b,N/b)}}{\sqrt{Nb}}\frac{b'}{(b,N/b)}b_0e^\e\sqrt{(b,N/b)} (N(1+|t|))^\e \nonumber \\
        \ &\ll\ Q^\e\frac{\sqrt{\pi}e^{-1}}{|\Gamma(1/2-it)|}\frac{b}{\sqrt{Nb}}(N(1+|t|))^\e.
    \end{align}
Using \eqref{eqn:8manybounds}, we analogously obtain
the bound above on $\varphi_\fc(\fn,t)$ with $e$ replaced with $\sqrt{\fn}$. Substituting into \eqref{eqn:CTNp|Nexp}, and using that $|\Gamma(1/2 - it)|^2 = \pi/\cosh(\pi t)$, we obtain
    \begin{align}
        \mathcal{CTN}_{\fn\mid N} \ &\ll \ \ \sum_{c\geq 1} \mathop{\sum \cdots \sum}_{\substack{p_i \mid  N \\ p_i \neq p_j}} \prod_i \frac{\log p_i}{\sqrt{p_i}}Q^\e \sum_{b \mid  N}\sum_{a \text{ mod } (b, N/b)}\mkern -33mu ^* \mkern 22mu \int_{-\infty}^\infty \frac{b}{N} (N(1+|t|))^\e|h_+(t)|\,dt. 
    \end{align}
Using that $\sum_{a\text{ mod } {(b,N/b)}}\mkern -90mu ^* \mkern 83mu 1 \leq N/b$, $\tau(N) \ll N^\e$, and dropping $p_i\neq p_j$ by positivity, we have that
    \begin{align}
        \mathcal{CTN}_{\fn \mid  N} \ &\ll \ \sum_{c\geq 1} \mathop{\sum \cdots \sum}_{p_i \mid  N} \prod_i \frac{\log p_i}{\sqrt{p_i}} (QN)^\e \int_{-\infty}^\infty (1+|t|)^\e|h_+(t)|\,dt.
    \end{align}
Since each $p_i\mid  N$, we have that
    \begin{align}
        \mathcal{CTN}_{\fn \mid  N} \ \ll \ \sum_{c \geq 1} (QN)^\e \int_{-\infty}^\infty (1+|t|)^\e|h_+(t)|\,dt.
    \end{align}
We use the bound on $h_+(t)$ from Lemma \ref{lemma:their3.3} and recall from \eqref{eqn:Xdef} that $X = \frac{4\pi L_1L_2 \sqrt{\prod_i P_i}}{cQ}$ and $N = cL_1rd$ where $L_1rd \ll Q$. We use the same argument as in \eqref{sqrt X garbage} to obtain
    \begin{align}
        \mathcal{CTN}_{\fn \mid N}\ &\ll\ Q^\e\sum_{c \geq 1}N^\e\min\left\{X^{k-1}, \frac{1}{\sqrt{X}}\right\}\int_{-\infty}^\infty (1 + |t|)^\e\frac{(1 + |\log X|)}{F^{1-\e}}\left(\frac{F}{1+|t|}\right)^2\,dt\nonumber \\
        \ &\ll \ (1 + |u|)^2 Q^\e \frac{\sqrt{\prod P_i}}{Q},
    \end{align}
where $F$ is such that $F < (1 + |u|)(1 + X)$. This completes the proof of Proposition \ref{prop:8alldivide}.

\subsection{Proof of Lemma \ref{lemma:8their8.7}}\label{appen b2}
\begin{proof}
We proceed by induction on $\kappa$, the number of $\alpha.$ Our base case is $\kappa = 1,$ and this is simply \cite[Lemma 8.7]{BCL}.
We assume the lemma for $\kappa$ and prove the lemma for $\kappa+1$. We have that 
\begin{align}
    \mathop{\sum\cdots \sum}_{\substack{p_1,\ldots, p_{\kappa+1} \nmid N\\ p_\alpha \neq p_\beta}}
    &\prod_{\alpha=1}^{\kappa+1} \frac{\log p_\alpha}{p_\alpha^{1/2\pm it}}V\left( \frac{p_\alpha}{P_\alpha} \right)\mathrm{e}\left(v_\alpha \frac{p_\alpha}{P_\alpha}\right)\nonumber\\
    =\ \sum_{p_1 \nmid N} &\frac{\log p_1}{p_1^{1/2\pm it}}V\left( \frac{p_1}{P_1} \right)\mathrm{e}\left(v_1 \frac{p_1}{P_1}\right) \mathop{\sum\cdots \sum}_{\substack{p_2,\ldots, p_{\kappa+1} \nmid Np_1\\ p_\alpha \neq p_\beta}}\prod_{\alpha=2}^{\kappa+1} \frac{\log p_\alpha}{p_\alpha^{1/2\pm it}}V\left( \frac{p_\alpha}{P_\alpha} \right)\mathrm{e}\left(v_\alpha \frac{p_\alpha}{P_\alpha}\right).
\end{align}
The sum over $p_2,\ldots,p_{\kappa + 1}$ is equal to
\begin{align}
    &=\  \prod_{\alpha=2}^{\kappa +1} \left(P_\alpha^{1/2\mp it}\widetilde V_\alpha\left(\frac{1}{2}\mp it\right)+ C'_\alpha P_\alpha^\e(1+|v_\alpha|)^2\log (Np_1)\right)\nonumber\\
    &=\  \prod_{\alpha=2}^{\kappa +1} \left(P_\alpha^{1/2\mp it}\widetilde V_\alpha\left(\frac{1}{2}\mp it\right)+ C_\alpha P_\alpha^\e(1+|v_\alpha|)^2\log N \log(P_1+3)\right),
\end{align}
where we have applied the inductive hypothesis with $Np_1$ to the sum from $\alpha = 2$ to $\kappa + 1$, and 
\begin{align}
    C_\alpha\ =\ C_{\alpha, p_1, P_2,\ldots, P_{\kappa+1}, t, v_2,\ldots, v_{\kappa+1}, \widetilde V_2,\ldots, \widetilde V_{\kappa+1}}
\end{align}
is a sequence of constants, where $||C'_\alpha||_\infty\leq  M$ is uniformly bounded independent of all variables that $C_\alpha$ is dependent on. Further, we have $||C_\alpha||_\infty \leq ||C'_\alpha||_\infty.$
We then expand this product, and split into cases based on whether the resulting term in the sum of products contains a factor of $C_\alpha$. We first bound
\begin{align}
    \mcE_{\kappa+1}\ &= \ \sum_{p_1 \nmid N} \frac{\log p_1}{p_1^{1/2\pm it}}V\left( \frac{p_1}{P_1} \right)\mathrm{e}\left(v_1 \frac{p_1}{P_1}\right) \prod_{\alpha=2}^{\kappa +1} \left(P_\alpha^{1/2\mp it}\widetilde V_\alpha\left(\frac{1}{2}\mp it\right)\right)\nonumber \\
    &= \ \prod_{\alpha=2}^{\kappa +1} \left(P_\alpha^{1/2\mp it}\widetilde V_\alpha\left(\frac{1}{2}\mp it\right)\right)\sum_{p_1 \nmid N} \frac{\log p_1}{p_1^{1/2\pm it}}V\left( \frac{p_1}{P_1} \right)\mathrm{e}\left(v_1 \frac{p_1}{P_1}\right),
\end{align}
and by \cite[Lemma 8.7]{BCL} this is
\begin{align}
    \mcE_{\kappa+1}\ = \ \left(P_1^{1/2\mp it}\widetilde V_1\left(\frac{1}{2}\mp it\right)+O(P_1^\e(1+|v_1|)^2\log N)\right)\prod_{\alpha=2}^{\kappa +1} \left(P_\alpha^{1/2\mp it}\widetilde V_\alpha\left(\frac{1}{2}\mp it\right)\right).\label{first big O}
\end{align}
For each $2\leq \alpha\leq \kappa+1$ define $\Xi_\alpha\colon\rr\to\rr$ to be a smooth, uniformly bounded by $M$ function with
\begin{align}\nonumber
    \Xi_\alpha(p_1)\ =\ C_{\alpha, p_1, P_2,\ldots, P_{\kappa+1}, t, v_2,\ldots, v_{\kappa+1}, \tilde V_2,\ldots, \tilde V_{\kappa+1}}
\end{align}
and whose derivatives up to $A$ are all bounded by $M'M$ for some constant $M' > 1$ independent of all the variables which $C_\alpha$ depends on. Without loss of generality, we assume that we are bounding
\begin{align}
    \mcE_\beta\ &=\ \sum_{p_1 \nmid N} \frac{\log p_1}{p_1^{1/2\pm it}}V\left( \frac{p_1}{P_1} \right)\mathrm{e}\left(v_1 \frac{p_1}{P_1}\right) \prod_{\alpha=2}^{\beta} P_\alpha^{1/2\mp it} \nonumber \\
    &\hspace{.5cm} \times 
    \widetilde V_\alpha\left(\frac{1}{2}\mp it\right)\prod_{\alpha=\beta+1}^{\kappa+1} \Xi_\alpha(p_1) P_\alpha^\e(1+|v_\alpha|)^2\log N \log(P_1+3)
\end{align}
for some $1\leq \beta \leq \kappa.$
Define 
\begin{align}
    V_1(x)\ =\ V\left( \frac{x}{P_1} \right)\mathrm{e}\left(v_1 \frac{x}{P_1}\right) \prod_{\alpha=\beta+1}^{\kappa+1} \Xi_\alpha(x) \log(P_1+3).
\end{align}
We observe that 
\begin{align}
    V_1(x)^{(l)}\ \ll\  (1+|v_1|)^{l}\log(P_1+3)^{\kappa+1},
\end{align}
with $\ll$ based only on $V$ and $l,$ and thus conclude that
\begin{align}
    \widetilde V_1(s)\ \ll\ P_1^\e \left(\frac{1+|v_1|}{1+|s|} \right)^A\ \ll\ Q^\e \left(\frac{1+|v_1|}{1+|s|} \right)^A.
\end{align}
Further, by the proof of \cite[Lemma 8.7]{BCL}, using our definition of $V_1(x)$, we conclude that 
\begin{align}
    \mcE_\beta\ &= \  \prod_{\alpha=2}^{\beta} P_\alpha^{1/2\mp it}\widetilde V_\alpha\left(\frac{1}{2}\mp it\right)\prod_{\alpha=\beta+1}^{\kappa+1}  P_\alpha^\e(1+|v_\alpha|)^2\log (N) (MM')^{\kappa+1}\nonumber \\
    &\hspace{.5cm} \times \left(P_1^{1/2\mp it}\widetilde V_1\left(\frac{1}{2}\mp it\right)+O(P_1^\e(1+|v_1|)^2\log N)\right).\label{second big O}
\end{align}
Summing over all $\beta$ (including $\mcE_{\kappa+1}$), we conclude that 
\begin{align}
    \Bigg|&\mathop{\sum\cdots \sum}_{\substack{p_1,\ldots, p_\kappa \nmid N\\ p_\alpha \neq p_\beta}}
    \prod_{\alpha=1}^\kappa \frac{\log p_\alpha}{p_\alpha^{1/2\pm it}}V\left( \frac{p_\alpha}{P_\alpha} \right)e\left(v_\alpha \frac{p_\alpha}{P_\alpha}\right)-\prod_{\alpha=1}^\kappa P_\alpha^{1/2\mp it}\widetilde V_\alpha\left(\frac{1}{2}\mp it\right)\Bigg|\nonumber  \\
    &\leq\ M''(MM')^{\kappa+1}\sum_{\substack{e_\alpha \in \{0,1\}\\ 1\leq \alpha\leq \kappa+1\\ \exists \alpha\colon e_\alpha = 0}} \prod_{\alpha=1}^{\kappa+1} \left(P_\alpha^{1/2\mp it}\widetilde V_\alpha\left(\frac{1}{2}\mp it\right)\right)^{e_\alpha} \Bigg(P_\alpha^\e(1+|v_\alpha|)^2\log (N)\Bigg)^{1-e_\alpha},
\end{align}
where $M''$ is the supremum of all of the big $O$ given in \eqref{first big O} and \eqref{second big O}.
Hence, the constants in the big $O$ are all bounded by $M''(MM')^{\kappa+1},$ and this completes the induction.
\end{proof}

\subsection{Proof of Lemma \ref{lemma:8nonprinc} -- The non-principal characters}\label{appen b3}
We now prove Lemma \ref{lemma:8nonprinc}.
The goal of this subsection is to bound the following sum:
\begin{align}
    E_{non-princ}\ &= \ \sum_{c\geq 1}\bsum\amod\mathop{\sum_{\fp}\sum_{\fo}\sum_{\fq}}_{p_i\neq p_j}\eightlog{i} \eightKloos{m} \label{enonprincdefn}
    \nonumber \\
    &\hspace{.5cm} \times \int_{-\infty}^\infty e^{-2it}\left( \prod_\fp \fp \prod_\fo \fo \prod_\fq \fq \right)^{it} \left( \prod_\fp \fp \right)\eightKloos{\prod_\fo \fo \prod_\fq \fq} \frac{1}{Nb_0} \nonumber \\
    &\hspace{.5cm} \times \conjdnonsense{m}{\chi}  \\
    &\hspace{.5cm} \times \chinonsensenonprinc{\prod_\fo \fo \prod_\fq \fq }{\psi}{1} h_+(t)dt. \nonumber 
\end{align}
Observe that each prime in $\fq$ is coprime to $N$ and thus is also coprime to $b_0$. Hence,
\begin{align}
    \left(\prod_\fo \fo \prod_\fq \fq, (b,N/b)\right)\ = \ \left(\prod_\fo \fo, (b,N/b)\right).
\end{align}
This allows us to bring the sum over $\fq$ inside the integral and past the sum over characters. Hence,
\begin{align}
    E_{non-princ}\ &= \ \sum_{c\geq 1}\bsum\amod\mathop{\sum_{\fp}\sum_{\fo}}_{p_i\neq p_j}\eightlog{i} \frac{1}{Nb_0}\int_{-\infty}^\infty e^{-2it} \left( \prod_\fp \fp \prod_\fo \fo \right)^{it}
    \nonumber\\
    &\hspace{-.5cm} \times  \left( \prod_\fp \fp \right) \sum_{\substack{d\mid m\\ (d,N/b)=1}}d^{2it}  
     \conjchinonsense{m}{\chi}{d} \nonumber \\
    &\hspace{-.5cm} \times \eightKloos{\prod_\fo \fo }
    \frac{1}{\varphi\left(\frac{(b,N/b)}{(\prod_\fo \fo ,(b,N/b))}\right)}
    \sum_{\substack{\psi \bmod \frac{(b,N/b)}{( \prod_\fo \fo ,(b,N/b))}\\ \\ \psi \neq \psi_0}}\frac{ \psi (-\frac{\prod_\fo \fo}{( \prod_\fo \fo ,(b,N/b))} \ol{b_0} a)\tau(\ol {\psi})}{L(1+2it, \ol{{ \psi }^2}{\chi}_0)} \nonumber\\
    &\hspace{-.5cm} \times  h_+(t) \eightKloos{m} \sum_{\substack{\fq\\ p_i\neq p_j}} \prod_i \frac{\psi(p_i) \log p_i}{p_i^{1/2\pm it}} V\left(\frac{p_i}{P_i}\right)e\left(v_i\frac{p_i}{P_i}\right)
    dt.
\end{align}
\begin{remark}
Note that Lemma \ref{lemma:mostbeautifullemmaintheworld} works with the $\pm it$ independent of the $\fq$ sum.
\end{remark}
Since the primes in $\mathfrak{q}$ do not divide $N$ and the primes in $\mathfrak{p}, \mathfrak{o}$ divide $N$, the primes in $\mathfrak{q}$ are necessarily distinct from the primes in $\mathfrak{p}, \mathfrak{o}$ and consequently these two sets of sums are independent from one another. Thus, we can apply Lemma \ref{lemma:mostbeautifullemmaintheworld} to the sum over primes $\mathfrak{q}$. By Remark \ref{rem:evalsums}, the function $\Psi_i(x) = V\left(cx/X\right)\mathrm{e}\left(v_ix/X\right)$ where $V$ is supported on $[0,c]$ and $X_i = cP_i$ satisfies the conditions of Lemma \ref{lemma:mostbeautifullemmaintheworld} with $A_i \ll (1 + |v_i|)^3$. Taking absolute values, and after dropping the condition on $\sum_{b|N}$ by positivity, we conclude that 
\begin{align}
    E_{non-princ}\ &\ll \ \sum_{c\geq 1}\sum_{b\mid N}\amod\mathop{\sum_{\fp}\sum_{\fo}}_{p_i\neq p_j} \prod_ i\frac{\log p_i}{\sqrt{p_i}} \frac{1}{Nb_0} \int_{-\infty}^\infty \left( \prod_\fp \fp \right)
    \left|\eightKloos{\prod_\fo \fo } \right|
    \nonumber\\
    &\hspace{.5cm} \times \left|\eightKloos{m} \right|\nonumber \mkern-1mu\conjdnonsenseabs{m}{\chi}\nonumber \\
    &\hspace{.5cm} \times
    \frac{1}{\varphi\left(\frac{(b,N/b)}{(\prod_\fo \fo ,(b,N/b))}\right)}
    \sum_{\substack{\psi \bmod \frac{(b,N/b)}{( \prod_\fo \fo ,(b,N/b))}\\ \\ \psi \neq \psi_0}}\frac{ \left|\tau(\ol {\psi})\right|}{\left|L( 1+2it, \ol{{ \psi }^2}{\chi}_0)\right|} \left|h_+(t)\right|\\
    &\hspace{.5cm} \times \prod_{i=\ell+1}^n\Bigg[(\log(P_i + 2))^{1+\e}\log\left(\frac{(b,N/b)}{( \prod_\fo \fo ,(b,N/b))} + |t|\right) (1+\left|v_i\right|)^3 + \log\left(NP_i\right)\Bigg] \nonumber dt.
\end{align}
The second Kloosterman sum is $\ll m\ll e^2,$ and the sum over $d$ occurs at most $\tau(m)\ll m^\e\ll e^\e$ many times. Further simplify using \eqref{eqn:8manybounds}, we conclude that
\begin{align}
    E_{non-princ}\ &\ll \ \sum_{c\geq 1}\sum_{b\mid N}\amod\mathop{\sum_{\fp}\sum_{\fo}}_{p_i\neq p_j}\frac{\log p_i}{\sqrt{p_i}} \frac{e^{2+\e} (NQ)^\e}{Nb_0}\int_{-\infty}^\infty  \prod_\fp \fp \nonumber \\
    &\hspace{.5cm} \times \left|\eightKloos{\prod_\fo \fo } \right|(b,N/b) (1+|t|)^\e \left|h_+(t)\right| \prod_{i=\ell+1}^n(1+|v_i|)^3 dt.
\end{align}
Each prime in the sum over $\fp$ is $\ll e \ll Q^\e,$ and the $\fo$ sums occur $\ll \tau(N) \ll N^\e$ many times. Similarly, the sum of $a$ occurs at most $(b,N/b)$ many times, and thus we obtain
\begin{align}
    E_{non-princ}\ &\ll \ \sum_{c\geq 1} \sum_{b\mid N} \frac{ (NQ)^\e}{N}\int_{-\infty}^\infty (b,N/b)^2 (1+|t|)^\e \prod_{i=\ell+1}^n (1+|v_i|)^3 \left|h_+(t)\right|\,dt \nonumber \\
    &\ll \ Q^\e \prod_{i=1}^n (1+|v_i|)^3 \sum_{c\geq 1} N^\e \int_{-\infty}^\infty  (1+|t|)^\e  \left|h_+(t)\right|\,dt,
\end{align}
where we use $(b,N/b)^2 \leq b\frac{N}b = N$. From \cite[p.~30]{BCL}, we conclude that 
\begin{align}
    E_{non-princ}\ &\ll \ Q^\e (1+|u|)^2 \frac{\sqrt{\prod P_i}}{Q} \prod_{i=1}^n (1+|v_i|)^3 .
\end{align}

\printbibliography

@article{ILS,
    author = {H. Iwaniec and W. Luo and P. Sarnak},
    journal = {Publications Mathématiques de l'IHÉS},
    pages = {55-131},
    publisher = {Institut des Hautes Études Scientifiques},
    title = {Low lying zeros of families of $L$-functions},
    volume = {91},
    year = {2000}
}

@misc{BCL,
      title={Low-lying zeros of a large orthogonal family of automorphic $L$-functions}, 
      author={S. Baluyot and V. Chandee and X. Li},
      year={2023},
      eprint={2310.07606},
      archivePrefix={arXiv},
}

@article{Blomer,
   title={The Second Moment of Twisted Modular $L$-Functions},
   volume={25},
   ISSN={1420-8970},
   DOI={10.1007/s00039-015-0318-7},
   number={2},
   journal={Geometric and Functional Analysis},
   publisher={Springer Science and Business Media LLC},
   author={V. Blomer and D. Milićević},
   year={2015},
   month=mar, pages={453–516} }

@article{jackmiller,
    author = {A. Mauro and J. B. Miller and S. J. Miller},
    title = {Extending the support of $1$- and $2$-level densities for cusp form $L$-functions under square-root cancellation hypotheses},
    journal = {Acta Arithmetica 214},
    year = {2024},
    DOI={10.4064/aa230529-26-9},
    pages={289-309}
}

@article{Kiral,
   title={Kloosterman sums and Fourier coefficients of Eisenstein series},
   volume={49},
   ISSN={1572-9303},
   DOI={10.1007/s11139-018-0031-x},
   number={2},
   journal={The Ramanujan Journal},
   publisher={Springer Science and Business Media LLC},
   author={E. M. Kıral and M. P. Young},
   year={2018},
   month=aug, pages={391–409} }

@book{mont,
    author = {H. L. Montgomery and R. C. Vaughan},
    title = {Mulitplicative Number Theory: I. Classical Theory},
    publisher ={Cambridge University Press} ,
    year = {2006}
}

@mastersthesis{ng,
    author = {M. H. Ng},
    title = {The Basis for Space of Cusp Forms and Petersson Trace Formula},
    school = {The University of Hong Kong},
    year = {2012}
}

@article{Hughes_2007,
   title={Low-lying zeros of L-functions with orthogonal symmetry},
   volume={136},
   ISSN={0012-7094},
   DOI={10.1215/s0012-7094-07-13614-7},
   number={1},
   journal={Duke Mathematical Journal},
   publisher={Duke University Press},
   author={C. P. Hughes and S. J. Miller},
   year={2007},
   month=jan 
}

@article{barrett,
      title={One-level density for holomorphic cusp forms of arbitrary level}, 
      author={O. Barrett and P. Burkhardt and J. DeWitt and R. Dorward and S. J. Miller},
      year={2017},
      journal = {Research in Number Theory},
      volume = {3},
      number = {25},
      DOI = {10.1007/s40993-017-0091-9}
}

@article{BHM,
    title = {A Burgess-like subconvex bound for twisted L-functions},
    author = {V. Blomer and G. Harcos and P. Michel and Z. Mao},
    pages = {61--105},
    volume = {19},
    number = {1},
    journal = {Forum Mathematicum},
    doi = {10.1515/FORUM.2007.003},
    year = {2007}
}

@article{theta,
    author = {H. Kim and P. Sarnak},
    title = {Functoriality for the exterior square of $\GL_4$ and the symmetric fourth of $\GL_2$},
    journal = {J. Amer. Math. Soc.},
    year = {2003},
    pages = {139-183},
    volume = {16},
    number = {1},
    DOI={10.1090/S0894-0347-02-00410-1}
}

@article{AL,
    author = {A.O.L Atkin and J. Lehner},
    journal = {Mathematische Annalen},
    pages = {134-160},
    title = {Hecke Operators on $\Gamma_0(m)$},
    volume = {185},
    year = {1970}
}

@misc{chandee2025nthcenteredmomentslarge,
      title={The $n^{th}$ centered moments of a large orthogonal family of automorphic $L$-functions}, 
      author={Vorrapan Chandee and Yoonbok Lee and Xiannan Li},
      year={2025},
      eprint={2510.07647},
      archivePrefix={arXiv},
      primaryClass={math.NT},
      url={https://arxiv.org/abs/2510.07647}, 
}

@book{IK,
  title={Analytic Number Theory},
  author={H. Iwaniec and E. Kowalski},
  isbn={9780821836330},
  lccn={2004045081},
  series={American Mathematical Society colloquium publications},
  year={2004},
  publisher={American Mathematical Society}
}

@article{Rouymi,
    title = {Mollification et non annulation de fonctions L automorphes en niveau primaire},
    journal = {Journal of Number Theory},
    volume = {132},
    number = {1},
    pages = {79-93},
    year = {2012},
    issn = {0022-314X},
    doi = {10.1016/j.jnt.2011.06.006},
    author = {D. Rouymi}
}

@article{S1,
    author = {G. Shimura},
    title = {On the holomorphy of certain Dirichlet series},
    journal = {Proc. London Math. Soc.},
    year = {1975},
    month = {7},
    pages = {79-98},
    DOI = {10.1112/plms/s3-31.1.79}
}

@misc{Tao,
    title = {254A, Notes 2: Complex-analytic multiplicative number theory},
    author = {T. Tao},
    year = {2014},
    url = {https://terrytao.wordpress.com/2014/12/09/254a-notes-2-complex-analytic-multiplicative-number-theory/#fcrude}, 
    note = {Accessed on 27 July 2024}
}

@book{katz1,
    author = {N. M. Katz and P. Sarnak},
    title = {Random matrices, Frobenius eigenvalues, and monodromy},
    series = {Colloquium publications},
    publisher = {American Mathematical Society},
    year = {1999},
    volume = {45},
    ISSN = {0065-9258}
}

@article{katz2,
    author = {N. M. Katz and P. Sarnak},
    title = {Zeroes of zeta functions and symmetry},
    journal = {Bull. Amer. Math. Soc.},
    year = {1999},
    month = jan,
    volume = {36},
    number = {1},
    pages = {1-26}
}

@article{Martin,
    author = {G. Martin},
    title = {Dimensions of the spaces of cusp forms and newforms on $\Gamma_0(N)$ and $\Gamma_1(N)$},
    journal = {Journal of Number Theory},
    year = {2005},
    DOI ={10.1016/j.jnt.2004.10.009}
}

@misc{10yearsmall,
      title={Extending support for the centered moments of the low lying zeroes of cuspidal newforms}, 
      author={P. Cohen and J. Dell and O. E. González and G. Iyer and S. Khunger and C.-H. Kwan and S. J. Miller and A. Shashkov and A. S. Reina and C. Sprunger and N. Triantafillou and N. Truong and R. V. Peski and S. Willis and Y. Yang},
      year={2022},
      eprint={2208.02625},
      archivePrefix={arXiv},
      primaryClass={math.NT},
      note={To appear in Algebra \& Number Theory}
}

@article{hughesrudnick,
    author = {C. P. Hughes and Z. Rudnick},
    title = {Linear Statistics of Low‐Lying Zeros of $L$‐Functions},
    journal = {The Quarterly Journal of Mathematics},
    volume = {54},
    number = {3},
    pages = {309-333},
    year = {2003},
    month = {09},
    issn = {0033-5606},
    doi = {10.1093/qmath/hag021}
}

@article{MillerEC,
   title={One- and two-level densities for rational families of elliptic curves: evidence for the underlying group symmetries},
   volume={140},
   ISSN={1570-5846},
   DOI={10.1112/s0010437x04000582},
   number={04},
   journal={Compositio Mathematica},
   publisher={Wiley},
   author={S. J. Miller},
   year={2004},
   month=jul, pages={952–992} }

@book{Ono,
  title={The Web of Modularity: Arithmetic of the Coefficients of Modular Forms and $q$-series: Arithmetic of the Coefficients of Modular Forms and Q-series},
  author={Ono, K.},
  isbn={9780821833681},
  lccn={2003063710},
  series={Regional conference series in mathematics},
  year={2004},
  publisher={American Mathematical Society}
}

@article{rubinstein,
    author  = {Rubinstein, Michael},
    title   = {Low-Lying Zeros of {$L$}-Functions and Random Matrix Theory},
    journal = {Duke Mathematical Journal},
    volume  = {109},
    number  = {1},
    pages   = {147--181},
    year    = {2001},
    doi     = {10.1215/S0012-7094-01-10916-2}
}
\end{document}